\documentclass[11pt]{amsart}
\usepackage[T1]{fontenc}

\usepackage[shortlabels]{enumitem}
\usepackage{multirow}

\usepackage[dvipsnames]{xcolor}
\usepackage{chngcntr}
\counterwithin{equation}{section}
\counterwithin{figure}{section}
\usepackage[export]{adjustbox}
\usepackage{comment}
\usepackage{import}
\usepackage{xifthen}
\usepackage{appendix}

\appendixtitleon

\usepackage[margin=2.5cm]{geometry}

\usepackage{amsfonts}
\usepackage{amsmath}
\usepackage{amsthm}
\usepackage{mathtools}
\usepackage{amssymb}
\usepackage{microtype}
\usepackage{amsbsy}
\usepackage{nicefrac}
\usepackage{tensor}
\usepackage{bm}
\usepackage{bbm}
\usepackage{dsfont}
\usepackage{cancel}
\usepackage{empheq}
\usepackage{braket}

\DeclareMathOperator{\spn}{span}

\DeclareMathOperator{\ad}{ad}

\DeclareMathOperator{\Vir}{Vir}

\DeclareMathOperator{\Ind}{Ind}

\DeclareSymbolFont{bbold}{U}{bbold}{m}{n}
\DeclareSymbolFontAlphabet{\mathbbold}{bbold}

\makeatletter
\renewcommand*\env@matrix[1][\arraystretch]{%
  \edef\arraystretch{#1}%
  \hskip -\arraycolsep
  \let\@ifnextchar\new@ifnextchar
  \array{*\c@MaxMatrixCols c}}
\makeatother

\makeatletter
\newcommand{\medoplus}{\mathbin{\mathpalette\make@med\oplus}}
\newcommand{\medotimes}{\mathbin{\mathpalette\make@med\otimes}}
\newcommand{\make@med}[2]{%
  \vcenter{\hbox{%
    \scalebox{1.5}{$\m@th#1#2$}%
  }}%
}
\makeatother

\newcommand{\dotr}{\mbox{$\boldsymbol{\cdot}$}} 
\newcommand{\semiinfforms}{\ensuremath{\Lambda^{\dotr}_\infty}}
\newcommand{\g}{\mathfrak{g}}
\newcommand{\h}{\mathfrak{h}}

\newcommand{\f}{\mathfrak{f}}
\newcommand{\p}{\mathfrak{p}}
\newcommand{\bcca}{\mathfrak{b}}
\newcommand{\bms}{\mathfrak{bms}}
\newcommand{\fa}{\mathfrak{a}}

\newcommand{\Ptilde}{\widetilde{P}}

\newcommand{\genvecbms}{\ket{\mathbf{M}, \mathbf{s}}}
\newcommand{\genvecbmsVerma}{\ket{h_L, h_M, c_L, c_M}}

\newcommand{\ZZ}{\mathbb{Z}}
\newcommand{\CC}{\mathbb{C}}
\newcommand{\NN}{\mathbb{N}}

\newcommand{\RR}{\mathbb{R}}

\newcommand{\kk}{\Bbbk}
\newcommand{\WW}{\mathbb{W}}

\newcommand{\W}{\mathcal{W}}
\newcommand{\OO}{\mathcal{O}}
\newcommand{\UU}{\mathcal{U}}
\newcommand{\VV}{\mathcal{V}}

\newcommand{\F}{\mathcal{F}}

\newcommand{\PP}{\mathcal{P}}
\newcommand{\PPtilde}{\widetilde{\mathcal{P}}}
\newcommand{\XX}{\mathcal{X}}
\newcommand{\YY}{\mathcal{Y}}

\newcommand{\Olambda}{\mathcal{O}^{(\lambda)}}
\newcommand{\ulambda}{u^{(\lambda)}}
\newcommand{\flambda}{f^{(\lambda)}}
\newcommand{\Pb}{P^{(b)}}
\newcommand{\Ptildeb}{\widetilde{P}^{(b)}}

\DeclareMathOperator{\Der}{Der}

\DeclareMathOperator{\gr}{gr}

\newcommand{\del}{\partial}
\newcommand{\diff}[1]{\frac{d}{d#1}}

\makeatletter
    \newcommand{\extp}{\@ifnextchar^\@extp{\@extp^{\,}}}
    \def\@extp^#1{\mathop{\bigwedge\nolimits^{\!#1}}}
\makeatother

\newcommand{\nonzero}{\setminus \{0\}}

\newcommand{\floor}[1]{\left\lfloor #1 \right\rfloor}

\newcommand\restr[2]{{
  \left.\kern-\nulldelimiterspace 
  #1 
  \vphantom{\big|}
  \right|_{#2}
  }}

\newtheorem{proposition}{Proposition}[section]
\newtheorem{lemma}[proposition]{Lemma}
\newtheorem{theorem}[proposition]{Theorem}
\newtheorem{corollary}[proposition]{Corollary}
\newtheorem{conjecture}[proposition]{Conjecture}

\theoremstyle{definition}
\newtheorem{defn}[proposition]{Definition}

\newtheorem{remark}[proposition]{Remark}
\newtheorem{example}[proposition]{Example}
\newtheorem{notation}[proposition]{Notation}

\makeatletter\@addtoreset{case}{theorem}\@addtoreset{case}{proposition}\@addtoreset{case}{lemma}\@addtoreset{case}{corollary}\makeatother

\newcommand{\etalchar}[1]{$^{#1}$}

\title{On the boundary Carrollian conformal algebra}

\author[Buzaglo]{Lucas Buzaglo}
\author[He]{Xiao He}
\author[Pham]{Tuan Anh Pham}
\author[Tan]{Haijun Tan}
\author[Vishwa]{Girish S Vishwa}
\author[Zhao]{Kaiming Zhao}

\address[Buzaglo]{Department of Mathematics, UC San Diego, La Jolla, CA 92093-0112, USA}
\email{\href{mailto:lbuzaglo@ucsd.edu}{lbuzaglo@ucsd.edu}, ORCID: \href{https://orcid.org/0000-0002-7662-1802}{0000-0002-7662-1802}}

\address[He]{Paris Curie Engineer School, Beijing University of Chemical Technology, Beijing, 100029,  P. R. China}
\email{\href{mailto:hexiao@amss.ac.cn}{hexiao@amss.ac.cn}, ORCID: \href{https://orcid.org/0009-0004-4927-7593}{0009-0004-4927-7593}}

\address[Pham]{Maxwell Institute and School of Mathematics, The University of Edinburgh, James Clerk Maxwell Building, Peter Guthrie Tait Road, Edinburgh EH9 3FD, Scotland, United Kingdom}
\email{\href{mailto:tuan.pham@ed.ac.uk}{tuan.pham@ed.ac.uk}, ORCID: \href{https://orcid.org/0000-0003-0269-220X}{0000-0003-0269-220X}}

\address[Tan]{School of Mathematics and Statistics, Northeast Normal University, Changchun, Jilin, 130024, P. R. China}
\email{\href{mailto:tanhj9999@163.com}{tanhj9999@163.com}, ORCID:
\href{https://orcid.org/0000-0002-0609-2044}{0000-0002-0609-2044}}

\address[Vishwa]{Maxwell Institute and School of Mathematics, The University of Edinburgh, James Clerk Maxwell Building, Peter Guthrie Tait Road, Edinburgh EH9 3FD, Scotland, United Kingdom}
\email{\href{mailto:G.S.Vishwa@sms.ed.ac.uk}{G.S.Vishwa@sms.ed.ac.uk}, ORCID: \href{https://orcid.org/0000-0001-5867-7207}{0000-0001-5867-7207}}

\address[Zhao]{Department of Mathematics, Wilfrid Laurier University, Waterloo, Ontario, N2L 3C5, Canada}
\email{\href{mailto:kzhao@wlu.ca}{kzhao@wlu.ca}, ORCID: \href{https://orcid.org/0000-0003-4526-853X}{0000-0003-4526-853X}}

\date{}

\usepackage{graphicx}
\usepackage[font=small, labelfont=bf, skip=5pt]{caption}
\usepackage{subcaption}
\usepackage{wrapfig}
\usepackage{float}
\usepackage{tikz}
\usetikzlibrary{decorations.markings}
\usetikzlibrary{cd}

\usepackage{url}
\usepackage{hyperref}
\hypersetup{
    colorlinks=true,
    linkcolor=Blue,
    filecolor=magenta,      
    urlcolor=RoyalBlue,
    citecolor=Maroon
}
\makeatletter
\g@addto@macro{\UrlBreaks}{\UrlOrds}   
\makeatother
\usepackage{cleveref}

\usepackage{array}
\usepackage{tabularx}
\usepackage{multirow}
\renewcommand{\arraystretch}{1.3}
\newcommand{\PreserveBackslash}[1]{\let\temp=\\#1\let\\=\temp}
\newcolumntype{C}[1]{>{\PreserveBackslash\centering}p{#1}}
\newcolumntype{R}[1]{>{\PreserveBackslash\raggedleft}p{#1}}
\newcolumntype{L}[1]{>{\PreserveBackslash\raggedright}p{#1}}

\keywords{Boundary Carrollian conformal field theory, BCCA, tensionless open string, Witt algebra, Virasoro algebra, BMS algebra, Whittaker module}
\subjclass[2020]{17B68, 17B10 (Primary), 17B81, 17B65 (Secondary)}

\begin{document}

\begin{abstract}
    We initiate the mathematical study of the boundary Carrollian conformal algebra (BCCA), an infinite-dimensional Lie algebra recently discovered in the context of Carrollian physics.
    The BCCA is an intriguing object from both physical and mathematical perspectives, since it is a filtered but not graded Lie algebra.
    In this paper, we first construct some modules for the BCCA and one of its subalgebras, which we call $\OO$, by restriction of well-known modules of the BMS$_3$ and Witt algebras respectively. Along the way, we prove the irreducibility criteria for the so-called ``induced modules'' of the BMS$_3$ algebra (which we prefer to call massive modules to avoid ambiguity) and show that this is the same criteria for the irreducibility of the Verma modules of the BMS$_3$ algebra. Interestingly, the modules generated by the action of the BCCA on the generating vector of the massive modules are also irreducible under the same criteria. When these criteria hold, every massive module decomposes into a direct sum of two BCCA-submodules, each of which we conjecture to be indecomposable.
    Meanwhile, restricting Verma modules to the BCCA and $\OO$ leads to free or ``almost free'' modules, which are not particularly interesting from a representation-theoretic viewpoint. This motivates the construction of BCCA modules intrinsically. To do this, we go through some structure theory on the BCCA to define a new basis and a decreasing filtration on the algebra, using which we construct Whittaker modules over the BCCA and the subalgebra $\OO$ and prove criteria for their irreducibility.
\end{abstract}

\maketitle
\tableofcontents

\section{Introduction}

Over the past two decades, research on infinite-dimensional Lie algebras has enjoyed continuous motivation and inspiration from theoretical physics. Indeed, with the discovery of more examples of infinite-dimensional Lie algebras in physical contexts comes the need to better understand these algebras and their representations. On many occasions, this leads to a productive cycle of progress in both fields, due to the resulting mathematics research informing physics in some manner, which once again motivates another mathematical study. The origin of this paper is no different -- we present a first study on another infinite-dimensional Lie algebra that was recently discovered in physics, known as the boundary Carrollian conformal algebra (BCCA), which we denote by $\widehat{\bcca}$. Its centreless version is denoted by $\bcca$.

The BCCA was discovered as the symmetry algebra of a Carrollian conformal field theory (CCFT) with boundary in \cite{BCCFGP2024}, in which it is formulated as a subalgebra of the BMS$_3$ algebra (realised as operators in a quantum field theory) that preserves a choice of boundary on a CCFT on a cylinder. Its original basis consists of the elements $\OO_n \coloneqq L_n - L_{-n}$ ($n \geq 1)$, $P_n \coloneqq M_n + M_{-n}$ ($n \geq 0$), and $C_M$, with its Lie bracket given by
\begin{equation} \label{eq:intro BCCA Lie bracket}
    \begin{gathered}
        [\OO_n,\OO_m] = (n - m)\OO_{n + m} - (n + m)\OO_{n - m}, \\
        [\OO_n, P_m] = (n - m)P_{n + m} + (n + m)P_{n - m} + \frac{1}{6}n(n^2 - 1) \delta_{m,n} C_M, \\
        [P_n,P_m] = 0,
    \end{gathered}
\end{equation}
where $L_n$, $M_n$, and $C_M$ are elements of the BMS$_3$ algebra defined in Definition \ref{def:BMS}. 

In fact, the authors of \cite{BCCFGP2024} showed that there exist at least three different boundary-preserving subalgebras of BMS$_3$, but they choose to work with the choice given by \eqref{eq:intro BCCA Lie bracket} since it emerges as the algebra of symmetries on the worldsheet of the tensionless open strings, as they demonstrate in their paper. This serves as the primary motivation for our work: understanding the representations of the symmetry algebra of a string theory worldsheet can provide insights into the spectrum of the string theory.
Nonetheless, it is interesting to note that one of the other two choices was also discovered in \cite{Albrychiewicz:2024tqe} as the worldsheet symmetry algebra of \emph{``tropological'' open strings}. This finding presents a further case for the relevance of our work: there may exist connections between the methods developed in this paper and tropical geometry, which in turn could inform string theory.

There are two important features that the BCCA lacks that makes its study both challenging and warranted, these being a grading and a Virasoro subalgebra. Symmetry algebras of string theory worldsheets usually have both, making them amenable to thorough, quantum mechanical treatment using the power tools of vertex operator algebras and semi-infinite cohomology \cite{BPZ1984, FGZ1986, Borcherds, FLM, LianZuckerman}. Today, these toolkits can be coupled with the troves of knowledge on the representation theory of integer-graded infinite-dimensional Lie algebras, to which the the past two decades of active research in both physics \cite{BagchiGopakumarMandalMiwa, BanerjeeJatkarMukhiNeogi, BanerjeeJatkarDileepLodatoMukhiNeogi, CampoleoniGonzalezOblakRiegler, BatlleCampelloGomis, FarahmandParsaSheikh-Jabbari, GrumillerPerezSheikh-JabbariTroncosoZwikel, BatlleCampelloGomis2, Figueroa-OFarrillVishwa, BatlleFigueroaGomisVishwa, Aggarwal:2025hji} and mathematics \cite{ZhangDong, LuZhao10, GaoJiangPei, BaoJiangPei, Radobolja, MazorchukZhao, JiangZhang, LiuZhao, AdamovicLuZhao, AdamovicRadobolja, JiangPeiZhang, BabichenkoKawasetsuRidoutStewart, DongPeiXia, LiuPeiXiaZhao, DilxatGaoLiu} have contributed. However, with neither a grading nor a Virasoro subalgebra, the BCCA cannot be studied using these well-established techniques. Hence, understanding the representations of the BCCA presents a novel challenge for both physicists and mathematicians. We therefore develop and employ new approaches that we hope can lay the groundwork for future studies on infinite-dimensional Lie algebras that are not integer-graded while shedding some light on the spectrum of the open tensionless string.

Since the BCCA is defined as a subalgebra of the BMS$_3$ algebra, it is natural to ask whether irreducible modules over BMS$_3$ remain irreducible upon restriction to the BCCA. We therefore start by studying representations of the BCCA and its subalgebra $\OO \coloneqq \spn\{\OO_n \mid n \geq 1\}$ by restricting well-known representations of the Virasoro and BMS$_3$ algebras, namely the tensor density modules, Verma modules, the so-called ``induced modules'' introduced in \cite{BarnichOblak} which we prefer to call ``massive modules'' (see \cite{CampoleoniGonzalezOblakRiegler}), and the modules $\Omega(\lambda,a)$ introduced in \cite{LuZhao14}. We work over an arbitrary, characteristic-zero field $\kk$ to prove the following.

\begin{theorem}[Lemma \ref{lem:tensor density direct sum}, Propositions \ref{prop:V(h,c)_free_rank1}, \ref{prop:BMS Verma is almost free over b}, \ref{prop:massive module over b}, and \ref{Omega-simplicity}]\label{thm:restriction modules intro}
    Let $h_L, h_M, c_L, c_M, a, b, \mathbf{M}, \mathbf{s} \in \kk$, and $\lambda \in \kk^*$. Then the following hold.
    \begin{enumerate}
        \item The tensor density module $I(0,b)$ (see Definition \ref{def:tensor density}) decomposes into a direct sum of two $\OO$-submodules.
        \item When restricted to $\OO$, the Virasoro Verma module $V(h_L,c_L)$ (see Definition \ref{def:Vir_Verma_module}) is a free module of rank 1. That is, $V(h_L,c_L) \cong U(\OO)$ as $\OO$-modules.
        \item When restricted to the BCCA, the BMS$_3$ Verma module $V(h_L, h_M, c_L, c_M)$ (see Definition \ref{def:BMS_Verma_module}) has the following form:
        $$V(h_L, h_M, c_L, c_M) \cong \frac{U(\widehat{\bcca})}{U(\widehat{\bcca})\cdot(P_0 - 2h_M, C_M - c_M)}.$$
        \item The BCCA-module $U(\widehat{\bcca})\genvecbms$ (see Definition \ref{def:BMS_massive_module}) is irreducible if and only if $\mathbf{M} + \frac{n^2 - 1}{24}c_M \neq 0$ for any positive integer $n$.\label{item:massive module intro}
        \item The $\OO$-module $\Omega(\lambda,a)$ (see Subsection \ref{subsec:Omega module}) is irreducible if and only if $\lambda \neq \pm 1$ and $a \neq 0$.
    \end{enumerate}
\end{theorem}

Taking Theorem \ref{thm:restriction modules intro}\eqref{item:massive module intro} a bit further, we showed that the massive module $\widetilde{V}(\mathbf{M},\mathbf{s},c_L ,c_M)$ is decomposable as a $\widehat{\bcca}$-modules when $U(\widehat{\bcca}) \genvecbms$ is irreducible (i.e., $\mathbf{M} + \frac{n^2 - 1}{24}c_M \neq 0$ for any positive integer $n$).

\begin{theorem}[Propositions \ref{prop:P and Phat acting on Ohat_n} and \ref{pre-formula}, Corollary \ref{cor:bms_massive_module_decomposable}]\label{thm:intro decompose}
    Let $\mathbf{M} + \frac{n^2 - 1}{24}c_M \neq 0$ for any positive integer $n$. Then the massive module $\widetilde{V}(\mathbf{M},\mathbf{s},c_L ,c_M)$ decomposes into a direct sum of two $\widehat{\bcca}$-submodules:
    \begin{equation*}
        \widetilde{V}(\mathbf{M},\mathbf{s},c_L ,c_M) \cong \widehat{V} \oplus \widehat{W},
    \end{equation*}
    where $\widehat{V}$ and $\widehat{W}$ are given by Notation \ref{ntt:Ohat and Phat} and Definition \ref{def:Vhat and What}.
\end{theorem}

Theorems \ref{thm:restriction modules intro} and \ref{thm:intro decompose} show that many standard examples of BMS$_3$ modules are no longer irreducible upon restriction to the BCCA. In fact, we even see that the tensor density modules and the massive $\bms$-modules become decomposable when we restrict to the BCCA. On the other hand, Theorem \ref{thm:restriction modules intro}\eqref{item:massive module intro} shows that the massive modules do contain irreducible BCCA-submodules.

We then present the following conjecture as an avenue for further study.

\begin{conjecture}[Conjecture \ref{conj:Ind}]
    The $\widehat{\bcca}$-modules $\widehat{V}$ and $\widehat{W}$ are indecomposable.
\end{conjecture}

On the way to proving Theorem \ref{thm:restriction modules intro}\eqref{item:massive module intro}, we also characterise when the massive modules over the BMS$_3$ algebra\footnote{These are known as \emph{induced modules} in the physics literature, with the name ``massive modules'' reserved for a subset of them. In this paper, we choose to refer to all modules given in Definition \ref{def:BMS_massive_module} as massive modules.} are irreducible, which to our knowledge was not known before.

\begin{theorem}[Theorem \ref{thm:bms_massive_modules_irreducibility}]
    Let $\mathbf{M},\mathbf{s},c_L,c_M \in \kk$. Then the massive module $\widetilde{V}(\mathbf{M},\mathbf{s},c_L ,c_M)$ over the BMS$_3$ algebra (see Definition \ref{def:BMS_massive_module}) is irreducible if and only if $\mathbf{M} + \frac{n^2 - 1}{24}c_M \neq 0$ for any positive integer $n$.
\end{theorem}

This is in line with what is expected based on the gravitational interpretation of the coadjoint orbits of the BMS$_3$ group (see \cite{BarnichOblak2} for more details). However, it is surprising that the irreducibility conditions for the massive modules over $\bms$ and the $\widehat{\bcca}$-module $U(\widehat{\bcca})\genvecbms$ are the same.

Theorem \ref{thm:restriction modules intro} shows that many of the well-known modules over the Virasoro and BMS$_3$ algebras become free, or ``almost free'' when restricted to the BCCA. From the perspective of representation theory, these are not the most interesting BCCA-modules, so we later shift our focus to studying the representation theory of the BCCA intrinsically, by way of Whittaker modules (see Definition \ref{def:Whittaker module}). The main issue is that the presentation of the BCCA given in \eqref{eq:intro BCCA Lie bracket} makes it difficult to even define what a Whittaker module over the BCCA should be. For this reason, we spend some time constructing a different basis of the (centreless) BCCA which better lends itself to the construction of Whittaker modules.

The change of basis is achieved by analysing the subalgebra $\OO$, which is a subalgebra of the Witt algebra $\W \coloneqq \Der(\kk[t,t^{-1}])$, using techniques from \cite{Buzaglo2, BellBuzaglo}. In fact, we work in a more general setting: for $\lambda \in \kk^*$ and $n \in \ZZ$, let $\Olambda_n \coloneqq L_n - \lambda^n L_{-n}$, and define $\OO(\lambda) \coloneqq \spn\{\Olambda_n \mid n \geq 1\}$. It is easy to see that $\OO(\lambda)$ is a Lie algebra with bracket given by
$$[\Olambda_n,\Olambda_m] = (n - m)\Olambda_{n + m} - \lambda^m (n + m) \Olambda_{n - m}$$
for all $n,m \in \ZZ_+$ (see Lemma \ref{lem:O lambda bracket}). 

\begin{theorem}[Theorem \ref{thm:O lambda = L(s)}, Proposition \ref{prop:g lambda isomorphism}, and Corollary \ref{cor:two_O_subalgebras}]\label{thm:intro O lambda}
    Let $\lambda, \mu \in \kk^*$. Then $\OO(\lambda)$ has another basis $\{\ulambda_n \mid n \geq 1\}$ with Lie bracket
    $$[\ulambda_n, \ulambda_m] = (n - m)(\ulambda_{n + m} - 4\lambda \ulambda_{n + m - 2}).$$
    Consequently, $\OO(\lambda) \cong \OO(\mu)$ if and only if $\sqrt{\frac{\lambda}{\mu}} \in \kk$. In particular, if $\kk$ is algebraically closed (for example, if $\kk = \CC$), then $\OO(\lambda) \cong \OO$ for all $\lambda \in \kk$. If $\kk = \RR$, then
    $$\OO(\lambda) \cong \begin{cases}
        \OO, &\text{if } \lambda > 0, \\
        \OO(-1), &\text{if } \lambda < 0,
    \end{cases}$$
    but $\OO(-1) \not\cong \OO$.
\end{theorem}

We further present some consequences of Theorem \ref{thm:intro O lambda} in Section \ref{sec:change of basis}. For example, it becomes easy to show that $\OO$ is not a simple Lie algebra (in fact, it is not even perfect) -- see Corollary \ref{cor:O is not simple}. Certainly, the same is true for both $\bcca$ and $\widehat{\bcca}$.

Theorem \ref{thm:intro O lambda} tells us that the Lie subalgebras $\OO$ and $\OO(-1)$ of the Witt algebra are isomorphic as complex Lie algebras but not as real ones. This leads to an interesting physical implication. In \cite{BCCFGP2024}, the authors start with the BMS$_3$ algebra realised as vector fields which generate the infinitesimal transformations that leave the Carrollian structure on a cylinder (given by a symmetric $(0,2)$-tensor field $h$ that is corank-1 everywhere on the cylinder and a nowhere-vanishing vector field $\xi$ such that $\ker h$ is generated by $\xi$) invariant up to an overall conformal factor. Using this, the Lie algebra $\OO$ was constructed as the subalgebra of the Witt algebra consisting of the vector fields that preserve a choice of boundaries on the cylinder at $0$ and $\pi$.
Likewise, the Lie algebra $\OO(-1)$ can be viewed as the analogous Lie algebra when the boundaries are instead placed at $\tfrac{\pi}{2}$ and $\tfrac{3\pi}{2}$. But one can always recover the original boundary conditions by translating the periodic coordinate by $\tfrac{\pi}{2}$, which puts back in the setting where $\OO$ is the distinguished boundary-preserving subalgebra. This equivalence must be reflected algebraically too, which indicates that $\OO(-1)$ and $\OO$ must be isomorphic Lie algebras. Thus, by Theorem \ref{thm:intro O lambda}, we confirm that the authors \cite{BCCFGP2024} must be working over $\CC$ and not $\RR$. We provide details of this observation in Remark \ref{rem:different boundaries}.

Having constructed the alternate basis of $\OO$ in Theorem \ref{thm:intro O lambda}, we then extend it to a new basis of the full centreless BCCA.

\begin{theorem}[Corollary \ref{cor:new_basis_for_b}]\label{thm:intro new basis}
    There is another basis of $\bcca$ denoted $\{u_n, v_m \mid n \geq 1, m \geq 0\}$. The bracket of the centreless BCCA with this basis is given by
    \begin{align*}
        [u_n,u_m] &= (n - m)(u_{n + m} - 4u_{n + m - 2}) & &(n,m \geq 1), \\
        [u_n,v_m] &= (n - m)v_{n + m} - 4(n - m - 1)v_{n + m - 2} & &(n \geq 1, m \geq 0), \\
        [v_n,v_m] &= 0 & &(n,m \geq 0).
    \end{align*}
\end{theorem}

Using the new basis for $\bcca$ from Theorem \ref{thm:intro new basis}, we can easily construct a filtration for the centreless BCCA, allowing us to define and study Whittaker modules for $\bcca$, since filtered Lie algebras always have Whittaker modules (see Lemma \ref{lem:filtration leads to Whittaker}). For convenience, we shift the basis as follows: let $\UU_n \coloneqq u_{n + 2}$ and $\VV_n \coloneqq v_{n + 1}$ for $n \geq -1$.

\begin{proposition}[Proposition \ref{prop:filtration for BCCA}]
    For $n \geq -1$, define
    $$\F_n\OO \coloneqq \spn\{\UU_k \mid k \geq n\}, \quad \F_n\bcca \coloneqq \spn\{\UU_k, \VV_k \mid k \geq n\}.$$
    Then $\F$ is a filtration of $\OO$ and $\bcca$. 
\end{proposition}

Consequently, we define \emph{Whittaker functions} on $\bcca$ to be Lie algebra homomorphisms $\psi_n \colon \F_n \bcca \to \kk$ for some $n \in \NN$. Given a Whittaker function $\psi_n \colon \F_n \bcca \to \kk$, we let $\kk_{\psi_n} = \kk 1_{\psi_n}$ be the one-dimensional $\F_n \bcca$-module defined by $x \cdot 1_{\psi_n} = \psi_n(x) 1_{\psi_n}$ for $x \in \F_n \bcca$. The \emph{universal Whittaker module of type $\psi_n$} is defined to be $M_{\psi_n} \coloneqq \Ind_{\F_n \bcca}^\bcca \kk_{\psi_n}$. Similar definitions can be made for the subalgebra $\OO$.

\begin{theorem}[Theorems \ref{theo: simplicity} and \ref{thm:bcca Whittaker}]\label{theorem: Whittaker simplicity intro}
    Let $n \in \NN$, and let $\varphi_n \colon \F_n \OO \to \kk$ and $\psi_n \colon \F_n \bcca \to \kk$ be Whittaker functions. Then the following hold.
    \begin{enumerate}
        \item If $n = 0$, then the $\OO$-module $M_{\varphi_0}$ is irreducible if and only if $\varphi_0(\UU_{2}) \neq 4 \varphi_0(\UU_{0})$.\label{item:intro O Whittaker n = 0}
        \item If $n \geq 1$, then the $\OO$-module $M_{\varphi_n}$ is irreducible if and only if $\varphi_n(\UU_{2n + 1}) \neq 4 \varphi_n(\UU_{2n - 1})$ or $\varphi_n(\UU_{2n + 2}) \neq 4\varphi_n(\UU_{2n})$.\label{item:intro O Whittaker}
        \item If $n = 0$, then the $\bcca$-module $M_{\psi_0}$ is irreducible if and only if $\psi_0(\VV_0) \neq 0$.\label{item:intro BCCA Whittaker n = 0}
        \item If $n \geq 1$, then the $\bcca$-module $M_{\psi_n}$ is irreducible if and only if $\psi_n(\VV_{2n}) \neq 0$ or $\psi_n(\VV_{2n - 1}) \neq 0$.\label{item:intro BCCA Whittaker}
    \end{enumerate}
\end{theorem}

The proof of Theorem \ref{theorem: Whittaker simplicity intro}\eqref{item:intro O Whittaker n = 0} and \eqref{item:intro O Whittaker} uses a version of the classical Kirillov orbit method \cite{Kirillov2004LecturesOT} for the Virasoro algebra developed in \cite{Ushirobira1998, pham2025orbit}. For parts \eqref{item:intro BCCA Whittaker n = 0} and \eqref{item:intro BCCA Whittaker}, our proof is more direct.

Unfortunately, there is no way to extend the filtration to $\widehat{\bcca}$, at least not while keeping the nice properties of the filtration (particularly \emph{weak convergence} -- see Definition \ref{def:filtered Lie algebra}). We elaborate on this in Remark \ref{rem:lack of center}. This means that we are currently unable to define Whittaker modules for the BCCA on which the centre acts nontrivially. It would be interesting to further explore representations of the BCCA on which the centre acts nontrivially.

Certainly, there exists strong physical and mathematical motivation for this paper. The former arises from the fact that an infinite-dimensional Lie algebra which is not integer-graded has made an appearance in a modern theoretical physics setting, meaning that we need to develop the tools needed to study such objects thoroughly, starting with the BCCA as the inaugural case study. The latter comes from the desire to understand infinite-dimensional subalgebras of infinite-dimensional Lie algebras more thoroughly.

This paper is organised as follows. In Section \ref{sec:preliminaries}, we set up notation, present key definitions along with minor or known results, and introduce the main objects of study of this paper. 
Section \ref{sec:restriction of modules} constructs modules of the BCCA and one of its subalgebras by restricting well-known modules of the BMS$_3$ and Virasoro algebras. This would be the naive approach to the studying the representations of the BCCA based on its physical origins in \cite{BCCFGP2024}. The emergence of free and ``almost free'' modules motivates the need to study modules of the BCCA intrinsically. This requires the change of basis presented in Section \ref{sec:change of basis}, from which we can construct a decreasing filtration of the BCCA and compute derived subalgebras, as done in Section \ref{sec:filtration}.
All this effort pays off in Section \ref{sec:Whittaker modules}, in which we construct the universal Whittaker modules over the BCCA and its subalgebra $\OO$, and determine criteria for their irreducibility. 
Finally, in Section \ref{sec:conclusions}, we present a myriad of future directions of work that can be undertaken on the representation theory of the BCCA, from both physical and mathematical perspectives.

\vspace{1em}

\noindent \textbf{Acknowledgements:} The authors would like to thank Blagoje Oblak for fascinating, insightful discussions on the induced representations of the BMS$_3$ group and algebra, Martin Schlichenmaier for interesting discussions about the relationship between the BCCA and Krichever--Novikov algebras, and Jos\'e Figueroa-O'Farrill for comments on a draft of this paper.
GSV would also like to thank Daniel Grumiller and Priyadarshini Pandit for discussions on the BCCA and the tensionless open string during the conference titled ``Non-Lorentzian Geometries and their Applications'', which took place at the Hamilton Mathematics Institute, Trinity College Dublin, from 29\textsuperscript{th} April to 2\textsuperscript{nd} May 2025. GSV is supported by the Science and Technologies Facilities Council studentship [grant number 2615874]. XH is supported by NSFC (No. 12201027). HT is partially supported by NSFC (No. 12271085) and  CSC (No. 202506620188). KZ is partially supported by NSERC (311907-2026).

\section{Preliminaries}\label{sec:preliminaries}
Throughout this paper, we use the following notation:
\begin{itemize}
    \item $\kk$ denotes an arbitrary field of characteristic zero.
    \item $\kk^*$ denotes the set of nonzero elements of the field $\kk$.
    \item $\NN$ denotes the set of non-negative integers (including zero).
    \item $\ZZ_+$ and $\ZZ_-$ denote the set of positive and negative integers, respectively.
    \item We reserve upper case $C$ (and variations such as $C_L$ and $C_M$) for central elements of a Lie algebra, while lower case $c$ (and variations such as $c_L$ and $c_M$) represent the scalars by which the central elements act on a given representation of the Lie algebra.
\end{itemize}
All Lie algebras and vector spaces are considered over $\kk$ unless otherwise specified. We will also abbreviate $\otimes_\kk$ to $\otimes$.

\subsection{The boundary Carrollian conformal algebra}

We define the fundamental objects of interest in this paper and present the construction of the boundary Carrollian conformal algebra (BCCA).

\begin{defn}\label{def:Witt}
    The \emph{Witt algebra} is defined as $\W \coloneqq \Der(\kk[t,t^{-1}]) = \kk[t,t^{-1}]\del$, where $\del \coloneqq \diff{t}$. The standard basis of the Witt algebra is $L_n \coloneqq -t^{n + 1}\del$ ($n \in \ZZ)$, and its Lie bracket is given by
    $$[f\del, g\del] = (fg' - f'g)\del, \qquad [L_n,L_m] = (n - m)L_{n + m},$$
    for all $f,g \in \kk[t,t^{-1}]$ and $n,m \in \ZZ$. The Witt algebra $\W$ can be seen as the Lie algebra of algebraic vector fields on $\kk^*$.
    
    The \emph{Virasoro algebra} $\Vir$ is the universal central extension of the Witt algebra. As a vector space, $\Vir = \W \oplus \kk C$ with bracket given by
    $$[f\del,g\del] = (fg' - f'g)\del + \frac{1}{12} \operatorname{Res}_0(f'g'')C, \qquad [L_n,L_m] = (n - m)L_{n + m} + \frac{1}{12} n(n^2 - 1) \delta_{m,-n} C,$$
    for all $f,g \in \kk[t,t^{-1}]$, where $C$ is central, and $\operatorname{Res}_0(h)$ is the residue of $h \in \kk[t,t^{-1}]$ at zero, or in other words, the coefficient of $t^{-1}$ in $h$.
    
    The \emph{one-sided Witt algebra} is defined as $\WW_1 \coloneqq \Der(\kk[t]) = \kk[t]\del$. Equivalently, it can be viewed as the subalgebra of $\W$ spanned by $\{L_n \mid n \geq -1\}$ and it is the Lie algebra of algebraic vector fields on $\kk$.
\end{defn}

We often switch between the \emph{basis-free} notation $f\del$ for elements of $\W$ and the basis notation $L_n$.

\begin{defn}\label{def:tensor density}
    For $a,b \in \kk$, define a $\W$-module $I(a,b) = \spn\{I_n \mid n \in \ZZ\}$ as follows:
    $$L_n \cdot I_m = - (a + bn + m)I_{n + m}.$$
    The modules $I(a,b)$ are known as \emph{tensor density modules}. In basis-free notation, we define $I(a,b) \coloneqq t^{a - b}\kk[t,t^{-1}] \ dt^b$ with $\W$-action given by
    $$f\del \cdot (g \ dt^b) = (fg' + bf'g)dt^b.$$
    With this perspective, we have $I_n \coloneqq - t^{n + a - b} \ dt^b$. Note that $I(0,-1)$ is the adjoint module of $\W$.
\end{defn}

We may then construct the two-parameter family of Lie algebras $\W(a,b) \coloneqq \W \ltimes I(a,b)$. The special case $\W(0,-1)$ is the centreless version of the BMS$_3$ algebra, presented in the next definition. The BMS$_3$ algebra first appeared in physics in \cite{AshtekarBicakSchmidt, BarnichCompere}, and in mathematics it was introduced as the $W(2,2)$ algebra\footnote{Not to be confused with our $\W(a,b)$ algebras.} by Zhang and Dong but with $C_L = C_M$ \cite{ZhangDong}.

\begin{defn}\label{def:BMS}
    The \emph{BMS$_3$ algebra}, denoted $\bms$, is the universal central extension of $\W(0,-1)$. It is the infinite dimensional Lie algebra spanned by $\{L_n,M_n \mid n \in \ZZ\}$ and two central elements $C_L, C_M$, with Lie bracket
    \begin{equation}\label{eq:bms bracket}
        \begin{gathered}
            [L_n,L_m] = (n - m) L_{m + n} + \frac{1}{12} n(n^2 - 1) \delta_{m,-n} C_L, \\
            [L_n,M_m] = (n - m) M_{m + n} + \frac{1}{12} n(n^2 - 1)\delta_{m,-n} C_M, \\
            [M_n,M_m] = 0.
        \end{gathered}
    \end{equation}
    Here, $\{L_n \mid n\in\ZZ\}$ and $C_L$ form a basis of a Virasoro subalgebra and $\{M_n \mid n \in \ZZ\}$ is the basis of $I(0,-1)$ presented in Definition \ref{def:tensor density}.
\end{defn}

We now define the BCCA in its original form given by \cite[Equation (7)]{BCCFGP2024}. This will be the main object of study of our paper.

\begin{defn}\label{def:BCCA}
    Let $\OO_n\coloneqq L_n - L_{-n}$ and $P_n \coloneqq M_n + M_{-n}$ for all $n \in \ZZ$, which means that $\OO_n = -\OO_{-n}$ and $\PP_n = \PP_{-n}$ for all $n\geq 1$. The \emph{boundary Carrollian conformal algebra (BCCA)}, denoted $\widehat{\bcca}$, is the Lie subalgebra of $\bms$ with basis $\{\OO_n, P_m, C_M \mid n \geq 1,\ m \geq 0\}$ and Lie bracket
    \begin{equation} \label{eq:BCCA Lie bracket}
        \begin{gathered}
            [\OO_n,\OO_m] = (n - m)\OO_{n + m} - (n + m)\OO_{n - m}, \\
            [\OO_n, P_m] = (n - m)P_{n + m} + (n + m)P_{n - m} + \frac{1}{6}n(n^2 - 1) \delta_{m,n} C_M, \\
            [P_n,P_m] = 0.
        \end{gathered}
    \end{equation}
    The \emph{centreless BCCA} (in other words, setting $C_M = 0$) is denoted $\bcca$.
\end{defn}

The latter half of this paper mostly focuses on the centreless BCCA. This is because it is not clear how one could define Whittaker modules over the BCCA with centre. We elaborate on this in Remark \ref{rem:lack of center}.

\begin{remark}
    The Lie algebra $\bcca$ is the fixed subalgebra of $\W(0,-1)$ under the automorphism $L_n \mapsto -L_{-n},\, M_n \mapsto M_{-n}$.
\end{remark}

First, we seek to write $\bcca$ as a semi-direct sum in a similar manner to how the centreless BMS$_3$ algebra can be written as the semi-direct sum $\W(0,-1) \coloneq \W \ltimes I(0,-1)$. To do this, we define the following subspaces of $\W$ and $I(a,b)$.

\begin{defn}\label{def:O_subalgebra}
    The subalgebra $\OO \subseteq \W$ is spanned by $\{\OO_n \mid n \geq 1\}$, where $\OO_n \coloneqq L_n - L_{-n}$. Its Lie bracket is given by
    \begin{equation}\label{eq:O_subalgebra_bracket}
        [\OO_n, \OO_m] = (n - m)\OO_{n + m} - (n + m)\OO_{n - m}.
    \end{equation}
    Let $I(0,b) = \spn\{I_n \mid n \in \ZZ\}$ be the tensor density modules from Definition \ref{def:tensor density}. Let $\Pb_n \coloneqq I_n + I_{-n}$ and construct $\PP_b \coloneqq \spn\{\Pb_n \mid n \in \NN\}$. Likewise, let $\Ptildeb_n \coloneqq I_n - I_{-n}$ and construct $\widetilde{\PP}_b \coloneqq \spn\{\Ptildeb_n \mid n \in \NN^+\}$. 
    The vector spaces $\PP_b$ and $\PPtilde_b$ are $\OO$-modules under the following actions of $\OO$ inherited from the action of $\W$ on $I(0,b)$:
    \begin{align}
        &\OO_n \cdot P_m^{(b)} = -(bn + m) P_{n + m}^{(b)} - (bn - m) P_{n-m}^{(b)} \quad (n \geq 1, m \geq 0) \label{eq:O action on Pb} \\
        &\OO_n \cdot \Ptilde_m^{(b)} = -(bn + m) \Ptilde_{n + m}^{(b)} - (bn - m) \Ptilde_{n-m}^{(b)} \quad (n,m \geq 1). \label{eq:O action on Ptildeb}
    \end{align}
    For $b = -1$, we write $\PP\coloneqq\PP_{-1}$. Note that $\PPtilde_{-1}$ is the adjoint module of $\OO$.
\end{defn}

It is easy to see that $I(0,b)$ decomposes as a direct sum of $\PP_b$ and $\PPtilde_b$ as an $\OO$-module.

\begin{lemma}\label{lem:tensor density direct sum}
    As $\OO$-modules, $I(0,b)\cong \PP_b \oplus \PPtilde_b$.
\end{lemma}
\begin{proof}
    There exists an $\OO$-module homomorphism $\rho\colon I(0,b)\to\PP_b$ given by $\rho(I_n) = P_n$ for all $n\in\ZZ$. We then have a short exact sequence of $\OO$-modules
    \begin{equation}
        \begin{tikzcd}
            0 \arrow[r] & \PPtilde_b \arrow[r] & I(0,b) \arrow[r, "\rho"] & \PP_b \arrow[r] & 0.
        \end{tikzcd}
    \end{equation}
    The inclusion $\PP_b \hookrightarrow I(0,b)$ is a section of the short exact sequence, and thus the sequence is split. It follows that $I(0,b)\cong \PP_b \oplus \PPtilde_b$ for all $b\in\kk$.  
\end{proof}

We therefore see that Definition \ref{def:BCCA} can be succinctly summarised as $\bcca \coloneqq \OO \ltimes \PP_{-1} = \OO \ltimes \PP$, analogous to how the centreless BMS$_3$ algebra is defined as $\W(0,-1) \coloneqq \W \ltimes I(0,-1)$. The \emph{higher spin (centreless) BCCAs} can be similarly defined as $\OO \ltimes \PP_b$ for $b \in \kk$. See the commentary around \cite[Equation (26)]{BCCFGP2024} for more details. We only focus on the case $b = -1$ in this paper, but our techniques could also be applied to these higher spin algebras.

\subsection{Whittaker modules}

Whittaker modules were first discovered for $\mathfrak{sl}_2$ by Arnal and Pinczon in \cite{ArnalPinczon}, following which Kostant defined them for finite-dimensional complex semisimple Lie algebras in \cite{Kostant}. This was then generalised to arbitrary Lie algebras with triangular decomposition. 
Today, Whittaker modules are defined more generally via the notion of ``Whittaker pairs'' \cite{BatraMazorchuk, MazorchukZhao} and  play an important role in the representation theory of various infinite-dimensional Lie algebras \cite{Konstantina, OndrusWiesner, ZhangTanLian, GuoLiu, Bin, LiuWuZhu, AdamovicLuZhao, LiuPeiXia, ChenJiang, Chen, DilxatGaoLiu, ChenGeLiWang}. 

In this paper, we construct Whittaker modules of $\OO$ and $\bcca$ based on the definition given in \cite[Section 3.2]{BatraMazorchuk} and \cite[Section 5.1]{MazorchukZhao}, thereby presenting explicit examples of their constructions for infinite-dimensional Lie algebras that are not integer-graded. 

\begin{defn}
    Let $\g$ be a Lie algebra.
    \begin{enumerate}
        \item Define $\g^0 \coloneq \g$ and $\g^{k+1}\coloneq [\g^k, \g]$ for all $k\in\NN$. We say that $\g$ is \emph{quasi-nilpotent} if $\bigcap_{k=0}^\infty \g^k = 0$.
        \item A $\g$-module $V$ is said to be \emph{locally nilpotent} if for every $v\in V$, there exists $s(v)\in \NN$ such that $x_1 x_2 \dots x_{s(v)}\cdot v = 0$ for all $x_1,\dots, x_{s(v)} \in \g$.
    \end{enumerate}
\end{defn}

Equipped with the above definition, we can now define Whittaker modules.

\begin{defn}\label{def:Whittaker module}
    Let $\g$ be a Lie algebra with a quasi-nilpotent subalgebra $\mathfrak{n} \subseteq \g$ such that $\g/\mathfrak{n}$ is a locally nilpotent $\mathfrak{n}$-module under the adjoint action, and let $V$ be a $\g$-module.
    \begin{enumerate}
        \item The $\g$-module $V$ is a \emph{Whittaker module} if $U(\mathfrak{n}) \cdot v$ is finite-dimensional for all $v \in V$ (in other words, $V$ is a \emph{locally finite} $\mathfrak{n}$-module).
        \item A Lie algebra homomorphism $\varphi \colon \mathfrak{n} \to \kk$ is called a \emph{Whittaker function} (so $\varphi([\mathfrak{n}, \mathfrak{n}]) = 0$). A vector $v \in V$ is called a \emph{Whittaker vector of type $\varphi$} if $x \cdot v = \varphi(x) v$ for all $x \in \mathfrak{n}$. The $\g$-module $V$ is called a \emph{Whittaker module of type $\varphi$} if it is generated by a Whittaker vector of type $\varphi$.
        \item Let $\kk_{\varphi} \coloneq \kk 1_\varphi$ be the one-dimensional $\mathfrak{n}$-module defined by $\varphi$, in other words, $x \cdot 1_\varphi = \varphi(x)1_\varphi$ for all $x \in \mathfrak{n}$. Then the induced module
        $$M_\varphi \coloneq \Ind_\mathfrak{n}^\g \kk_\varphi$$ 
        is the \emph{universal Whittaker module of type $\varphi$}.
    \end{enumerate}
\end{defn}

We proceed to briefly show how Definition \ref{def:Whittaker module} can be readily applied to filtered Lie algebras.

\begin{defn} \label{def:filtered Lie algebra}
    A \emph{filtered Lie algebra} is a Lie algebra $\g$ with a decreasing sequence of subalgebras of $\g$
    $$\dots \supseteq \g_{-1} \supseteq \g_0 \supseteq \g_1 \supseteq \g_2 \supseteq \cdots$$
    such that $[\g_i, \g_j] \subseteq \g_{i + j}$. We further demand that the following two conditions are satisfied:
    \begin{enumerate}
        \item The filtration is \emph{weakly convergent}. That is, $\bigcap_{k = 0}^\infty \g_k = 0$.\label{item:weakly convergent}
        \item The filtration is \emph{bounded from below}. That is, there exists $N \in \ZZ$ such that $\g_{k} = \g$ for all $k < N$.\label{item:bounded from below}
    \end{enumerate}
\end{defn}

Using this definition, it is easy to show that filtered Lie algebras always have Whittaker modules.

\begin{lemma}\label{lem:filtration leads to Whittaker}
    Let $\g$ be a filtered Lie algebra as per Definition \ref{def:filtered Lie algebra}. If $n \geq 1$, then $\g_n$ is a quasi-nilpotent Lie algebra, and $\g/\g_n$ is a locally nilpotent $\g_n$-module.
    
    In particular, one can always construct Whittaker modules for $\g$ by choosing $\mathfrak{n} \coloneqq \g_n$ in Definition \ref{def:Whittaker module} for any $n \geq 1$.
\end{lemma}
\begin{proof}
    Use conditions \eqref{item:weakly convergent} and \eqref{item:bounded from below} from Definition \ref{def:filtered Lie algebra}.
\end{proof} 

We introduce a related but slightly different class of modules known as \emph{quasi-Whittaker modules} that were recently studied in \cite{ChengGaoLiuZhaoZhao}.

\begin{defn}\label{def:quasi-Whittaker modules}
    Let $\h$ be a non-perfect ideal of a non-semisimple Lie algebra $\g$, let $\varphi \colon \h \to \kk$ be a Lie algebra homomorphism, and let $V$ be a $\g$-module.
    \begin{enumerate}
        \item A vector $v \in V$ is called a \emph{quasi-Whittaker vector of type $\varphi$} if $h\cdot v =  \varphi(h) v$ for all $h\in\h$. The $\g$-module $V$ is called a \emph{quasi-Whittaker module of type $\varphi$} if it is generated by a Whittaker vector of type $\varphi$.
        \item Let $\kk_{\varphi} \coloneq \kk 1_\varphi$ be the one-dimensional $\mathfrak{h}$-module defined by $\varphi$. Then the induced module
        $$W_\varphi \coloneq \Ind_\mathfrak{h}^\g \kk_\varphi$$ is the \emph{universal quasi-Whittaker module of type $\varphi$}.
    \end{enumerate}
\end{defn}

\begin{remark}
    Definition \ref{def:quasi-Whittaker modules} applies to $\mathfrak{h}$, which need not be quasi-nilpotent, as opposed to $\mathfrak{n}$ in Definition \ref{def:Whittaker module}. However, it requires that $\mathfrak{h}$ is an ideal of $\g$, while $\mathfrak{n}$ need not be.
\end{remark}

Next, we define the notion of the \emph{Whittaker annihilator} corresponding to a Whittaker function, a key concept introduced in \cite{ChengGaoLiuZhaoZhao}.

\begin{defn} \label{def:Whittaker annihilator}
    Let $\h$ be a non-perfect ideal of a non-semisimple Lie algebra $\g$, and let $\varphi \colon \h \to \kk$ be a Lie algebra homomorphism. The \emph{Whittaker annihilator of $\varphi$ over $\g$} is defined to be
    $$\g^\varphi \coloneqq \{x \in \g \mid \varphi([x,y]) = 0 \text{ for all } y \in \h\}.$$
\end{defn}

We conclude this section by presenting two key results from \cite{ChengGaoLiuZhaoZhao} which emphasise the importance of the Whittaker annihilator.

\begin{theorem}[{\cite[Theorem 3.3 and Corollary 3.12]{ChengGaoLiuZhaoZhao}}]\label{thm:Whittaker annihilator}
    Let $\g$ be non-semisimple Lie algebra, $\h$ a non-perfect ideal of $\g$, and $\varphi \colon \h \to \kk$ a Lie algebra homomorphism. Then the following hold.
    \begin{enumerate}
        \item The universal quasi-Whittaker module $W_\varphi = \Ind_\h^\g \kk_\varphi$ is irreducible if and only if $\g^\varphi = \h$.
        \item If $\dim(\g^\varphi/\h) = 1$, choose $y \in \g^\varphi \setminus \h$, so that $\g^\varphi = \h \oplus \kk y$. Then all maximal submodules of $W_\varphi$ are of the form
        $$J_\xi \coloneqq U(\g)(y - \xi) \cdot 1_\varphi$$
        for $\xi \in \kk$.
    \end{enumerate}
\end{theorem}

\section{Representations of \texorpdfstring{$\OO$}{O} and \texorpdfstring{$\widehat{\bcca}$}{b} by restriction}\label{sec:restriction of modules}

Recall that in its original form presented in \cite{BCCFGP2024}, the BCCA emerged as the subalgebra 
$$\widehat{\bcca}\coloneqq\spn\{\OO_n\coloneqq L_n - L_{-n}, P_m \coloneqq M_m + M_{-m}, C_M \mid n \geq 1,\ m \geq 0\}$$ 
of the BMS$_3$ algebra. Hence, the most natural course of action is to study the restrictions of representations of $\Vir$ and $\bms$ to $\OO$ and $\widehat{\bcca}$ respectively. In doing so, we will also motivate the need for the alternative, intrinsic approaches of constructing representations of $\bcca$ presented in this paper. Since the representation theory of $\Vir$ and $\bms$ is fairly extensive, we restrict our attention to the most popular choices of modules over these algebras, namely the Verma modules (over both) and the so-called massive modules (over $\bms$).

\subsection{\texorpdfstring{$\OO$}{O}-modules via restriction of Verma modules}

We first recall the definition of a Verma module of the Virasoro algebra. Recalling Definition \ref{def:Witt}, the Virasoro algebra admits a triangular decomposition 
$$\Vir = \Vir_+ \oplus \Vir_0 \oplus \Vir_{-}, \text{ where } \Vir_{\pm} = \bigoplus_{n \in \ZZ_\pm} \kk L_{n} \text{ and } \Vir_0 = \kk L_0 \oplus \kk C.$$

\begin{defn}\label{def:Vir_Verma_module}
    Let $h,c \in \kk$. Define the one-dimensional $\Vir_0\oplus\Vir_+$-module $\kk_{h,c} \coloneqq \kk\ket{h,c}$, where
    $$L_0\ket{h,c} = h\ket{h,c}, \quad C\ket{h,c} = c\ket{h,c}, \quad \text{and } \Vir_+\ket{h,c}=0.$$
    The \emph{Verma module of the Virasoro algebra} $V(h,c)$ is the $\Vir$-module induced from $\kk_{h,c}$. That is, 
    $$V(h,c) \coloneqq \Ind_{\Vir_0\oplus\Vir_+}^{\Vir} \kk_{h,c}.$$ 
\end{defn}

\begin{remark} \label{rem:V(h,c) as quotient ideal}
    For all $h,c \in \kk$, the module $V(h,c)$ admits a Poincar\'{e}--Birkhoff--Witt (PBW) basis of monomials 
    \begin{equation} \label{eq:PBW_monomial_Vir}
        L_{-n_k} \dots L_{-n_1} \ket{h,c},
    \end{equation}
    where $n_k \geq n_{k-1} \geq  \dots \geq n_1\geq 1$. Therefore, $V(h,c)$ is cyclically generated by the action of $\Vir_-$ on $\ket{h,c}$. As left $U(\Vir)$-modules,
    $$V(h,c) \cong U(\Vir) / \big(U(\Vir) \cdot (L_{n>0}, L_0 - h, C - c)\big).$$
\end{remark}

The following proposition describes $V(h,c)$ as an $\OO$-module.

\begin{proposition}\label{prop:V(h,c)_free_rank1}
    The Verma module $V(h,c)$ is a free $U(\OO)$-module of rank 1 for all $h,c \in \kk$.
\end{proposition}
\begin{proof}
    Consider the basis $\{\OO_n, L_m, C \mid n \geq 1, m \geq 0\}$ of $\Vir$. This then induces a different PBW basis on $U(\Vir)$:
    $$\OO_{n_k} \dots \OO_{n_1} L_{m_\ell} \dots L_{m_1} C^r,$$
    where $1 \leq n_1 \leq \dots \leq n_k$, $0 \leq m_1 \leq \dots \leq m_\ell$, and $r \in \NN$. In this PBW basis, the description of $V(h,c)$ given in Remark \ref{rem:V(h,c) as quotient ideal} as a quotient of $U(\Vir)$ by a left ideal readily translates to the fact that
    $$\{\OO_{n_k} \dots \OO_{n_1} \ket{h,c} \mid 1 \leq n_1 \leq \dots \leq n_k\}$$
    is a basis for $V(h,c)$ with no relations on the $\OO_n$s. This proves the proposition.
\end{proof}

Free modules are not the most interesting from a representation theoretic point of view as they are not irreducible. This motivates our approach in the later sections, where we change basis and construct $\OO$-modules as modules of a subalgebra of $\WW_1$ that is isomorphic to $\OO$. This will give us explicit examples of classes of irreducible $\OO$-modules.

\subsection{\texorpdfstring{$\widehat{\bcca}$}{b}-modules via restriction of Verma modules}

Just like the Virasoro algebra, the BMS$_3$ algebra is also a graded Lie algebra: we have $\bms = \bigoplus_{n\in\ZZ} \bms_n$, where $\bms_{n} \coloneqq \kk L_n \oplus \kk M_n$ for $n \neq 0$, and $\bms_0\coloneqq \kk L_0 \oplus \kk M_0 \oplus \kk C_L \oplus \kk C_M$. 
Its triangular decomposition is given by 
$\bms = \bms_+ \oplus \bms_0 \oplus \bms_{-}$, where
$\bms_{\pm} = \bigoplus_{n \in \ZZ_\pm} \kk L_n \oplus \kk M_n$.
\begin{defn}\label{def:BMS_Verma_module}
    Let $h_L, h_M, c_L, c_M \in \kk$. Define the one-dimensional $\bms_0 \oplus \bms_+$-module
    $$\kk_{h_L,h_M,c_L,c_M} \coloneqq \kk\genvecbmsVerma,$$
    where
    \begin{equation*}\label{eq:bms_Verma_module_gen_vec}
    \begin{alignedat}{2}
       &L_0\genvecbmsVerma = h_L\genvecbmsVerma,  \quad && M_0\genvecbmsVerma = h_M\genvecbmsVerma, \\
       &C_L\genvecbmsVerma = c_L\genvecbmsVerma,  \quad && C_M\genvecbmsVerma = c_M\genvecbmsVerma,
    \end{alignedat}
    \end{equation*}
    and $L_{n>0} \genvecbmsVerma = M_{n>0} \genvecbmsVerma = 0$.
    The \emph{Verma module of the BMS$_3$ algebra} $V(h_L,h_M,c_L,c_M)$ is the $\bms$-module induced from $\kk_{h_L,h_M,c_L,c_M}$. That is, 
    $$V(h_L,h_M,c_L,c_M) \coloneqq \Ind_{\bms_0\oplus\bms_+}^{\bms}\kk_{h_L,h_M,c_L,c_M}.$$ 
\end{defn}

The Verma module $V(h_L,h_M,c_L,c_M)$ admits a Poincar\'{e}--Birkhoff--Witt basis of monomials 
$$L_{-n_k} \dots L_{-n_1}M_{-m_\ell} \dots M_{-m_1}\genvecbmsVerma,$$
where $n_k \geq n_{k-1} \geq  \dots \geq n_1\geq 1$, $m_\ell \geq m_{\ell-1} \geq \dots m_1 \geq 1$.
$V(h_L,h_M,c_L,c_M)$ is therefore cyclically generated by the action of $\bms_-$ on $\genvecbmsVerma$. 
As left $U(\bms)$-modules,
$$V(h_L,h_M,c_L,c_M) \cong U(\bms) / \big(U(\bms) \cdot (L_{n>0},M_{n>0}, L_0-h_L, M_0-h_M, C_L-c_L, C_M-c_M)\big).$$
The following proposition describes $V(h_L,h_M,c_L,c_M)$ as a $\widehat{\bcca}$-module.

\begin{proposition}\label{prop:BMS Verma is almost free over b}
    As $\widehat{\bcca}$-modules, 
    $$V(h_L, h_M, c_L, c_M) \cong \frac{U(\widehat{\bcca})}{U(\widehat{\bcca})\cdot(P_0-2h_M, C_M - c_M)}.$$
\end{proposition}
\begin{proof}
    Similar to the proof of Proposition \ref{prop:V(h,c)_free_rank1}, with the only relations given by the ideal by which we quotient coming from $P_0, C_M \in \widehat{\bcca}$. 
\end{proof}

Certainly, one could analyse the structure of the irreducible quotient of a reducible $\bms$-Verma module upon restriction to $\widehat{\bcca}$. However, as shown in \cite{Radobolja, JiangZhang, JiangPeiZhang, JiangLiuPeiZhao}, the structure of such irreducible quotients is rather complicated, so an explicit description of this $\widehat{\bcca}$-module would require careful analysis. That being said, we do not expect such $\widehat{\bcca}$-modules to be particularly interesting from a representation theoretic perspective, similarly to the free or ``almost-free'' modules appearing in Propositions \ref{prop:V(h,c)_free_rank1} and \ref{prop:BMS Verma is almost free over b}.

\subsection{\texorpdfstring{$\widehat{\bcca}$}{b}-modules via restriction of massive modules}

We consider a class of BMS$_3$ modules which we call massive modules\footnote{In the physics literature, the modules we introduce in this subsection are called \emph{induced modules}, and the term ``massive modules'' only refers to a subset of them. However, we use the term ``massive modules'' to refer to all such modules instead, as the term ``induced modules'' would be too vague in our context.} \cite{BarnichOblak, BarnichOblak2, CampoleoniGonzalezOblakRiegler}, defined as follows.
\begin{defn} \label{def:BMS_massive_module}
    Let $\h$ denote the subalgebra of $\bms$ spanned by $\{M_n \mid n\in\ZZ\}$ and $L_0$. In other words, $\h$ is the \emph{extension-by-derivation} of the abelian ideal $I(0,-1)$ of $\bms$ spanned by $\{M_n \mid n \in \ZZ\}$, where $\ad_{L_0}$ is the derivation by which it is extended. This is summarised by the short exact sequence
    \begin{equation*}
        \begin{tikzcd}
            0 \arrow[r] & I(0,-1) \arrow[r] & \h \arrow[r] & \kk L_0 \arrow[r] & 0.
        \end{tikzcd}
    \end{equation*}
    Let $\widehat{\h}$ be the subalgebra $\h \oplus \kk C_L \oplus \kk C_M$.
    Letting $\mathbf{M}, \mathbf{s}, c_L, c_M \in \kk$, consider the one-dimensional $\widehat{\h}$-module 
    $$\widetilde{\kk}_{\mathbf{M}, \mathbf{s}, c_L, c_M} = \kk\genvecbms,$$
    where
    \begin{equation*}\label{eq:bms_ind_module_gen_vec}
    \begin{alignedat}{2}
       &L_0\genvecbms = \mathbf{s}\genvecbms,  \quad && M_0\genvecbms = \mathbf{M}\genvecbms,  \\
       &C_L\genvecbms = c_L\genvecbms, \quad && C_M\genvecbms = c_M\genvecbms,
    \end{alignedat}
    \end{equation*}
    and $M_{n}\genvecbms = 0$ for $n \neq 0$. The $\bms$-modules $\widetilde{V}(\mathbf{M}, \mathbf{s}, c_L, c_M)$ induced from $\widetilde{\kk}_{\mathbf{M}, \mathbf{s},c_L,c_M}$ are called \emph{massive modules} \cite{BarnichOblak, BarnichOblak2, CampoleoniGonzalezOblakRiegler, BagchiBanerjeeChakraborttyDuttaParekh}. 
    That is,
    \begin{equation*}
        \widetilde{V}(\mathbf{M}, \mathbf{s}, c_L, c_M) \coloneqq \Ind_{\widehat{\h}}^{\bms} \widetilde{\kk}_{\mathbf{M},\mathbf{s},c_L,c_M}.
    \end{equation*}
\end{defn}

The massive module $\widetilde{V}(\mathbf{M}, \mathbf{s}, c_L, c_M)$ admits a PBW basis of monomials
\begin{equation*}
    L_{n_k} \dots L_{n_1} \genvecbms,
\end{equation*}
where $n_1,\dots, n_k \in \ZZ\setminus\{0\}$ with $n_k \geq \dots \geq n_1$.

The following result characterises the reducibility of the massive modules in terms of the parameters $\mathbf{M},\mathbf{s},c_L,c_M$. To the best of our knowledge, this result has not been explicitly stated and proven in the literature before.

\begin{theorem}\label{thm:bms_massive_modules_irreducibility}
    Let $\mathbf{M},\mathbf{s},c_L,c_M \in \kk$. The massive $\bms$-module $\widetilde{V}(\mathbf{M},\mathbf{s},c_L ,c_M)$ is irreducible if and only if $\mathbf{M}+\frac{n^2-1}{24}c_M\ne0$ for any positive integer $n$.
\end{theorem}
\begin{proof}
    We would like to apply Theorem \ref{thm:Whittaker annihilator}. Hence, we first aim to relate $\widetilde{V}(\mathbf{M},\mathbf{s},c_L,c_M)$ to a universal quasi-Whittaker module of $\bms$ as per Definition \ref{def:quasi-Whittaker modules}. Consider the ideal $\fa \coloneqq \spn\{C_M, C_L, M_k \mid k \in \mathbb{Z}\}$ of $\bms$ and define the Whittaker function $\phi \colon \fa \to \kk$ given by
    $$\phi(M_i) = \delta_{i,0} \mathbf{M}, \quad \phi(C_L) = c_L, \quad \phi(C_M) = c_M.$$
    Letting $W_\phi \coloneqq \Ind_\fa^\bms \kk_\phi$ be the associated universal quasi-Whittaker module, it is easy to see that
    $$\widetilde{V}(\mathbf{M}, \mathbf{s},c_L ,c_M) \cong W_\phi/J_{\mathbf{s}},$$
    where $J_{\mathbf{s}} \coloneqq U(\bms)(L_0 - \mathbf{s}) \cdot 1_\phi$.
    
    Suppose first that $\mathbf{M} + \frac{n^2-1}{24}c_M \ne 0$ for any positive integer $n$. By setting $n = 1$, we see that $\mathbf{M} \neq 0$. From Definition \ref{def:Whittaker annihilator}, we may quickly deduce that the Whittaker annihilator of $\phi$ is $\fa \oplus \kk L_0$. Now, Theorem \ref{thm:Whittaker annihilator} implies that $J_{\mathbf{s}}$ is a maximal $\bms$-submodule of $W_\phi$. Therefore, the quotient $\widetilde{V}(\mathbf{M}, \mathbf{s},c_L ,c_M) \cong W_\phi/J_{\mathbf{s}}$ is irreducible.
    
    Conversely, assume that $\mathbf{M} + \frac{n^2 - 1}{24} c_M = 0$ for a positive integer $n$. It is easy to compute that the Whittaker annihilator of $\phi$ contains $\fa \oplus \kk L_0 \oplus \kk L_n \oplus \kk L_{-n}$. Therefore, we see that $L_{n}\genvecbms$ generates a nonzero proper submodule of $\widetilde{V}(\mathbf{M}, \mathbf{s},c_L,c_M)$. Thus, in this case, $\widetilde{V}(\mathbf{M}, \mathbf{s},c_L ,c_M)$ is reducible.
\end{proof}

\begin{remark}
    By comparing \cite[Theorem 2.3]{JiangZhang} with Theorem \ref{thm:bms_massive_modules_irreducibility}, we see that the massive modules $\widetilde{V}(\mathbf{M}, \mathbf{s}, c_L, c_M)$ are irreducible under the same conditions as the Verma modules of $\bms$ from Definition \ref{def:BMS_Verma_module}. This is an unexpected, intriguing similarity between two very different classes of $\bms$-modules that could potentially be explained by structural features of these modules that are yet to be discovered.
\end{remark}

Using the findings from the proof of Theorem \ref{thm:bms_massive_modules_irreducibility} and another result from \cite{ChengGaoLiuZhaoZhao}, it is not difficult to show that $L_n \genvecbms$ and $L_{-n} \genvecbms$ together generate a maximal submodule of the massive module $\widetilde{V}(\mathbf{M}, \mathbf{s}, c_L, c_M)$ when $\mathbf{M} + \frac{n^2 - 1}{24}c_M = \mathbf{s} + \frac{n^2 - 1}{24}c_L = 0$ and $c_M \neq 0$.

\begin{corollary} \label{cor:bms_massive_modules_irreducible_quotient}
    Let $\mathbf{M}, \mathbf{s}, c_L, c_M \in \kk$ with $c_M \neq 0$, and suppose 
    $\mathbf{M} + \frac{n^2 - 1}{24}c_M = \mathbf{s} + \frac{n^2 - 1}{24} c_L = 0$
    for some positive integer $n$. Then the $\bms$-submodule of the massive module $\widetilde{V}(\mathbf{M}, \mathbf{s}, c_L, c_M)$ generated by 
    $L_n\genvecbms$ and $L_{-n}\genvecbms$ is a maximal submodule of $\widetilde{V}(\mathbf{M}, \mathbf{s}, c_L, c_M)$. 
\end{corollary}
\begin{proof}
    Let $\phi \colon \fa \to \kk$ be as in the proof of Theorem \ref{thm:bms_massive_modules_irreducibility}. Let
    $$\widetilde{\fa} \coloneqq \spn\{L_0, L_n, L_{-n}, M_k, C_M, C_L \mid k \in \ZZ\} = \fa \oplus \kk L_0 \oplus \kk L_n \oplus \kk L_{-n},$$
    and define $\Phi \colon \widetilde{\fa} \to \kk$ by $\restr{\Phi}{\fa} = \phi$, and
    $$\Phi(L_0) = \mathbf{s}, \quad \Phi(L_n) = \Phi(L_{-n}) = 0.$$
    One can compute that $\bms^\Phi = \widetilde{\fa}$. The condition that $\mathbf{s} + \frac{n^2 - 1}{24}c_L = 0$ guarantees that $\Phi$ is a well-defined Lie algebra homomorphism, since
    $$\Phi([L_n, L_{-n}]) = \Phi\left(2n L_0 + \frac{n^3 - n}{12} C_L\right) = 2n \mathbf{s} + \frac{n^3 - n}{12} c_L = 0.$$
    Therefore, $\Phi$ extends the quasi-Whittaker function $\phi$ to the entirety of its Whittaker annihilator. By \cite[Lemma 3.7]{ChengGaoLiuZhaoZhao}, it follows that $\Ind_{\widetilde{\fa}}^\bms \kk_{\Phi}$ is an irreducible $\bms$-module. It is easy to see that this module is isomorphic to the quotient of $\widetilde{V}(\mathbf{M}, \mathbf{s}, c_L, c_M)$ by its submodule generated by $L_n\genvecbms$ and $L_{-n}\genvecbms$, yielding the result.
\end{proof}

\begin{remark}
    We were informed by Blagoje Oblak that Theorem \ref{thm:bms_massive_modules_irreducibility} and Corollary \ref{cor:bms_massive_modules_irreducible_quotient} match the intuitive expectation stemming from his work \cite{BarnichOblak2} with Glenn Barnich, in which the authors study coadjoint orbits of the BMS$_3$ group and provide classical and quantum mechanical interpretations of these findings in the context of three-dimensional asymptotically flat gravity. It is reassuring that we are able to make this physical intuition mathematically precise.
\end{remark}

We now consider the reducibility of the $\widehat{\bcca}$-submodule of $\widetilde{V}(\mathbf{M}, \mathbf{s}, c_L, c_M)$ generated by the element $\genvecbms$. The condition governing its reducibility turns out to be exactly the same as the one in Theorem \ref{thm:bms_massive_modules_irreducibility}, and the proof is easier, since $U(\widehat{\bcca})\genvecbms$ is a quasi-Whittaker module over $\widehat{\bcca}$ on the nose.

\begin{proposition}\label{prop:massive module over b}
    Let $\mathbf{M}, \mathbf{s},c_L,c_M \in \kk$. The $\widehat{\bcca}$-module $U(\widehat{\bcca})\genvecbms$ is irreducible if and only if $\mathbf{M} + \frac{n^2 - 1}{24}c_M \ne 0$ for any positive integer $n$.
\end{proposition}
\begin{proof}
    Suppose first that $\mathbf{M} +\frac{n^2-1}{24}c_M\ne0$ for any positive integer $n.$ Again, setting $n=1$ lets us infer that $\mathbf{M} \neq 0$. Take $\fa \coloneqq \spn\{C_M, P_k:k\ge0\}$, which is an ideal of $\widehat{\bcca}$, and define a linear map $\phi \colon \fa \to \kk$ by
    $$\phi(P_i) = 2\delta_{i,0}\mathbf{M}, \quad \phi(C_M) = c_M.$$
    The Whittaker annihilator of $\phi$, denoted $\widehat{\bcca}^\phi$, is easily computed to be $\widehat{\bcca}^\phi = \fa$. It now follows immediately from Theorem \ref{thm:Whittaker annihilator} that the $\widehat{\bcca}$-module $U(\widehat{\bcca})\genvecbms \cong W_\phi = \Ind_\fa^{\widehat{\bcca}} \kk_\phi$ is irreducible.
    
    Conversely, assume that $\mathbf{M} +\frac{n^2 - 1}{24}c_M = 0$ for a positive integer $n$. In this case, $\widehat{\bcca}^\phi \supseteq \fa \oplus \kk \OO_n$. Then from Theorem \ref{thm:Whittaker annihilator}, we see that $\OO_{n}\genvecbms$ generates a nonzero proper submodule of $U(\widehat{\bcca})\genvecbms$. Thus, $U(\widehat{\bcca})\genvecbms$ is reducible in this case.
\end{proof}

Thus, the massive modules of $\bms$ seem to give rise to irreducible $\widehat{\bcca}$-modules more readily than the Verma modules.

We finish our analysis of massive modules by studying the full space $\widetilde{V}(\mathbf{M},\mathbf{s},c_L ,c_M)$ as a module over $\widehat{\bcca}$. In determining the structure $\widetilde{V}(\mathbf{M},\mathbf{s},c_L,c_M)$ as a $\widehat{\bcca}$-module, we assume that $\mathbf{M} + \frac{n^2 - 1}{24}c_M \ne 0$ for any positive integer $n$ for the rest of this section, meaning that $U(\widehat{\bcca})\genvecbms$ is irreducible. We first introduce two new symbols.

\begin{notation} \label{ntt:Ohat and Phat}
    For $n \in \ZZ$, define
    $$\widehat{P}_n \coloneqq M_n - M_{-n}, \quad \widehat{\OO}_n \coloneqq L_n + L_{-n},$$
    as elements of $\bms$. Note that $\widehat{P}_n$ and $\widehat{\OO}_n$ are not elements of $\widehat{\bcca}$.
\end{notation}

By \eqref{eq:bms bracket}, a simple computation shows that for any $n,m \in \NN$ we have
\begin{equation}\label{eq:bra-P-hO}
    \begin{aligned}
        &[P_n, \widehat{\OO}_m] = (n - m)\widehat{P}_{n + m} + (n + m)\widehat{P}_{n - m}, \\
        &[\widehat{P}_n, \widehat{\OO}_m] = (n - m){P}_{n + m} + (n + m){P}_{n - m} + \frac{1}{6} n (n^2 - 1)\delta_{n,m} C_M.
    \end{aligned}
\end{equation}
Notice that $\{\OO_n, \widehat{\OO}_m, P_m, \widehat{P}_n \mid n \in \ZZ_+, m \in \NN\}$ is another basis of $\bms$. Hence, by the PBW theorem, $\widetilde{V}(\mathbf{M}, \mathbf{s}, c_L, c_M)$ has a basis given by monomials of the form
\begin{equation} \label{eq:Vtilde alt PBW basis}
    \OO_{p_k}^{q_k} \dots \OO_{p_1}^{q_1} \widehat{\OO}^{s_\ell}_{r_\ell} \dots {\widehat{\OO}}^{s_1}_{r_1} \genvecbms,
\end{equation}
where $p_k > \dots > p_1 \geq 1$, $r_\ell > \dots > r_1 \geq 1$ and $q_1, \dots, q_k, s_1, \dots, s_\ell \in \NN$. 

We now prove that the element $\widehat{\OO}_n \genvecbms$ of $\widetilde{V}(\mathbf{M}, \mathbf{s}, c_L, c_M)$ generates an irreducible $\widehat{\bcca}$-submodule, for all $n \in \ZZ_+$.

\begin{proposition}\label{prop:P and Phat acting on Ohat_n}
    For any $n \in \ZZ_+$, the $\widehat \bcca$-module $U(\widehat{\bcca}) \widehat{\OO}_n \genvecbms$ is  isomorphic to $U(\widehat{\bcca}) \genvecbms$. Consequently, $U(\widehat{\bcca}) \widehat{\OO}_n \genvecbms$ is an irreducible $\widehat{\bcca}$-module.
\end{proposition}
\begin{proof}
    It follows from (\ref{eq:bra-P-hO}) that for any $r\in\NN$,
    \begin{equation*}
        P_r \widehat{\OO}_n \genvecbms = \delta_{0,r} 2M \genvecbms.
    \end{equation*}
    As above, the PBW theorem implies that $U(\widehat{\bcca}) \widehat{\OO}_n \genvecbms$ has a basis given by
    $$\OO_{p_k}^{q_k} \dots \OO_{p_1}^{q_1} \widehat{\OO}_n\genvecbms,$$
    where $p_k > \dots > p_1 \geq 1$, and $q_1, \dots, q_k\in \NN$. Therefore, we see that the map
    \begin{align*}
        U(\widehat{\bcca}) \widehat{\OO}_n \genvecbms &\to U(\widehat{\bcca}) \genvecbms \\
        \OO_{p_k}^{q_k} \dots \OO_{p_1}^{q_1} \widehat{\OO}_n \genvecbms &\mapsto \OO_{p_k}^{q_k} \dots \OO_{p_1}^{q_1} \genvecbms,
    \end{align*}
    where $p_k > \dots > p_1 \geq 1$ and $q_1, \dots, q_k \in \NN$, is an isomorphism of $\widehat{\bcca}$-modules.

    The final sentence follows from Proposition \ref{prop:massive module over b} upon recalling that we are assuming that $\mathbf{M} + \frac{n^2 - 1}{24}c_M \ne 0$.
\end{proof}

It is not surprising that the massive $\bms$-module $\widetilde{V}(\mathbf{M}, \mathbf{s}, c_L, c_M)$ does not remain irreducible when viewed as a $\widehat{\bcca}$-module. What is perhaps more surprising is that $\widetilde{V}(\mathbf{M}, \mathbf{s}, c_L, c_M)$ decomposes into a direct sum of two $\widehat{\bcca}$-submodules, as we will soon see. We define these submodules below.

\begin{defn} \label{def:Vhat and What}
    Let $\widehat{V}$ be the $\widehat{\bcca}$-submodule of $\widetilde{V}(\mathbf{M}, \mathbf{s}, c_L, c_M)$ generated by
    $$\left\{\widehat{\OO}^{s_\ell}_{r_\ell} \dots \widehat{\OO}^{s_1}_{r_1} \genvecbms \mid \ell \in \NN, r_\ell > \dots > r_1 \geq 1, s_1 + \dots + s_\ell = 0 \pmod 2\right\},$$
    and let $\widehat{W}$ be the $\widehat{\bcca}$-submodule of $\widetilde{V}(\mathbf{M}, \mathbf{s}, c_L, c_M)$ generated by
    $$\left\{\widehat{\OO}^{s_\ell}_{r_\ell} \dots \widehat{\OO}^{s_1}_{r_1} \genvecbms \mid \ell \in \NN, r_\ell > \dots > r_1 \geq 1, s_1 + \dots + s_\ell = 1 \pmod 2\right\}.$$
\end{defn}

The following result shows that the action of $P_r$ on $\widehat{\OO}^{s_\ell}_{r_\ell} \dots \widehat{\OO}^{s_1}_{r_1} \genvecbms$ preserves the parity of $\sum s_i$. 

\begin{proposition}\label{pre-formula}
    For any $r, s_1,\ldots, s_{\ell}\in\NN$ and $r_\ell > \dots > r_1 \geq 1$, the abelian part of the BMS$_3$ algebra has the following action on $\widehat{\OO}^{s_\ell}_{r_\ell} \dots \widehat{\OO}^{s_1}_{r_1} \genvecbms$:
    \begin{align*}
        P_r \cdot \widehat{\OO}^{s_\ell}_{r_\ell} \dots \widehat{\OO}^{s_1}_{r_1} \genvecbms &= \delta_{0,r} 2{\bf M} \widehat{\OO}^{s_\ell}_{r_\ell} \dots \widehat{\OO}^{s_1}_{r_1} \genvecbms && \pmod{\bigoplus_{\substack{0 \le u_i \le s_i, \sum u_i < \sum s_i, \\ \sum(s_i - u_i) = 0 \pmod{2}}} \kk  \widehat{\OO}^{u_\ell}_{r_\ell} \dots \widehat{\OO}^{u_1}_{r_1} \genvecbms}, \\
        \widehat{P}_r \cdot \widehat{\OO}^{s_\ell}_{r_\ell} \dots \widehat{\OO}^{s_1}_{r_1} \genvecbms &= 0 && \pmod{\bigoplus_{\substack{0 \le u_i \le s_i, \\ \sum(s_i - u_i) = 1 \pmod{2}}} \kk \widehat{\OO}^{u_\ell}_{r_\ell} \dots \widehat{\OO}^{u_1}_{r_1} \genvecbms}.
    \end{align*}
\end{proposition}
\begin{proof}
    We prove the result by induction on $d \coloneqq \sum_{i = 1}^{\ell} s_i$. It is not difficult to check that the two equalities hold for $d = 0, 1$. Now assume that $d>1$ and the two equalities hold for all $0 \le d'< d$. Consider  the first equality for the case $d = \sum_{i = 1}^{\ell} s_i$. We have
    \begin{align*}
        P_r \cdot \widehat{\OO}^{s_\ell}_{r_\ell} \dots \widehat{\OO}^{s_1}_{r_1} \genvecbms &= [P_r, \widehat{\OO}_{r_\ell}] \widehat{\OO}^{s_\ell-1}_{r_\ell} \dots \widehat{\OO}^{s_1}_{r_1} \genvecbms + \widehat{\OO}_{r_\ell} P_r\widehat{\OO}^{s_\ell - 1}_{r_\ell} \dots \widehat{\OO}^{s_1}_{r_1} \genvecbms \\
        &= \left((r - r_{\ell}) \widehat{P}_{r + r_{\ell}} + (r + r_{\ell}) \widehat{P}_{r - r_{\ell}} + \widehat{\OO}_{r_\ell} P_r\right) \widehat{\OO}^{s_\ell - 1}_{r_\ell} \dots \widehat{\OO}^{s_1}_{r_1} \genvecbms.
    \end{align*}
    The induction hypothesis implies that
    $\left((r - r_{\ell}) \widehat{P}_{r + r_{\ell}} + (r + r_{\ell}) \widehat{P}_{r - r_{\ell}}\right) \widehat{\OO}^{s_\ell - 1}_{r_\ell} \dots \widehat{\OO}^{s_1}_{r_1} \genvecbms$ belongs to 
    $$\bigoplus_{\substack{0 \le u_i \le s_i - \delta_{i,\ell}, \\
    \sum(s_i - u_i) - 1 = 1 \pmod{2}}} \kk \widehat{\OO}^{u_\ell}_{r_\ell} \dots \widehat{\OO}^{u_1}_{r_1} \genvecbms \subseteq \bigoplus_{\substack{0 \le u_i \le s_i, \sum u_i < \sum s_i, \\
    \sum(s_i - u_i) = 0 \pmod{2}}} \kk  \widehat{\OO}^{u_\ell}_{r_\ell} \dots \widehat{\OO}^{u_1}_{r_1} \genvecbms,$$
    and that $P_r\widehat{\OO}^{s_\ell-1}_{r_\ell}\dots \widehat{\OO}^{s_1}_{r_1}\genvecbms - \delta_{0,r} 2{\bf M} \widehat{\OO}^{s_\ell-1}_{r_\ell} \dots \widehat{\OO}^{s_1}_{r_1} \genvecbms$ belongs to
    $$\bigoplus_
    {\substack{0 \le u_i \le s_i-\delta_{i,\ell}, \sum u_i < \sum s_i - 1, \\
    \sum(s_i - u_i) - 1= 0 \pmod{2}}} \kk  \widehat{\OO}^{u_\ell}_{r_\ell} \dots \widehat{\OO}^{u_1}_{r_1} \genvecbms.$$
    The first equality follows. The second equality can be proved similarly.
\end{proof}

This naturally leads to the decomposition of $\widetilde{V}(\mathbf{M},\mathbf{s},c_L,c_M)$ given by the following corollary.

\begin{corollary}\label{cor:bms_massive_module_decomposable}
    As a $\widehat{\bcca}$-module, the massive module $\widetilde{V}(\mathbf{M},\mathbf{s},c_L,c_M)$ is decomposable. More precisely, we have
    $$\widetilde{V}(\mathbf{M},\mathbf{s},c_L,c_M) = \widehat{V} \oplus \widehat{W}.$$
\end{corollary}
\begin{proof}
    From Proposition \ref{pre-formula} and the PBW  theorem, we may deduce that $\widehat{V}$ and $\widehat{W}$ have bases given by
    \begin{gather*}
        \left\{\left.\OO_{p_k}^{q_k} \dots \OO_{p_1}^{q_1} \widehat{\OO}^{s_\ell}_{r_\ell} \dots {\widehat{\OO}}^{s_1}_{r_1} \genvecbms \right| p_k > \dots > p_1 \geq 1, r_\ell > \dots > r_1 \geq 1, \sum_{i = 1}^\ell s_i = 0 \pmod 2\right\}, \\
        \left\{\left.\OO_{p_k}^{q_k} \dots \OO_{p_1}^{q_1} \widehat{\OO}^{s_\ell}_{r_\ell} \dots {\widehat{\OO}}^{s_1}_{r_1} \genvecbms \right| p_k > \dots > p_1 \geq 1, r_\ell > \dots > r_1 \geq 1, \sum_{i = 1}^\ell s_i = 1 \pmod 2\right\},
    \end{gather*}
    respectively.
    Comparing this with the PBW basis of $\widetilde{V}(\mathbf{M}, \mathbf{s}, c_L, c_M)$ given by \eqref{eq:Vtilde alt PBW basis} leads to the result.
\end{proof}

It is interesting to note that the massive module $\widetilde{V}(\mathbf{M},\mathbf{s},c_L,c_M)$ with $\mathbf{M}+\frac{n^2-1}{24} \neq 0$ for all $n \in \ZZ_+$ is irreducible as a $\bms$-module but decomposable as a $\widehat{\bcca}$-module.

We finish our discussion of $\widetilde{V}(\mathbf{M},\mathbf{s},c_L,c_M)$ as a $\widehat{\bcca}$-module with the following conjecture.

\begin{conjecture}\label{conj:Ind}
    The $\widehat \bcca$-modules $\widehat{V}$ and $\widehat{W}$ are  indecomposable. 
\end{conjecture}

\subsection{The \texorpdfstring{$\OO$}{O}-module \texorpdfstring{$\Omega(\lambda,a)$}{Omega(lambda,a)}}\label{subsec:Omega module}

We finish this section by restricting one final class of Virasoro modules to $\OO$, first defined in \cite{LuZhao14}. For any pair  $(a,\lambda)\in\kk\times \kk^*$, one can define a Virasoro algebra module structure on $\kk[X]$ by
$$C \cdot f(X) = 0, \quad L_n \cdot f(X) = \lambda^n (X + na) f(X + n), \quad n \in \ZZ.$$
The module is irreducible if and only if $a \ne 0$, and $\Omega(\lambda,0)$ has a unique irreducible submodule which is isomorphic to $\Omega(\lambda,1)$ \cite[Section 4.3]{LuZhao10}. Restricting the above action to the subalgebra $\OO$ of $\Vir$, the $\OO$-module structure on $\kk[X]$ is given by 
\begin{equation}\label{O-act}
    \OO_n \cdot f(X) = \lambda^{n}(X + na)f(X + n) - \lambda^{-n} (X - na) f(X - n), \quad n \geq 1.
\end{equation}
We still denote the corresponding $\OO$-module by $\Omega(\lambda, a)$.

\begin{proposition}\label{Omega-simplicity}
    Suppose $\lambda\in\kk^*$ and $a\in\kk$. If $\lambda\ne \pm 1$ and $a\ne 0$, then the $\OO$-module $\Omega(\lambda,a)$ is irreducible.    
\end{proposition}

To prove Proposition \ref{Omega-simplicity}, we need the following lemma from \cite{TanZhao}.

\begin{lemma}[{\cite[Lemma 2]{TanZhao}}]\label{pre-lem}
    Let $\lambda_1, \lambda_2, \dots,\lambda_m \in \kk$, $s_1,s_2, \dots, s_m \in \ZZ_+$ with 
    $s_1 + \dots + s_m = s$. Define a sequence of functions on $\ZZ$ as follows: 
    $f_1(n) = \lambda_1^n, f_2(n) = n\lambda_1^n,\dots, f_{s_1}(n) = n^{s_1 - 1} \lambda_1^n, f_{s_1+1}(n) = \lambda_2^n, \dots, f_{s_1 + s_2}(n) = n^{s_2 - 1} \lambda_2^n, \dots, f_s(n) = n^{s_m - 1}\lambda_m^n$. Let $\mathcal{M}=(y_{pq})$ be the $s \times s$ matrix with $y_{pq} = f_{q}(p-1)$, $q = 1,2,\dots, s$, $p = u + 1, u + 2, \dots, u + s$ where $u \in \NN$. Then
    $$\det(\mathcal{M}) = \prod_{j = 1}^m(s_j - 1)!!\lambda_j^{s_j(s_j + 2u - 1)/2}\prod_{1 \le i < j \le m}
    (\lambda_j - \lambda_i)^{s_i s_j},$$
    where $s_{j}!! = s_j!\times (s_j-1)! \times \dots \times 2! \times 1!$ with $0!!=1$, for convenience.
\end{lemma}

In particular, we have the following.

\begin{lemma}\label{key-det}
    Suppose $\mu\in\kk$ and $\mu\ne 0, 1$. 
    For $r \in \ZZ_+$ and a fixed $u \in \NN$ let $k_j = u + j, j = 1,2, \dots, 2r$. Letting
    $$\tau(\mu,r,u) \coloneqq \begin{vmatrix}
        1&k_1&k_1^2&\cdots&k_1^{r-1}&\mu^{k_1}&\mu^{k_1}k_1&\mu^{k_1}k_1^2&\cdots&\mu^{k_1}k_1^{r-1}\\
        1&k_2&k_2^2&\cdots&k_2^{r-1}&\mu^{k_2}&\mu^{k_2}k_2&\mu^{k_2}k_2^2&\cdots&\mu^{k_2}k_2^{r-1}\\
        1&k_3&k_3^2&\cdots&k_3^{r-1}&\mu^{k_3}&\mu^{k_3}k_3&\mu^{k_3}k_3^2&\cdots&\mu^{k_3}k_3^{r-1}\\
        \vdots & \vdots & \vdots & \ddots & \vdots & \vdots & \vdots & \vdots & \ddots & \vdots \\
        1&k_{2r}&k_{2r}^2&\cdots&k_{2r}^{r-1}&\mu^{k_{2r}}&\mu^{k_{2r}}k_{2r}&\mu^{k_{2r}}k_{2r}^2&\cdots&\mu^{k_{2r}}k_{2r}^{r-1}\\
    \end{vmatrix},$$
    then $\tau(\mu,r,u)\ne 0$. 
\end{lemma}

\begin{proof}
    Taking $m=2$, $\lambda_1=1,\lambda_2=\mu$ and $s_1=s_2=r$ in Lemma \ref{pre-lem} we see
    $\tau(\mu,r,u)=\det\big((y_{pq})_{2r\times 2r }\big)$ with $y_{pq}=f_{q}(p-1)$ and $q=1, 2, \dots, 2r,  p=u+1,u+2,\dots, u+2r$. So 
    $$\tau(\mu,r,u)=\big((r-1)!!\big)^2\mu^{r(r+2u-1)/2}(\mu-1)^{r^2}\ne 0,
    $$as desired. This completes the proof.
\end{proof}

As the following result shows, the key difference between the cases where $\lambda = \pm 1$ or $a = 0$ and the other cases is whether $\Omega(\lambda, a)$ is generated by $1$.

\begin{lemma}\label{mod-generator}
    If $\lambda \ne \pm 1$ and $a\ne 0$, then the $\OO$-module $\Omega(\lambda,a)$ is generated by $1$.    
\end{lemma}
\begin{proof}
    It is sufficient to show that $X^n\in U(\OO) \cdot 1$ for any $n \in \NN.$ We prove this by induction on $n$, the base case $n = 0$ being trivial. For the induction step, let $n > 0$ and assume $X^{n'} \in U(\OO) \cdot 1$ for any $0 \le n'< n$. We aim to show that $X^n \in U(\OO) \cdot 1.$ Since $X^{n - 1} \in U(\OO) \cdot 1$, it follows from \eqref{O-act} that
    \begin{align*}
        \lambda^k \OO_k \cdot X^{n-1} &= \lambda^{2k}(X + ka)(X + k)^{n - 1} - (X - ka)(X - k)^{n - 1} \\
        &= (\lambda^{2k} - 1)X^n + g(X) \in U(\OO) \cdot 1
    \end{align*}
    for some $g(X) \in \kk[X]$. Note that $g(X) \in U(\OO) \cdot 1$ by the induction hypothesis, since $\deg(g(X)) \le n - 1$. Therefore,
    $$(\lambda^{2k} - 1) X^n = \lambda^k \OO_k \cdot X^{n - 1} - g(X) \in U(\OO) \cdot 1.$$
    Since $\lambda \ne \pm 1$, there always exists $k \geq 1$ such that $(\lambda^{2k} - 1) \ne 0$, which implies that $X^n \in U(\OO) \cdot 1$. Thus, $U(\OO) \cdot 1 = \Omega(\lambda,a)$. This completes the proof.
\end{proof}

Now we can prove Proposition \ref{Omega-simplicity}.

\begin{proof}[Proof of Proposition \ref{Omega-simplicity}.]
    To show that $\Omega(\lambda, a)$ is simple is enough to  show that any nonzero submodule of $\Omega(\lambda, a)$ is the  module $\Omega(\lambda, a)$ itself. Let $M$ be a nonzero submodule of $\Omega(\lambda, a)$. To show $M=\Omega(\lambda, a)$ is, by Lemma \ref{mod-generator}, to show $1\in M$. To the contrary, assume that $1\notin M$. Let $f(X)\in M$ be a nonzero element with minimal degree $r>0$. Denote
    $f(X) = \sum_{i=0}^r c_i X^i$ with $c_i \in \kk$ and $c_r \ne 0$. Then
     by \eqref{O-act} we have
    \begin{equation}\label{lambda-O-act}
        \begin{aligned}
            \lambda^k \OO_k \cdot f(X) &= \lambda^{2k}(X + ka) f(X + k) - (X - ka) f(X - k)\\
            &= \sum_{i = 0}^{r + 1} \varphi_{i}(X) \lambda^{2k} k^i + \sum_{i = 0}^{r + 1} \psi_i(X)k^i \in M\\ 
        \end{aligned}
    \end{equation}
    for all $k \in \ZZ_+$, where $\varphi_i(X), \psi_i(X) \in \kk[X]$ and $ \varphi_{r + 1}(X) = a c_r$ and $\psi_{r + 1}(X) = (-1)^r a c_r$ are nonzero constants. Let $\mu = \lambda^2$, $k_j = j$, where $1 \le j \le 2(r + 2)$ and $\mu\ne 0, 1$. By Lemma \ref{key-det} we know that 
    $\tau(\mu, r+2, 0)\ne 0$, which implies that
    $\varphi_i(X),\psi_i(X)\in M$ for all $0 \le i \le r + 1$. In particular, $\varphi_{r + 1}(X), \psi_{r + 1}(X)\in M$, contrary to the choice of $f(X)$. Therefore, $1 \in M$. So $M = \Omega(\lambda,a)$ and $\Omega(\lambda,a)$ is a simple $\OO$-module. This completes the proof.
\end{proof}

Finally, we consider the case $\lambda = \pm 1.$

\begin{proposition}\label{lambda-pm1}
    If $\lambda = \pm 1$ and $a\ne 0$, then the $\OO$-module $\Omega(\lambda,a)$ carries a filtration of submodules of $\Omega(\lambda,a)$ as follows:
    $$\Omega(\lambda,a)_n \subseteq \Omega(\lambda,a)_{n + 1}, \quad n \in \NN,$$
    where 
    $$\Omega(\lambda,a)_n \coloneqq \{f(X) \in \Omega(\lambda,a) \mid \deg f(X) \leq n\} = \left\{\sum_{i=0}^n c_i X^i \mid c_i \in \kk \right\}.$$
\end{proposition}
\begin{proof}
    We only need to show that each $\Omega(\lambda,a)_n$ is a submodule of $\Omega(\lambda,a)$, which follows  from the fact that for any $k \geq 1$ and any nonzero $f(X) \in \kk[X]$,
    $$\OO_k\cdot f(X) = \lambda^k \Big(X(f(X + k) - f(X - k)) + k a (f(X + k) + f(X - k))\Big)$$
    has degree $\le \deg(f(X))$ for $\deg(f(X+k)-f(X-k))=\deg(f(X))-1$ if $\deg(f(X))\ge 1$.   
\end{proof}
\begin{remark}
    Since $\Omega(\lambda,0)$ as a $\Vir$-module has a submodule $X\kk[X]$, which is isomorphic to $\Omega(\lambda, 1)$, we see that 
    $\Omega(\lambda,0)$ as a $\OO$-module also has a submodule $X\kk[X]\cong \Omega(\lambda,1)$. If $\lambda\ne \pm 1$, then, by Proposition \ref{Omega-simplicity}, $X\kk[X]$ is a simple $\OO$-module; and if $\lambda=\pm 1$, then, by Proposition \ref{lambda-pm1}, $\Omega(\lambda,0)$ as an $\OO$-module has a filtration of submodules:
    $$X\Omega(\lambda,0)_0 \subseteq X\Omega(\lambda,0)_1 \subseteq \dots \subseteq X\Omega(\lambda,0)_n \subseteq \dots \subseteq X\Omega(\lambda,0) \subseteq \Omega(\lambda,0).$$
\end{remark}

\section{Alternate basis of the BCCA} \label{sec:change of basis}

Our next goal is to construct a decreasing filtration on $\OO$ and $\bcca$, which will allow us to define and study Whittaker modules for these Lie algebras as per Lemma \ref{lem:filtration leads to Whittaker}. Unfortunately, the basis of $\bcca$ from Section \ref{sec:preliminaries} makes it difficult to construct the desired filtration. Therefore, in this section we present an alternate basis for $\OO$ and $\bcca$, from which the decreasing filtration becomes manifest.

\subsection{Subalgebras of the Witt algebra}\label{subsec:subalgebras of Witt}

Since $\bcca = \OO \ltimes \PP$, where $\OO$ is a subalgebra of $\W$, we now spend some time outlining some general facts about subalgebras of the Witt algebra. Note that $\W$ is naturally a $\kk[t,t^{-1}]$-module, and thus the most natural subalgebras of $\W$ are those which are also $\kk[t,t^{-1}]$-submodules.

\begin{defn}
    Given $f \in \kk[t,t^{-1}]$, define
    $$f\W \coloneqq \{hf\del \mid h \in \kk[t,t^{-1}]\},$$
    so that $f\W$ consists of all derivations of $\kk[t,t^{-1}]$ which are divisible by $f$. Given $f \in \kk[t]$, we define $f\WW_1$ in a similar way. The Lie subalgebras $f\W$ and $f\WW_1$ are known as \emph{submodule-subalgebras} of $\W$ and $\WW_1$.
\end{defn}

Note that the bracket in $f\W$ is given by
\begin{equation}\label{eq:fW bracket}
    [gf\del,hf\del] = (gh' - g'h)f^2\del
\end{equation}
for all $g,h \in \kk[t,t^{-1}]$. The same formula holds in $f\WW_1$, except we require that $f,g,h \in \kk[t]$.

\begin{example}\label{ex:fW subalgebra}
    Letting $f \coloneqq t^2 - 4$, we define $\mathfrak{f} \coloneqq f\mathbb{W}_1$. Then $\mathfrak{f}$ is spanned by
    $$f_n \coloneqq -t^{n - 1}f\del = -(t^{n + 1} - 4t^{n - 1})\del = L_n - 4L_{n - 2},$$
    where $n \geq 1$. The bracket in $\mathfrak{f}$ is
    \begin{align*}
        [f_n,f_m] &= [L_n - 4L_{n - 2},L_m - 4L_{m - 2}] \\
        &= (n - m)(L_{n + m} - 8L_{n + m - 2} + 16L_{n + m - 4}) \\
        &= (n - m)(f_{n + m} - 4f_{n + m - 2}).
    \end{align*}
\end{example}

In fact, all subalgebras of $\W$ and $\WW_1$ of finite codimension are ``essentially'' submodule-subalgebras, in a sense that is made precise in the next result.

\begin{proposition}[{\cite[Proposition 3.2.7]{PetukhovSierra}}]\label{prop:finite codimension subalgebras}
    Let $\g$ be a subalgebra of $\W$ of finite codimension. Then there exist $f \in \kk[t,t^{-1}] \nonzero$ and $n \in \NN$ such that
    $$f^n\W \subseteq \g \subseteq f\W.$$
\end{proposition}

The analogous result for finite codimension subalgebras of $\WW_1$ is also true with a nearly identical proof, although this is not explicitly proved in \cite{PetukhovSierra}.

However, the Lie algebra $\OO$ is a subalgebra of $\W$ of infinite codimension, so Proposition \ref{prop:finite codimension subalgebras} does not apply. Thankfully, there is an analogous classification of subalgebras of $\WW_1$ and $\W$ of infinite codimension in \cite[Theorem 2.8]{BellBuzaglo} and \cite[Theorem 9.3]{BuzagloIngalls}. To state the result, we must first introduce the subalgebras of $\W$ which play the role of the submodule-subalgebras in the infinite codimension setting. Following \cite{Buzaglo2, BellBuzaglo, BuzagloIngalls}, we introduce the notions of the \emph{set} and \emph{field of ratios} associated to a subalgebra of $\W$.

\begin{defn}\label{def:ratios}
    Let $\g$ be a Lie subalgebra of $\W$. The \emph{set of ratios} of $\g$ is
    $$R(\g) = \left.\left\{\frac{g}{h} \in \kk(t) \right| g\del,h\del \in \g, h \neq 0\right\}.$$
    The \emph{field of ratios} $F(\g)$ of $\g$ is the subfield of $\kk(t)$ generated by $R(\g)$.
\end{defn}

The next result gives one of the most important reasons why $F(\g)$ is a useful tool to study $\g$.

\begin{lemma}[{\cite[Lemma 3.4]{BuzagloIngalls}}]\label{lem:derivations of field of ratios}
    Let $\g$ be a subalgebra of $\W$. Then $\g \subseteq \Der(F(\g))$.
\end{lemma}
\begin{proof}
    Let $u\del,v\del,w\del \in \g$ with $v \neq 0$, where $u,v,w \in \kk[t,t^{-1}]$, so that $\frac{u}{v}, \frac{w}{v} \in R(\g)$. We claim that $w\del(\frac{u}{v}) \in F(\g)$, which means that elements of $\g$ map $R(\g)$ to $F(\g)$. Since $F(\g)$ is generated by $R(\g)$ as a field, it will then follow that all derivations in $\g$ preserve $F(\g)$.

    We have
    $$w\del\left(\frac{u}{v}\right) = w \cdot \frac{v\del(u) - u\del(v)}{v^2} = \frac{w}{v} \cdot \frac{v\del(u) - u\del(v)}{v}.$$
    Certainly, $\frac{w}{v} \in R(\g)$ by definition. Note that $[v\del,u\del] = (v\del(u) - u\del(v))\del \in \g$, so we also have $\frac{v\del(u) - u\del(v)}{v} \in R(\g)$. It follows that $w\del(\frac{u}{v}) \in F(\g)$, as claimed.
\end{proof}

As the next example illustrates, it is often useful to consider $\g$ as a subalgebra of $\Der(F(\g))$ rather than as a subalgebra of $\W$, at least when $\g$ has infinite codimension in $\W$.

\begin{example}
    Consider the subalgebra $\g = \spn\{L_{2n} \mid n \in \ZZ\}$. It is easy to see that $\g$ is a subalgebra of $\W$ of infinite codimension and that $F(\g) = R(\g) = \kk(t^2)$, since $\g = \kk[t^2,t^{-2}]t\del$. Thus, we see that $\g$ is the subalgebra $\Der(\kk[t^2,t^{-2}])$ of $\Der(\kk(t^2))$. Therefore, when we view $\g$ as a subalgebra of $\Der(F(\g))$, we forget about the basis elements $L_{2n + 1}$ for all $n \in \ZZ$, which are not relevant when studying Lie algebraic properties of $\g$.
\end{example}

Since $F(\g) \subseteq \kk(t)$, L\"uroth's theorem implies that $F(\g) = \kk(s)$ for some $s \in \kk(t)$, so Lemma \ref{lem:derivations of field of ratios} implies that $\g \subseteq \Der(\kk(s))$. In \cite[Theorem 5.3]{Buzaglo2}, it was shown that $s$ can be chosen to be a polynomial when $\g \subseteq \WW_1$. In fact, even more is true: $\g$ is contained in $\Der(\kk[s])$ with finite codimension.

\begin{theorem}[{\cite[Theorem 2.8]{BellBuzaglo}}]\label{thm:classification of subalgebras}
    Let $\g$ be an infinite-dimensional subalgebra of $\WW_1$. Then there exists $s \in \kk[t]$ such that $F(\g) = \kk(s)$, and $\g$ has finite codimension in $\Der(\kk[s])$. In particular, there exists $g(s) \in \kk[s]$ such that
    $$g(s)\kk[s]\del_s \subseteq \g \subseteq \kk[s]\del_s,$$
    where $\del_s \in \Der(\kk[s])$ is the unique derivation of $\kk[s]$ such that $\del_s(s) = 1$, so that $\Der(\kk[s]) = \kk[s]\del_s$.
\end{theorem}

Let $\g \subseteq \WW_1$, $s \in \kk[t]$, and $\del_s$ be as in Theorem \ref{thm:classification of subalgebras}. Certainly, the Lie algebra $\Der(\kk[s])$ is isomorphic to $\WW_1$, so Theorem \ref{thm:classification of subalgebras} says that $\g$ is isomorphic to a subalgebra of $\WW_1$ of finite codimension. The derivation $\del_s \in \Der(\kk[s])$ can be uniquely extended to a derivation of $\kk(t)$: we have $\del_s = \frac{1}{s'}\del$, where $s' \coloneqq \del(s)$ is the derivative of $s$ with respect to $t$. This allows us to embed $\Der(\kk[s])$ into $\Der(\kk(t))$ as follows:
$$\Der(\kk[s]) = \kk[s]\del_s = \frac{1}{s'}\kk[s]\del \subseteq \kk(t)\del = \Der(\kk(t)).$$
Of course, $\WW_1 = \kk[t]\del$ is also contained in $\Der(\kk(t))$. Therefore, Theorem \ref{thm:classification of subalgebras} implies that $\g \subseteq \WW_1 \cap \Der(\kk[s])$, where the intersection is taken in $\Der(\kk(t))$.

Although Theorem \ref{thm:classification of subalgebras} only considers subalgebras of $\WW_1$, the same ideas can be applied to subalgebras of $\W$ (see \cite[Section 9]{BuzagloIngalls}). To that end, we introduce the following definition.

\begin{notation}\label{ntt:L(s)}
    If $s \in \kk[t,t^{-1}]$, we define $L(s) \coloneqq \W \cap \Der(\kk[s])$.
\end{notation}

The Lie algebra $L(s)$ can be described as follows.

\begin{lemma}[{\cite[Lemma 2.7]{BellBuzaglo}}]\label{lem:L(s) subalgebra}
    Let $s \in \kk[t,t^{-1}]$, and let $g_s \in \kk[t]$ be the unique monic polynomial of minimal degree such that $s'g_s \in \kk[s]$. Letting $f_s \in \kk[t]$ such that $s'g_s = f_s(s)$, we have $L(s) = \kk[s]g_s\del$.
\end{lemma}

\begin{remark}
    Although the condition $s'g_s \in \kk[s]$ might seem unusual, it is precisely the necessary condition for $g_s\del$ to be a derivation of $\kk[s]$. Furthermore, it is not obvious that such a polynomial $g_s$ exists for every $s \in \kk[t,t^{-1}]$; this is proved in \cite[Proposition 4.13]{Buzaglo2}.
\end{remark}

We can also describe $L(s)$ as a subalgebra of $\Der(\kk[s]) = \kk[s]\del_s$ (in other words, viewing it as a Lie algebra of derivations of $\kk[s]$ instead of $\kk[t,t^{-1}]$). Since $\del = s'\del_s$, it follows that
$$L(s) = \kk[s]g_s\del = \kk[s]s'g_s\del_s = \kk[s]f_s(s)\del_s.$$
Therefore, $L(s)$ is the subalgebra $\kk[s]f_s(s)\del_s$ of $\Der(\kk[s])$. By changing variables $s \mapsto t$, we see that $L(s) \cong \kk[t]f_s\del = f_s\WW_1$, so $L(s)$ is isomorphic to a subalgebra of $\WW_1$ of finite codimension. We now give an example to illustrate the construction.

\begin{example}[{\cite[Example 4.16]{Buzaglo2}}]
    If we let $s = t^3 + 3t$ then $g_s = (t^2 + 1)(t^2 + 4)$ and $f_s = 3(t^2 + 4)$, since $s'g_s = 3(s^2 + 4) \in \kk[s]$. Therefore, $L(s) = \kk[t^3 + 3t](t^2 + 1)(t^2 + 4)\del$ is a Lie subalgebra of $\W$. As a subalgebra of $\Der(\kk[s])$, we have $L(s) = \kk[s](s^2 + 4)\del_s$. Changing variables $s \mapsto t$, we see that
    $$L(s) \cong \kk[t](t^2 + 4)\del = (t^2 + 4)\WW_1,$$
    so $L(s)$ is isomorphic to the subalgebra $(t^2 + 4)\WW_1$ of $\WW_1$ of codimension 2.
\end{example}

We summarise this discussion below.

\begin{proposition}[{\cite[Lemma 4.12, Proposition 4.13]{Buzaglo2}}]\label{prop:L(s)}
    Let $s \in \kk[t,t^{-1}]$, and let $f_s, g_s \in \kk[t]$ be as in Lemma \ref{lem:L(s) subalgebra}. Then
    $$L(s) = \kk[s]g_s\del = \kk[s]f_s(s)\del_s \cong f_s\WW_1,$$
    where the isomorphism comes from the change of variables $s \mapsto t$.
\end{proposition}

\subsection{Change of basis for \texorpdfstring{$\OO$}{O}}\label{subsec:new basis for O}

We now apply the ideas of Subsection \ref{subsec:subalgebras of Witt} to the Lie subalgebra $\OO$ of $\W$ to construct a new basis of $\OO$ which will allow us to define a descending filtration for the BCCA. In fact, we work in a slightly more general setting: $\OO$ is part of a family of Lie subalgebras of $\W$. It is not much more difficult to consider the entire family for the proofs in this section, so we work in this level of generality.

\begin{defn}
    For $\lambda \in \kk^*$ and $n \in \ZZ$, define $\Olambda_n \coloneqq L_n - \lambda^n L_{-n}$. Notice that $\Olambda_0 = 0$ and $\Olambda_{-n} = -\lambda^{-n}\Olambda_n$ for $n \geq 1$. Define $\OO(\lambda) \coloneqq \spn\{\Olambda_n \mid n \geq 1\}$.
\end{defn}

Of course $\OO(1) = \OO$. On the other hand, it is not immediately clear that $\OO(\lambda)$ is closed under the Lie bracket of $\W$ for $\lambda \neq 1$. We prove this next.

\begin{lemma}\label{lem:O lambda bracket}
    Let $\lambda \in \kk^*$. Then $\OO(\lambda)$ is a Lie algebra with the following Lie bracket:
    $$[\Olambda_n,\Olambda_m] = (n - m)\Olambda_{n + m} - \lambda^m (n + m) \Olambda_{n - m}$$
    for all $n,m \geq 1$. In particular, $\OO(\lambda)$ is generated by $\Olambda_1$ and $\Olambda_2$.
\end{lemma}
\begin{proof}
    We compute the bracket $[\Olambda_n,\Olambda_m]$ as follows:
    \begin{align*}
        [\Olambda_n,\Olambda_m] &= [L_n - \lambda^n L_{-n}, L_m - \lambda^m L_{-m}] \\
        &= (n - m)L_{n + m} - \lambda^m (n + m) L_{n - m} + \lambda^n (n + m) L_{m - n} - \lambda^{n + m}(n - m)L_{-(n + m)} \\
        &= (n - m)\Olambda_{n + m} - \lambda^m (n + m)\Olambda_{n - m}.
    \end{align*}
    The final sentence follows by an easy induction.
\end{proof}

\begin{remark}
    One might ask if we can get different subalgebras of the Witt algebra by combining $L_n$ and $L_{-n}$ differently to the way it is done in $\OO(\lambda)$. In other words, can we generate different Lie subalgebras of $\W$ with elements $L_1 - \lambda L_{-1}$ and $L_2 - \mu L_{-2}$, where $\lambda, \mu \in \kk^*$? This question is answered in Appendix \ref{appendix}: if $\mu \neq \lambda^2$, then these two elements generate the entire Witt algebra. As a result, the subalgebras $\OO(\lambda)$ are the unique family of proper subalgebras of $\W$ generated by elements of the form $L_1 - \lambda L_{-1}$ and $L_2 - \mu L_{-2}$ for some $\lambda, \mu \in \kk^*$.
\end{remark}

We now aim to describe an alternative basis $u_n^{(\lambda)}$ for $\OO(\lambda)$. To achieve this, we analyse the structure of $\OO(\lambda)$ following the methods of Subsection \ref{subsec:subalgebras of Witt}. In terms of polynomials, we have $\Olambda_1 = -(t^2 - \lambda)\del$, and
$$\Olambda_2 = -(t^3 - \lambda^2 t^{-1})\del = -t^{-1}(t^4 - \lambda^2)\del = -t^{-1}(t^2 - \lambda)(t^2 + \lambda)\del.$$
Dividing $\Olambda_2$ by $\Olambda_1$, it follows that
$$\frac{t^{-1}(t^2 - \lambda)(t^2 + \lambda)}{t^2 - \lambda} = t + \lambda t^{-1} \in R(\OO(\lambda)).$$
We therefore define $s_\lambda \coloneqq t + \lambda t^{-1}$, so that $\kk(s_\lambda) \subseteq F(\OO(\lambda))$. In fact, as we show next, we have $R(\OO(\lambda)) = F(\OO(\lambda)) = \kk(s_\lambda)$ and $\OO(\lambda) = L(s_\lambda)$, where this notation is defined in Definition \ref{def:ratios} and Notation \ref{ntt:L(s)}.

\begin{theorem}\label{thm:O lambda = L(s)}
    Let $\lambda \in \kk^*$, and define $\f_\lambda \coloneqq f_{s_\lambda} \WW_1$, where $f_{s_\lambda} = t^2 - 4\lambda$. Then
    $$\OO(\lambda) = L(s_\lambda) = \kk[t + \lambda t^{-1}](t^2 - \lambda)\del = \kk[s_\lambda](s_\lambda^2 - 4\lambda)\del_{s_\lambda} \cong \f_\lambda,$$
    where the isomorphism comes from the change of variables $s_\lambda \mapsto t$. Consequently, $\OO(\lambda)$ has another basis $\{\ulambda_n \mid n \geq 1\}$ defined as follows: for $n \geq 1$, let
    $$\ulambda_n \coloneqq -s_\lambda^{n - 1}(s_\lambda^2 - 4\lambda)\del_{s_\lambda} = -(t + \lambda t^{-1})^{n - 1}(t^2 - \lambda)\del.$$
    The Lie bracket of $\OO(\lambda)$ with this basis is given by
    $$[\ulambda_n,\ulambda_m] = (n - m)(\ulambda_{n + m} - 4\lambda \ulambda_{n + m - 2})$$
    for all $n, m \geq 1$.
\end{theorem}
\begin{proof}
    We adopt the notation $g_{s_\lambda}$ and $f_{s_\lambda}$ from Lemma \ref{lem:L(s) subalgebra}. It is not difficult to see that $g_{s_\lambda} = t^2 - \lambda$, since
    $$s_\lambda'(t^2 - \lambda) = (1 - \lambda t^{-2})(t^2 - \lambda) = t^2 - 2\lambda + \lambda^2 t^{-2} = (t + \lambda t^{-1})^2 - 4\lambda = s_\lambda^2 - 4\lambda \in \kk[s_\lambda].$$
    Consequently, $f_{s_\lambda} = t^2 - 4\lambda$. It follows from Proposition \ref{prop:L(s)} that
    $$L(s_\lambda) = \kk[s_\lambda](s_\lambda^2 - 4\lambda)\del_{s_\lambda} = \kk[s_\lambda]g_{s_\lambda}\del = \kk[t + \lambda t^{-1}](t^2 - \lambda)\del.$$
    Note that $\Olambda_1 = -(t^2 - \lambda)\del \in L(s_\lambda)$ and
    $$\Olambda_2 = -(t^3 - \lambda^2 t^{-1})\del = -(t + \lambda t^{-1})(t^2 - \lambda)\del \in L(s_\lambda).$$
    Since $\Olambda_1$ and $\Olambda_2$ generate $\OO(\lambda)$ by Lemma \ref{lem:O lambda bracket}, this implies that $\OO(\lambda) \subseteq L(s_\lambda)$.

    For the other inclusion $L(s_\lambda) \subseteq \OO(\lambda)$, we will show that $L(s_\lambda)$ is also generated by $\Olambda_1$ and $\Olambda_2$. To that end, noting that $\{\ulambda_n \mid n \geq 1\}$ is a basis for $L(s_\lambda)$, we compute the bracket $[\ulambda_n,\ulambda_m]$. Letting $\f_\lambda \coloneqq f_{s_\lambda}\WW_1 = (t^2 - 4\lambda)\WW_1$, the change of variables $s_\lambda \mapsto t$ yields an isomorphism $L(s_\lambda) \cong \f_\lambda$, as in Proposition \ref{prop:L(s)}. Noting that
    $$\ulambda_n = -(t + \lambda t^{-1})^{n - 1}(t^2 - \lambda)\del = -s_\lambda^{n - 1}(s_\lambda^2 - 4\lambda)\del_{s_\lambda},$$
    we see that
    $$\ulambda_n \mapsto \flambda_n \coloneqq -t^{n - 1} f_\lambda \del = L_n - 4\lambda L_{n - 2}$$
    under the change of variables $s_\lambda \mapsto t$. We have
    \begin{align*}
        [\flambda_n,\flambda_m] &= [L_n - 4\lambda L_{n - 2},L_m - 4\lambda L_{m - 2}] \\
        &= (n - m)(L_{n + m} - 8\lambda L_{n + m - 2} + 16\lambda^2 L_{n + m - 4}) \\
        &= (n - m)(\flambda_{n + m} - 4 \lambda \flambda_{n + m -2}),
    \end{align*}
    and therefore,
    $$[\ulambda_n,\ulambda_m] = (n - m)(\ulambda_{n + m} - 4\lambda \ulambda_{n + m - 2}).$$
    An easy induction shows that $\Olambda_1 = \ulambda_1$ and $\Olambda_2 = \ulambda_2$ generate $L(s_\lambda)$, which concludes the proof.
\end{proof}

\begin{remark}
    Theorem \ref{thm:O lambda = L(s)} implies that $R(\OO(\lambda)) = F(\OO(\lambda)) = \kk(s_\lambda)$, although a priori it is not obvious that $R(\OO(\lambda))$ is a field.
\end{remark}

\begin{example}
    The Lie algebra $\OO(-1)$ is spanned by $L_n + (-1)^{n + 1}L_{-n}$ and is isomorphic to $(t^2 + 4)\WW_1$.
\end{example}

Using Theorem \ref{thm:O lambda = L(s)}, it is now easy to see that $\OO(\lambda)$ is never simple.

\begin{corollary}\label{cor:O is not simple}
    For all $\lambda \in \kk^*$, the Lie algebra $\OO(\lambda)$ is not perfect, and therefore is not simple.
\end{corollary}
\begin{proof}
    The Lie algebra $\f_\lambda$ is not perfect: this is because the derived subalgebra of $\flambda$ is
    $$[\f_\lambda,\f_\lambda] = [f_{s_\lambda}\mathbb{W}_1,f_{s_\lambda}\mathbb{W}_1] = f_{s_\lambda}^2\mathbb{W}_1,$$
    by \eqref{eq:fW bracket}. The result now follows by the isomorphism $\OO(\lambda) \cong \f_\lambda$ from Theorem \ref{thm:O lambda = L(s)}.
\end{proof}

\begin{remark}
    Corollary \ref{cor:O is not simple} implies that $\bcca$ is not perfect, meaning it does not have a universal central extension \cite[Theorem 7.9.2]{Weibel}. Nonetheless, it would be interesting to compute the \emph{centrally closed} version of $\bcca$, which would be the central extension of $\bcca$ with trivial second cohomology group (i.e., does not admit any more non-trivial central extensions). From a physics perspective, this would tell us the possible ``central charges'' that a quantum field theory with $\bcca$ symmetry could admit. It remains to be seen whether this centrally closed Lie algebra is indeed $\widehat{\bcca}$.
    
    We highlight that the computation of central extensions of $\bcca$ is mathematically nontrivial, mainly due to the fact that $\bcca$ is not graded. There are many results that greatly simplify the computation of central extensions of graded Lie algebras, such as \cite[Theorem 1.5.2]{Fuchs}. On the other hand, the centreless BCCA is a semi-direct sum Lie algebra, which can be studied using techniques developed in \cite[Proposition 1]{GaoLiuPei} and \cite[Proposition 3.1]{BuzagloVishwa}. While these results would certainly help when computing the central extensions of $\bcca$, we expect this to be a challenging task. See \cite{SchlichenmaierCrelle} for an example of computing central extensions of non-graded Lie algebras.
\end{remark}

As we aim to show next, the Lie algebras $\OO(\lambda)$ are often isomorphic to each other for different values of $\lambda \in \kk^*$. Recalling that $\OO(\lambda) \cong \f_\lambda = (t^2 - 4\lambda)\WW_1$, it suffices to consider when two subalgebras of $\WW_1$ of finite codimension are isomorphic. In fact, this can only happen if the isomorphism is induced by an automorphism of $\WW_1$ \cite{Buzaglo}. The automorphisms of $\WW_1$ are well-known \cite{Rudakov2}.

\begin{theorem}[{\cite[Theorem 4.1 and Corollary 4.12]{Buzaglo}}]\label{thm:automorphisms Witt}
    Let $\g$ and $\h$ be subalgebras of $\WW_1$ of finite codimension and suppose there is an isomorphism $\varphi: \g \to \h$. Then $\varphi$ extends to an automorphism of $\WW_1$. In particular, $g\WW_1 \cong h\WW_1$ if and only if there exist $\alpha, \gamma \in \kk^*$ and $\beta \in \kk$ such that $g(z) = \gamma h(t)$, where $z = \alpha t + \beta$ and $g,h \in \kk[t]$.
\end{theorem}

We are now ready to characterise when $\OO(\lambda) \cong \OO(\mu)$.

\begin{proposition}\label{prop:g lambda isomorphism}
    Let $\lambda, \mu \in \kk^*$. Then $\OO(\lambda) \cong \OO(\mu)$ if and only if $\sqrt{\frac{\lambda}{\mu}} \in \kk$.
\end{proposition}
\begin{proof}
    By Theorem \ref{thm:O lambda = L(s)}, we have $\OO(\lambda) \cong \f_\lambda = (t^2 - 4\lambda)\WW_1$ and $\OO(\mu) \cong \f_\mu = (t^2 - 4\mu)\WW_1$. Therefore, we will prove the result by showing that $\f_\lambda \cong \f_\mu$ if and only if $\sqrt{\frac{\mu}{\lambda}} \in \kk$.

    First, assume that $\alpha \coloneqq \sqrt{\frac{\lambda}{\mu}} \in \kk$. Letting $z \coloneqq \alpha t$, we have
    $$f_{s_\lambda}(z) = z^2 - 4\lambda = \alpha^2 t^2 - 4\lambda = \frac{\lambda}{\mu} t^2 - 4\lambda = \frac{\lambda}{\mu}(t^2 - 4\mu) = \frac{\lambda}{\mu} f_{s_\mu}(t).$$
    By Theorem \ref{thm:automorphisms Witt}, the change of variables $t \mapsto \alpha t$ induces an isomorphism $\f_\lambda \cong \f_\mu$.

    Conversely, assume that $\f_\lambda \cong \f_\mu$. By Theorem \ref{thm:automorphisms Witt}, there exist $\alpha, \gamma \in \kk^*$ and $\beta \in \kk$ such that $f_{s_\lambda}(z) = \gamma f_{s_\mu}(t)$, where $z = \alpha t + \beta$. Therefore,
    \begin{align*}
        t^2 - 4\mu &= f_{s_\mu}(t) = \frac{1}{\gamma} f_{s_\lambda}(z) = \frac{1}{\gamma}(z^2 - 4\lambda) = \frac{1}{\gamma}\Big((\alpha t + \beta)^2 - 4\lambda\Big) \\
        &= \frac{1}{\gamma}(\alpha^2 t^2 + 2 \alpha \beta t + \beta^2 - 4\lambda).
    \end{align*}
    It follows that $\gamma = \alpha^2$, $\beta = 0$, and $\lambda = \gamma \mu = \alpha^2 \mu$. Therefore, $\alpha = \sqrt{\frac{\lambda}{\mu}}$ must be an element of $\kk$.
\end{proof}

With $\OO(1) = \OO$, the next result follows immediately from Proposition \ref{prop:g lambda isomorphism}.

\begin{corollary}\label{cor:two_O_subalgebras}
    If $\kk$ is algebraically closed (for example, if $\kk = \CC$), then $\OO(\lambda) \cong \OO$ for all $\lambda \in \kk$. If $\kk = \RR$, then
    $$\OO(\lambda) \cong \begin{cases}
        \OO, &\text{if } \lambda > 0, \\
        \OO(-1), &\text{if } \lambda < 0,
    \end{cases}$$
    but $\OO(-1) \not\cong \OO$. \qed
\end{corollary}

\begin{remark}\label{rem:different boundaries}
    Corollary \ref{cor:two_O_subalgebras} is intriguing in the context in which the BCCA was discovered \cite{BCCFGP2024}. Specifically, we have shown that the subalgebras $\OO = \spn\{L_n - L_{-n} \mid n \geq 1\}$ and $\OO(-1) = \spn\{L_n + (-1)^{n + 1} L_{-n}\}$ are isomorphic as complex Lie algebras but not as real ones. In the realisation of the BMS$_3$ algebra given by \cite[Equation (4)]{BCCFGP2024} in terms of vector fields which preserve the Carrollian structure on cylinder (see \cite[Section 2.1]{Chen:2024voz}), we have
    $$L_n = e^{in\sigma}(\del_\sigma + in \tau \del_\tau),$$
    where $\sigma$ and $\tau$ are the coordinates of a cylinder. The Lie algebra $\OO$ was the constructed as the subalgebra of the Witt algebra spanned by $\{L_n \mid n \in \ZZ\}$ which leaves the choice of boundary conditions $\sigma = 0, \pi$ invariant \cite[Equation (5)]{BCCFGP2024}
    $$\OO_n = L_n - L_{-n} = 2i \sin(n\sigma)\del_\sigma + 2in \tau \cos(n\sigma)\del_\tau.$$
    At $\sigma = 0,\pi$, the vector fields $\OO_n = \pm 2in \tau \del_\tau$ do not contain any nonzero $\del_\sigma$ terms and thus satisfy the boundary conditions.
    
    In terms of vector fields on the cylinder, the algebra $\OO(-1)$ preserves a different choice of boundaries, in this case at $\sigma = \frac{\pi}{2}, \frac{3\pi}{2}$. This can be seen as follows:
    \begin{align*}
        \OO_{2n}^{(-1)} &= L_{2n} - L_{-2n} = 2i \sin(2n\sigma)\del_\sigma + 4in \tau \cos(2n\sigma) \del_\tau, \\
        \OO_{2n + 1}^{(-1)} &= L_{2n + 1} + L_{-2n - 1} = 2\cos((2n + 1)\sigma)\del_\sigma - (4n + 2) \tau \sin((2n + 1)\sigma)\del_\tau.
    \end{align*}
    Therefore, at $\sigma = \frac{\pi}{2}, \frac{3\pi}{2}$ we have $\OO_{2n}^{(-1)} = \pm 4in \tau \del_\tau$ and $\OO_{2n + 1}^{(-1)} = \pm (4n + 2) \tau \del_\tau$, which satisfy the boundary conditions.
    
    When working over $\CC$, this can be undone by a simple change of coordinates to $\sigma' = \sigma - \frac{\pi}{2}$, with the resulting boundary conditions in terms of $\sigma'$ being equivalent to the original one. Consequently, one would expect to have recovered $\OO$. Indeed, this change of coordinates is just a rescaling of $\OO$ by $i$, which reiterates the fact that the authors of the original paper were working over $\CC$. However, Corollary \ref{cor:two_O_subalgebras} implies that the boundary conditions at $\sigma = \frac{\pi}{2}, \frac{3\pi}{2}$ result in a non-isomorphic Lie algebra when working over $\RR$.
\end{remark}

We now focus on the case of $\OO = \OO(1)$.

\begin{notation}
    We make the following definitions:
    \begin{align*}
        s &\coloneqq s_1 = t + t^{-1}, \\
        u_n &\coloneqq u_n^{(1)} = -(t + t^{-1})^{n - 1}(t^2 - \lambda)\del = -s^{n - 1}(s^2 - 4)\del_s, \\
        f &\coloneqq f_{s_1} = t^2 - 4, \\
        f_n &\coloneqq f_n^{(1)} = -t^{n - 1}f\del = L_n - 4L_{n - 2}, \\
        \f &\coloneqq \f_1 = f\WW_1 = (t^2 - 4)\WW_1,
    \end{align*}
    where $n \in \ZZ_+$.
\end{notation}

Of course, $\{\OO_n \mid n \geq 1\}$ and $\{u_n \mid n \geq 1\}$ are two bases for the same Lie algebra $\OO$, so it is natural to ask how they are related. The following result answers this.

\begin{lemma}\label{lem:u in terms of O}
    We have
    $$u_n = \OO_n + \sum_{k = 1}^{\floor{\frac{n}{2}}} \binom{n}{k}\left(1 - \frac{2k}{n}\right) \OO_{n - 2k}$$
    for all $n \geq 1$.
\end{lemma}
\begin{proof}
    First, note that $u_1 = -(t^2 - 1)\del = L_1 - L_{-1} = \OO_1$. Furthermore,
    \begin{equation}\label{eq:difference of two Os}
        \begin{aligned}
            -(t^n + t^{-n})(t^2 - 1)\del &= -\Big(t^{n + 2} - t^n + t^{-n + 2} - t^{-n}\Big)\del \\
            &= L_{n + 1} - L_{n - 1} + L_{-n + 1} - L_{-n - 1} \\
            &= (L_{n + 1} - L_{-(n + 1)}) - (L_{n - 1} - L_{-(n - 1)}) \\
            &= \OO_{n + 1} - \OO_{n - 1}.
        \end{aligned}
    \end{equation}
    We now compute for $n \geq 2$:
    $$u_{n} = -(t + t^{-1})^{n - 1}(t^2 - 1)\del = -\left(\sum_{k = 0}^{n - 1} \binom{n - 1}{k}t^{n - 2k-1}\right)(t^2 - 1)\del.$$
    Pairing up $t^m$ with $t^{-m}$ and using \eqref{eq:difference of two Os}, we get
    $$u_n = \OO_n + \sum_{k = 1}^{\floor{\frac{n}{2}}} \left(\binom{n - 1}{k} - \binom{n - 1}{k - 1}\right)\OO_{n - 2k}.$$
    Applying the identity
    $$\binom{n - 1}{k} - \binom{n - 1}{k - 1} = \binom{n}{k}\left(1 - \frac{2k}{n}\right),$$
    the result now follows.
\end{proof}

\subsection{Change of basis for \texorpdfstring{$\PP_b$}{Pb}}

In Theorem \ref{thm:O lambda = L(s)}, it was shown that the change of variables $s \mapsto t$ yields the isomorphism $\OO \cong \f = \kk[t](t^2 - 4)\del$. Therefore, it is natural to ask if this change of variables gives a similar description of $\PP$. In this section, we show that this is indeed the case: the change of variables will allow us to define a new basis for $\PP$. We will work in a slightly more general setting: we consider the change of variables on the entire family of $\OO$-modules $\PP_b$ for arbitrary $b \in \kk$ (see Definition \ref{def:O_subalgebra}). Just like for $\OO$, it will be useful to have a generating set for the $\OO$-module $\PP_b$.

\begin{lemma}\label{lem:generators of Pb}
    For all $b \in \kk^*$, the $\OO$-module $\PP_b$ is generated by $\Pb_0$. For $b = 0$, the $\OO$-module $\PP_0$ is generated by $P_0^{(0)}$ and $P_1^{(0)}$.
\end{lemma}
\begin{proof}
    If $b \neq 0$, this follows immediately from \eqref{eq:O action on Pb}, since
    $$\OO_n \cdot \Pb_0 = -2bn \Pb_n$$
    for all $n \geq 1$. If $b = 0$, then \eqref{eq:O action on Pb} gives
    $$\OO_1 \cdot P_n^{(0)} = n (P_{n + 1}^{(0)} - P_{n - 1}^{(0)})$$
    for all $n \in \NN$. The result now follows by an easy induction.
\end{proof}

To describe how the change of variables $s \mapsto t$ affects $\PP_b$, we first introduce some notation. As above, we make the identification $\del_s = \frac{1}{s'}\del$, so that $\del_s(s) = 1$. Letting $r \coloneqq t - t^{-1}$, we have
\begin{equation}\label{eq:del_s = t/r del}
    \del_s = \frac{1}{s'}\del = \frac{1}{1 - t^{-2}}\del = \frac{t}{t - t^{-1}}\del = \frac{t}{r}\del.
\end{equation}
Note that
\begin{equation}\label{eq:r^2 = s^2 - 4}
    r^2 = (t - t^{-1})^2 = t^2 - 2 + t^{-2} = (t + t^{-1})^2 - 4 = s^2 - 4.
\end{equation}
In other words, $r$ is a square root of $s^2 - 4$.

Recall that $\PP_b$ is the $\OO$-submodule of $I(0,b)$ spanned by $\Pb_n = I_n + I_{-n}$. As mentioned in Definition \ref{def:tensor density}, we view $I(0,b)$ as $I(0,b) = t^{-b}\kk[t,t^{-1}] dt^b$ with $I_n = t^{n - b} \ dt^b$. With this perspective, we have
\begin{equation}\label{eq:Pn in terms of polynomials}
    \Pb_n = I_n + I_{-n} = -(t^n + t^{-n})t^{-b} \ dt^b.
\end{equation}
Now, we want to express $\PP_b$ in terms of polynomials in $s$, just like we did with $\OO$. By definition of the K\"ahler differential, we have $ds = s' \ dt$. It follows that $ds^b = (s' \ dt)^b = (s')^b \ dt^b$, and therefore
\begin{equation}\label{eq:dt^b in terms of ds^b}
    dt^b = (s')^{-b} \ ds^b = (1 - t^{-2})^{-b} \ ds^b = \frac{t^b}{r^b} \ ds^b,
\end{equation}
so $t^{-b} \ dt^b = r^{-b} \ ds^b$. The next result shows what happens to the generators of $\PP_b$ when written in terms of $ds^b$.

\begin{lemma}\label{lem:P0 and P1 in terms of s}
    For $b \in \kk$, we have $\Pb_0 = -2r^{-b} \ ds^b$ and $\Pb_1 = -sr^{-b} \ ds^b$.
\end{lemma}
\begin{proof}
    Follows immediately from \eqref{eq:Pn in terms of polynomials} and \eqref{eq:dt^b in terms of ds^b} upon recalling that $s = t + t^{-1}$.
\end{proof}

Inspired by Lemma \ref{lem:P0 and P1 in terms of s}, we now introduce the following notation.

\begin{notation}
    For $b \in \kk$ and $n \in \NN$, we define $v_n^{(b)} \coloneqq -s^n r^{-b} \ ds^b = -(t + t^{-1})^n t^{-b} \ dt^b \in I(0,b)$.
\end{notation}

By Lemma \ref{lem:P0 and P1 in terms of s}, we have $v_0^{(b)} = \frac{1}{2}\Pb_0$ and $v_1^{(b)} = \Pb_1$. In fact, as we show next, $\{v_n^{(b)} \mid n \in \NN\}$ is another basis for $\PP_b$.

\begin{theorem}\label{thm:new basis for P}
    For all $b \in \kk$, we have
    $$\PP_b = \spn\{v_n^{(b)} \mid n \in \NN\} = \kk[s]r^{-b} \ ds^b = \kk[t + t^{-1}] t^{-b} \ dt^b.$$
    The $\OO$-action on $\PP_b$ with this basis is given by
    $$u_n \cdot v_m^{(b)} = -(bn + m)v_{n + m}^{(b)} + 4(b(n - 1) + m)v_{n + m - 2}^{(b)}$$
    for all $n \geq 1$ and $m \geq 0$. In basis-free notation, the $\OO$-action is
    \begin{equation}\label{eq:basis-free O action on P}
        (s^2 - 4)g\del_s \cdot (hr^{-b} \ ds^b) = \Big((g\del_s(h) + b \del_s(g)h)(s^2 - 4) + b ghs\Big)r^{-b} \ ds^b
    \end{equation}
    for all $g,h \in \kk[s]$.
\end{theorem}
\begin{proof}
    We start by computing $u_n \cdot v_m^{(b)}$. To that end, note that
    $$\del_s(r) = \frac{r'}{s'} = \frac{(t - t^{-1})'}{(t + t^{-1})'} = \frac{1 + t^{-2}}{1 - t^{-2}} = \frac{t + t^{-1}}{t - t^{-1}} = \frac{s}{r},$$
    and thus
    $$\del_s(r^{-b}) = -b r^{-b - 1} \del_s(r) = \frac{-b s}{r^{b + 2}} = \frac{-bsr^{-b}}{s^2 - 4},$$
    where we used that $r^2 = s^2 - 4$ from \eqref{eq:r^2 = s^2 - 4} in the last equality. Therefore, we have
    \begin{align*}
        (s^2 - 4)g\del_s &\cdot (hr^{-b} \ ds^b) = \Big((s^2 - 4) g \del_s(hr^{-b}) + b\del_s((s^2 - 4)g)hr^{-b}\Big)ds^b \\
        &= \Big((s^2 - 4) g \big(\del_s(h)r^{-b} + h\del_s(r^{-b})\big) + b\big(\del_s(s^2 - 4)g + (s^2 - 4)\del_s(g)\big)hr^{-b}\Big)ds^b \\
        &= \left((s^2 - 4) g \left(\del_s(h)r^{-b} - \frac{bhsr^{-b}}{s^2 - 4}\right) + b\big(2sg + (s^2 - 4)\del_s(g)\big)hr^{-b}\right)ds^b \\
        &= \Big(\big(g \del_s(h) + b\del_s(g)h\big)(s^2 - 4) + b ghs\Big)r^{-b} \ ds^b,
    \end{align*}
    as required. It follows that
    \begin{align*}
        u_n \cdot v_m^{(b)} &= [s^{n - 1}(s^2 - 4)\del_s,s^m r^{-b} \ ds^b] \\
        &= \Big(\big(s^{n - 1}\del_s(s^m) + b\del_s(s^{n - 1})s^m\big)(s^2 - 4) + b s^{n + m}\Big)r^{-b} \ ds^b \\
        &= \Big((m + b(n - 1)) s^{n + m - 2} (s^2 - 4) + b s^{n + m}\Big)r^{-b} \ ds^b \\
        &= \Big((m + b n)s^{n + m} - 4(b(n - 1) - m)s^{n + m - 2}\Big)r^{-b} \ ds^b \\
        &= (n - m)v_{n + m}^{(b)} + 4(b(n - 1) - m)v_{n + m - 2}^{(b)}.
    \end{align*}
    Thus, letting $V \coloneqq \spn\{v_n^{(b)} \mid n \in \NN\} = \kk[s]r^{-b} \ ds^b$, we see that $V$ is an $\OO$-module. Since $\Pb_0 = 2v_0^{(b)}$ and $\Pb_1 = v_1^{(b)}$ are elements of $V$, Lemma \ref{lem:generators of Pb} implies that $\PP_b \subseteq V$.

    For the other inclusion $V \subseteq \PP_b$, it suffices to show that $v_0^{(b)}$ and $v_1^{(b)}$ generate $V$. We have
    \begin{align*}
        u_n \cdot v_0^{(b)} &= -bn v_n^{(b)} + 4b(n - 1) v_{n - 2}^{(b)} \\
        u_1 \cdot v_n^{(b)} &= -(m + b)v_{n + 1}^{(b)} + 4m v_{n - 1}^{(b)}
    \end{align*}
    for all $n$. The result now follows by an easy induction.
\end{proof}

The next result, which shows that $\PP_b$ is a reducible $\OO$-module for all $b \in \kk$, follows immediately from Theorem \ref{thm:new basis for P}.

\begin{corollary}\label{cor:Pb is reducible}
    Let $b \in \kk$. Then $(s + 2)\PP_b$ and $(s - 2)\PP_b$ are proper $\OO$-submodules of $\PP_b$. Consequently, $\PP_b$ is a reducible $\OO$-module for all $b \in \kk$.
\end{corollary}
\begin{proof}
    If $h$ is a multiple of $s \pm 2$ in \eqref{eq:basis-free O action on P}, then $[(s^2 - 4)g\del_s, hr \ ds^{-1}]$ is also a multiple of $s \pm 2$ for all $g \in \kk[s]$, so we see that $(s \pm 2)\PP_b$ is an $\OO$-submodule of $\PP_b$.
\end{proof}

On the other hand, $\PP_b$ is indecomposable as an $\OO$-module, provided $b \neq 0$.

\begin{lemma}
    If $b \in \kk \nonzero$, then $\PP_b$ is an indecomposable $\OO$-module. Furthermore, $\PP_0$ decomposes as
    $$\PP_0 = \kk P_0^{(0)} \oplus (s + 2)\PP_0 = \kk P_0^{(0)} \oplus (s - 2)\PP_0.$$
\end{lemma}
\begin{proof}
    The first part of the result follows from the following easy claim: if $M$ is a nonzero submodule of $\PP_b$ with $b \neq 0$, then $M$ has finite codimension in $\PP_b$. The second part is immediate from \eqref{eq:O action on Pb} and Corollary \ref{cor:Pb is reducible}.
\end{proof}

We finish by considering some consequences of Theorem \ref{thm:new basis for P} for $\PP = \PP_{-1}$. First, we simplify the notation for the new basis of $\PP$, since this is the main $\OO$-module of interest in this paper.

\begin{notation}
    We define $v_n \coloneqq v_n^{(-1)} = -s^n r \ ds^{-1} = -(t + t^{-1})^n t \ dt^{-1} \in \PP$ for $n \in \NN$.
\end{notation}

The next result computes the bracket of the centreless BCCA $\bcca$ with the new bases for $\OO$ and $\PP$.

\begin{corollary}\label{cor:new_basis_for_b}
    We have $\bcca = \spn\{u_n, v_m \mid n \geq 1, m \in \NN\}$. The bracket of $\bcca$ with this basis is given by
    \begin{align*}
        [u_n,u_m] &= (n - m)(u_{n + m} - 4u_{n + m - 2}) & &(n,m \geq 1), \\
        [u_n,v_m] &= (n - m)v_{n + m} - 4(n - m - 1)v_{n + m - 2} & &(n \geq 1, m \geq 0), \\
        [v_n,v_m] &= 0 & &(n,m \geq 0).
    \end{align*}
\end{corollary}
\begin{proof}
    Follows immediately from Theorems \ref{thm:O lambda = L(s)} and \ref{thm:new basis for P}.
\end{proof}

As was done in Lemma \ref{lem:u in terms of O}, we consider the relationship between the two bases $P_n$ and $v_n$ of $\PP$. Since $P_0 = 2v_0$, we need to be a bit careful, but otherwise the proof is nearly identical to that of Lemma \ref{lem:u in terms of O}.

\begin{lemma}\label{lem:v in terms of P}
    For all $n \in \NN$, we have
    $$v_n = Q_n + \sum_{k = 1}^{\lfloor \frac{n}{2} \rfloor} \binom{n}{k} Q_{n - 2k},$$
    where $Q_n \coloneqq \begin{cases}
        P_n, &\text{if } n \neq 0, \\
        \frac{1}{2}P_0, &\text{if } n = 0.
    \end{cases}$
\end{lemma}
\begin{proof}
    Similar to Lemma \ref{lem:u in terms of O}, except we apply the identity
    $$\binom{n - 1}{k} + \binom{n - 1}{k - 1} = \binom{n}{k}$$
    in the last part of the proof.
\end{proof}

Thanks to Lemmas \ref{lem:u in terms of O} and \ref{lem:v in terms of P}, we can easily deduce what the central extension of $\bcca$ giving rise to $\widehat{\bcca}$ looks like using the new basis $\{u_n, v_m \mid n \geq 1, m \geq 0\}$ of $\bcca$. As the next result shows, the central extension becomes rather unwieldy with this choice of basis.

\begin{corollary}\label{cor:central extension new basis}
    View $u_n$ and $v_n$ as elements of the (centrally extended) BCCA $\widehat{\bcca}$. Then
    \begin{align*}
        [u_{2n},v_{2m + 1}] &= (2(n - m) - 1)v_{2(n + m) + 1} - 8(n - m - 1)v_{2(n + m) - 1} \\
        [u_{2n + 1},v_{2m}] &= (2(n - m) + 1)v_{2(n + m) + 1} - 8(n - m)v_{2(n + m) - 1} \\
        [u_{2n},v_{2m}] &= 2(n - m)v_{2(n + m)} - 4(2(n - m) - 1)v_{2(n + m - 1)} \\
        &\quad+ \left(\sum_{k = 0}^{\min(n,m)} \binom{2n}{n - k} \binom{2m}{m - k} \frac{k^2 (4k^2 - 1)}{3n}\right)C_M \\
        [u_{2n + 1},v_{2m + 1}] &= 2(n - m)v_{2(n + m + 1)} - 4(2(n - m) - 1)v_{2(n + m)} \\
        &\quad+ \left(\sum_{k = 0}^{\min(n,m)} \binom{2n + 1}{n - k} \binom{2m + 1}{m - k} \frac{2 k (2k + 1)^2 (k + 1)}{3(2n + 1)}\right)C_M
    \end{align*}
    for all $n,m$. \qed
\end{corollary}

\begin{remark}\label{rem:lack of center}
    The centreless BCCA is strongly almost-graded according to \cite[Definition 4.1]{Schlichenmaier}: setting $\bcca_n \coloneqq \spn\{u_n, v_n\}$ for $n \in \NN$ (where $u_0 = 0$), we see that
    $$[\bcca_n, \bcca_m] \subseteq \bigoplus_{k = n + m - 2}^{n + m} \bcca_k$$
    for all $n,m \in \NN$, by Corollary \ref{cor:new_basis_for_b}. However, Corollary \ref{cor:central extension new basis} shows that cocycle on $\bcca$ giving rise to $\widehat{\bcca}$ is not compatible with the almost-grading, in the sense that is impossible to extend the almost-grading of $\bcca$ to an almost-grading of $\widehat{\bcca}$.
    
    This motivates our approach in later sections, where we focus on the centreless BCCA, since we would like to use the new basis $\{u_n, v_m \mid n \geq 1, m \geq 0\}$ to study the representation theory of the BCCA. This becomes very difficult if we have to deal with the central extension as described in Corollary \ref{cor:central extension new basis}. Indeed, when one has an almost-graded Lie algebra, the \emph{local} central extensions are generally preferred and are better understood (see \cite[Definition 4.1]{SchlichenmaierCrelle}). The central extension of $\bcca$ giving rise to $\widehat{\bcca}$ is not local with respect to the almost-grading defined above, making it difficult to work with.

    Furthermore, recall that our goal is to construct a decreasing filtration on the BCCA. However, the observations above imply that that this central extension of the BCCA is not compatible with the construction of a filtration for the BCCA as per Definition \ref{def:filtered Lie algebra}. This is because any filtration on $\widehat{\bcca}$ would necessarily fail to be weakly convergent. Because of this, we will restrict our attention to the centreless BCCA in Sections \ref{sec:filtration} and \ref{sec:Whittaker modules}.
\end{remark}

Finally, we consider what happens to $\PP$ under the change of variables $s \mapsto t$. Since $r$ is a square root of $s^2 - 4$ by \eqref{eq:r^2 = s^2 - 4}, we get that $r \mapsto \sqrt{t^2 - 4} = \sqrt{f}$ under this change of variables.

\begin{corollary}\label{cor:change of variables on P}
    Under the change of variables $s \mapsto t$, the $\OO$-module $\PP$ corresponds to the $\f$-module $\mathfrak{m} \coloneqq \kk[t] \sqrt{f} \ dt^{-1}$. In other words, $\bcca \cong \f \ltimes \mathfrak{m}$.
\end{corollary}
\begin{proof}
    Follows from Theorems \ref{thm:O lambda = L(s)} and \ref{thm:new basis for P}, and the observation that $r = \sqrt{s^2 - 4}$ from \eqref{eq:r^2 = s^2 - 4}.
\end{proof}

We finish by summarising the results of this section for the centreless BCCA.
\begin{itemize}
    \item We have defined two different bases for the centreless BCCA, denoted by $\{\OO_n, \PP_m \mid n \geq 1, m \geq 0\}$ and $\{u_n, v_m \mid n \geq 1, m \geq 0\}$. The Lie bracket of $\bcca$ with these two bases is given by
    \begin{align*}
        [\OO_n,\OO_m] &= (n - m)\OO_{n + m} - (n + m)\OO_{n - m} & &(n,m \geq 1), \\
        [\OO_n, P_m] &= (n - m)P_{n + m} + (n + m)P_{n - m} & &(n \geq 1, m \geq 0), \\
        [P_n,P_m] &= 0 & &(n,m \geq 0),
    \end{align*}
    \begin{align*}
        [u_n,u_m] &= (n - m)(u_{n + m} - 4u_{n + m - 2}) & &(n,m \geq 1), \\
        [u_n,v_m] &= (n - m)v_{n + m} - 4(n - m - 1)v_{n + m - 2} & &(n \geq 1, m \geq 0), \\
        [v_n,v_m] &= 0 & &(n,m \geq 0).
    \end{align*}
    \item In basis-free notation, we have $\OO = \kk[t + t^{-1}](t^2 - 1)\del$ and $\PP = \kk[t + t^{-1}] t \ dt^{-1}$. Letting $s = t + t^{-1}$, $r = t - t^{-1}$, and $\del_s = \frac{1}{s'}\del = \frac{t^2}{t^2 - 1}\del$, we can also write $\OO$ and $\PP$ as
    $$\OO = \kk[s](s^2 - 4)\del_s, \quad \PP = \kk[s] r \ ds^{-1}.$$
    \item There is an isomorphism $\Phi \colon \OO \ltimes \PP \to \f \ltimes \mathfrak{m}$ given by the change of variables $s \mapsto t$, where $\f = f\WW_1 = (t^2 - 4)\WW_1$ and $\mathfrak{m} = \kk[t]\sqrt{f} \ dt^{-1}$. Explicitly, $\Phi(u_n) = f_n$ and $\Phi(v_n) = -t^n\sqrt{f} \ dt^{-1}$ for all $n$, where $f_n = L_n - 4L_{n - 2}$. In basis-free notation,
    $$\Phi(p(s)(s^2 - 4)\del_s) = p(t)(t^2 - 4)\del, \quad \Phi(p(s)r \ ds^{-1}) = p(t)\sqrt{f} \ dt^{-1}$$
    for all $p \in \kk[t]$. In particular, $\OO$ is isomorphic to a subalgebra of $\WW_1$ of finite codimension.
\end{itemize}

\section{Filtering the BCCA} \label{sec:filtration}

\subsection{Defining the filtration}

Using the alternate basis for the BCCA from Corollary \ref{cor:new_basis_for_b}, we construct a descending filtration for $\bcca$, thus allowing us to define and study Whittaker modules for the BCCA. We first filter the subalgebra $\OO$ of the BCCA. It is convenient to introduce a shift in the basis $u_n$ of $\OO$: define
$$\mathcal{U}_n \coloneqq u_{n + 2} = -(t + t^{-1})^{n + 1}(t^2 - 1)\del = -s^{n + 1}(s^2 - 4)\del_s$$
for $n \geq -1$. We can now construct a filtration of $\OO$.

\begin{lemma}\label{lem:filtration of O}
    For $n \in \ZZ$, define
    $$\F_n\OO \coloneqq s^{n + 1}\OO = \spn\{\mathcal{U}_k \mid k \geq n\}.$$
    Then $\F$ is a filtration of $\OO$.
\end{lemma}
\begin{proof}
    We have
    \begin{equation}\label{eq:shifted bracket UU}
        \begin{aligned}
            [\mathcal{U}_n,\mathcal{U}_m] &= [u_{n + 2},u_{m + 2}] = (n - m)(u_{n + m + 4} - 4u_{n + m + 2}) \\
            &= (n - m)(\mathcal{U}_{n + m + 2} - 4 \, \mathcal{U}_{n + m}),
        \end{aligned}
    \end{equation}
    so we see that $[\F_n\OO, \F_m\OO] \subseteq \F_{n + m}\OO$ for all $n,m$.
\end{proof}

We now extend the filtration $\F$ to the full BCCA. As before, we must shift the basis $v_n$: define
$$\VV_n \coloneqq v_{n + 1} = -(t + t^{-1})^{n + 1} t \ dt^{-1} = -s^{n + 1} r \ ds^{-1}$$
for $n \geq -1$.

\begin{proposition}\label{prop:filtration for BCCA}
    For $n \in \ZZ$, define
    $$\F_n\PP \coloneqq s^{n + 1}\PP = \spn\{\VV_k \mid k \geq n\}.$$
    Then $[\F_n\OO,\F_m\PP] \subseteq \F_{n + m}\PP$ for all $n,m \in \ZZ$. Consequently, we can filter the full centreless BCCA as follows:
    $$\F_n \bcca = \F_n (\OO\ltimes\PP) \coloneqq \F_n \OO \ltimes \F_n \PP = \spn\{\UU_k, \VV_k \mid k \geq n\}.$$
\end{proposition}
\begin{proof}
    For all $n,m \geq -1$, we have
    \begin{equation}\label{eq:shifted bracket UU-VV}
        \begin{aligned}
            [\UU_n,\VV_m] &= [u_{n + 2},v_{m + 1}] = (n - m + 1)v_{n + m + 3} - 4(n - m)v_{n + m + 1} \\
            &= (n - m + 1)\VV_{n + m + 2} - 4(n - m)\VV_{n + m},
        \end{aligned}
    \end{equation}
    and thus $[\F_n\OO, \F_m \PP] \subseteq \F_{n + m} \PP$. It is now clear that $\F$ is a filtration of the entirety of $\bcca$, by Lemma \ref{lem:filtration of O}. In other words, $[\F_n\bcca,\F_m\bcca] \subseteq \F_{n + m}\bcca$, and
    \begin{equation*}
        \bcca = \F_{-1}\bcca \supseteq \F_{0}\bcca \supseteq \F_{1}\bcca \supseteq \dots. \qedhere
    \end{equation*}
\end{proof}

We finish this section by considering the associated graded algebra to the filtration $\F$.

\begin{lemma}\label{lem:assoc_graded}
    The associated graded Lie algebras to the filtration $\F$ are 
    \begin{enumerate}
        \item $\gr_\F \OO \cong \WW_1$, and
        \item $\gr_\F \bcca \cong \WW_1 \ltimes \ad(\WW_1)$, where $\ad(\WW_1)$ is another copy of $\WW_1$, viewed as the adjoint module of $\WW_1$ with abelian Lie bracket.
    \end{enumerate}
\end{lemma}
\begin{proof}
    We have
    \begin{align*}
        [\UU_n + \F_{n+1}\OO, \UU_m + \F_{m + 1}\OO] &= [\UU_n,\UU_m]+\F_{m+n+1}\OO \\
        &= -4(n - m)\UU_{n + m} + \F_{n + m + 1}\OO.
    \end{align*}
    Thus, the Lie algebra homomorphism $\UU_n + \F_{n + 1}\OO \mapsto -4L_n$ establishes the isomorphism $\gr_\F \OO \cong \WW_1$. Similarly, we have
    \begin{align*}
        [\UU_n + \F_{n + 1}\bcca, \VV_m + \F_{m + 1}\bcca] &= [\UU_n,\VV_m]+\F_{n+m+1}\bcca \\
        &= -4(n - m)\VV_{n + m} + \F_{n + m + 1}\bcca.
    \end{align*}
    It is now clear that $\gr_\F \bcca \cong \WW_1 \ltimes \ad(\WW_1)$ using a similar isomorphism to the above, as required.
\end{proof}

\subsection{The derived subalgebras of the filtered subspaces}

To study Whittaker modules over $\OO$ and $\bcca$, we must study their Whittaker functions. By definition, these are Lie algebra homomorphisms $\F_n \OO \to \kk$ or $\F_n \bcca \to \kk$, for some $n \in \NN$. Since $\kk$ is an abelian Lie algebra, these homomorphisms must necessarily vanish on the derived subalgebras $[\F_n \OO, \F_n \OO]$ or $[\F_n \bcca, \F_n \bcca]$. We compute these in this section.

First, note that $[\F_n\OO,\F_n\OO] \subseteq \F_{2n + 1}\OO$, by \eqref{eq:shifted bracket UU}. We now prove that $[\F_n\OO,\F_n\OO]$ is spanned by $\UU_{2n + 2 + k} - 4 \, \UU_{2n + k}$ for $k \geq 1$.

\begin{lemma}\label{lem:derived subalgebra Fn O}
    For all $n \geq -1$, the derived subalgebra $[\F_n\OO,\F_n\OO]$ has the following form:
    $$[\F_n\OO,\F_n\OO] = (s^{2n + 4} - 4s^{2n + 2})\OO = (s^{2n + 4} - 4s^{2n + 2})(s^2 - 4)\kk[s]\del_s.$$
    Consequently, the set $\{\UU_{2n + 3 + k} - 4 \, \UU_{2n + 1 + k} \mid k \in \NN\}$ is a basis for $[\F_n\OO,\F_n\OO]$.
\end{lemma}
\begin{proof}
    Recall from Theorem \ref{thm:O lambda = L(s)} that $\OO = (s^2 - 4)\kk[s]\del_s$, and that $\F_n\OO = s^{n + 1}\OO$ by definition. Therefore,
    $$[\F_n\OO,\F_n\OO] = [s^{n + 1}(s^2 - 4)\kk[s]\del_s, s^{n + 1}(s^2 - 4)\kk[s]\del_s] = \Big(s^{n + 1}(s^2 - 4)\Big)^2\kk[s]\del_s,$$
    by \eqref{eq:fW bracket}. Simplifying, we get
    $$[\F_n\OO,\F_n\OO] = \Big(s^{n + 1}(s^2 - 4)\Big)^2\kk[s]\del_s = s^{2n + 2}(s^2 - 4)^2\kk[s]\del_s = (s^{2n + 4} - 4s^{2n + 2})(s^2 - 4)\kk[s]\del_s.$$
    Noting that 
    $$\UU_{2n + 3 + k} - 4 \, \UU_{2n + 1 + k} = (s^{2n + 4 + k} - 4s^{2n + 2 + k})(s^2 - 4)\del_s,$$
    it is now clear that the result follows.
\end{proof}

We now proceed as in Lemma \ref{lem:derived subalgebra Fn O} to compute the derived subalgebra $[\F_n\bcca,\F_n\bcca]$ for the filtration of the full BCCA. Similarly to the above, \eqref{eq:shifted bracket UU-VV} implies that $\F_n\OO \cdot \F_n\PP \subseteq \F_{2n + 1}\PP$. In fact, this is an equality: we have $\F_n\OO \cdot \F_n\PP = \F_{2n + 1} \PP$, as we prove next.

\begin{lemma}\label{lem:derived subalgebra Fn b}
    For all $n \geq -1$, we have $[\F_n\OO,\F_n\PP] = \F_{2n + 1}\PP$. Consequently,
    $$[\F_n\bcca,\F_n\bcca] = [\F_n\OO,\F_n\OO] \ltimes \F_{2n + 1}\PP,$$
    and therefore the set $\{\UU_{2n + 3 + k} - 4 \, \UU_{2n + 1 + k}, \VV_{2n + 1 + k} \mid k \in \NN\}$ is a basis for $[\F_n\bcca,\F_n\bcca]$.
\end{lemma}
\begin{proof}
    For all $m \geq n$, we have
    \begin{align*}
        \VV_{2m + 1} &= \frac{1}{4}[\UU_m,\VV_{m + 1}] \in [\F_n\OO,\F_n\PP], \\
        \VV_{2m + 2} &= [\UU_m,\VV_m] \in [\F_n\OO,\F_n\PP],
    \end{align*}
    by \eqref{eq:shifted bracket UU-VV}. The result follows.
\end{proof}

\subsection{The abelianisations of the filtered subspaces}

We finish this section by computing the dimensions of the abelianisations of $\F_n \OO$ and $\F_n \bcca$. This computation will tell us exactly how many degrees of freedom one has when defining a Whittaker function on $\OO$ or on $\bcca$. The following easy observation will be useful for this purpose.

\begin{lemma}\label{lem:codimension of submodule-subalgebra}
    For any $g \in \kk[t]$, the codimension of $g\WW_1$ in $\WW_1$ is $\deg(g)$. \qed
\end{lemma}

Of course, Lemma \ref{lem:codimension of submodule-subalgebra} also implies that $g(s)\kk[s]\del_s$ has codimension $\deg(g)$ in $\kk[s]\del_s$, via the change of variables $s \mapsto t$.

\begin{proposition}\label{prop:abelianisation}
    We have
    $$\dim\left(\frac{\F_n\OO}{[\F_n\OO,\F_n\OO]}\right) = n + 3, \quad \dim\left(\frac{\F_n\bcca}{[\F_n\bcca,\F_n\bcca]}\right) = 2n + 4,$$
    for all $n \geq -1$.
\end{proposition}
\begin{proof}
    By Theorem \ref{thm:O lambda = L(s)}, we have
    $$\F_n\OO = s^{n + 1}(s^2 - 4)\kk[s]\del_s.$$
    This is a subalgebra of $\kk[s]\del_s$ of codimension $n + 3$, by Lemma \ref{lem:codimension of submodule-subalgebra}. On the other hand, Lemma \ref{lem:codimension of submodule-subalgebra} also implies that
    $$[\F_n\OO,\F_n\OO] = \Big(s^{n + 1}(s^2 - 4)\Big)^2\kk[s]\del_s$$
    is a subalgebra of $\kk[s]\del_s$ of codimension $2n + 6$. Therefore, $[\F_n\OO,\F_n\OO]$ has codimension $n + 3$ in $\F_n\OO$, as required.

    For the second part of the result, we observe that the codimension of $\F_m\PP$ in $\F_n\PP$ is $m - n$, for all $m \geq n \geq -1$. Now, Lemmas \ref{lem:derived subalgebra Fn b} and \ref{lem:codimension of submodule-subalgebra} imply that
    $$\dim\left(\frac{\F_n\bcca}{[\F_n\bcca,\F_n\bcca]}\right) = \dim\left(\frac{\F_n\OO}{[\F_n\OO,\F_n\OO]}\right) + \dim(\F_n\PP/\F_{2n + 1}\PP) = n + 3 + n + 1 = 2n + 4,$$
    which concludes the proof.
\end{proof}

More precisely, write $\overline{x}$ for the image of $x$ in the abelianisation $\F_n \bcca/[\F_n \bcca, \F_n \bcca]$ or $\F_n \OO/[\F_n \OO, \F_n \OO]$, where $x \in \F_n\bcca$ or $x \in \F_n \OO$, and $n \in \ZZ$. Then Proposition \ref{prop:abelianisation} implies that $\{\overline{\mathcal{U}}_n, \dots,  \overline{\mathcal{U}}_{2n+2}\}$ forms a basis of $\F_n \OO/[\F_n\OO,\F_n\OO]$, and $\{\overline{\mathcal{U}}_n, \dots,  \overline{\mathcal{U}}_{2n+2}\} \cup \{\overline{\VV}_n, \dots, \overline{\VV}_{2n}\}$ forms a basis of $\F_n \bcca/[\F_n\bcca,\F_n\bcca]$.

The simplest application of Proposition \ref{prop:abelianisation} is that we can completely classify all one-dimensional representations of $\OO$ and $\bcca$.

\begin{defn}
    Let $\lambda, \mu \in \kk$, and define the $\OO$-module $\kk_{\lambda,\mu}$ as follows: as a vector space, $\kk_{\lambda,\mu} = \kk 1_{\lambda,\mu}$, and the $\OO$-action is given by
    $$\UU_{-1} \cdot 1_{\lambda,\mu} = \lambda 1_{\lambda,\mu}, \quad \UU_0 \cdot 1_{\lambda,\mu} = \mu 1_{\lambda,\mu}.$$
    The action of the other basis elements is determined by these and the formula
    $$\UU_n \cdot 1_{\lambda,\mu} = 4 \, \UU_{n - 2} \cdot 1_{\lambda,\mu} \quad (n \geq 1).$$
\end{defn}

For example, we have $\UU_1 \cdot 1_{\lambda,\mu} = 4\lambda 1_{\lambda,\mu}$ and $\UU_2 \cdot 1_{\lambda,\mu} = 4\mu 1_{\lambda,\mu}$. In fact, all one-dimensional representations of $\OO$ and $\bcca$ are of the above form, which is an immediate consequence of Lemmas \ref{lem:derived subalgebra Fn O} and \ref{lem:derived subalgebra Fn b}, and Proposition \ref{prop:abelianisation}.

\begin{corollary}
    If $M$ is a one-dimensional $\OO$-representation, then $M \cong \kk_{\lambda,\mu}$ for some $\lambda, \mu \in \kk$. If $M$ is a one-dimensional $\bcca$-representation, then $\PP \cdot M = 0$, so $M$ can be viewed as a representation over $\bcca/\PP \cong \OO$. \qed
\end{corollary}

\section{Whittaker modules}\label{sec:Whittaker modules}

Having defined and analysed the filtrations for $\OO$ and $\bcca$, we now study Whittaker modules for these Lie algebras (recall Lemma \ref{lem:filtration leads to Whittaker}). We emphasise that the construction of the filtration from Section \ref{sec:filtration} is not compatible with the known central extension of the BCCA, so we focus on the centreless BCCA in this section.

\subsection{Whittaker modules for \texorpdfstring{$\OO$}{O}}\label{subsec:Whittaker O}

We start by defining Whittaker functions for $\OO$, which in turn define universal Whittaker modules for $\OO$. We further characterise precisely when these modules are irreducible.

\begin{defn}
    A \emph{Whittaker function on $\OO$} is a Lie algebra homomorphism $\varphi_n \colon \F_n\OO \to \kk$ for some $n \in \NN$.
\end{defn}

As we show next, a Whittaker function $\varphi_n \colon \F_n \OO \to \kk$ is determined by its values on $\mathcal{U}_n, \dots, \mathcal{U}_{2n+2}$.

\begin{lemma}\label{lem:Whittaker generators}
    A Whittaker function $\varphi_n \colon \F_n\OO \to \kk$ is determined by its values on  $\mathcal{U}_n, \dots,  \mathcal{U}_{2n+2}$, and $\varphi_n(\UU_{m + 2}) = 4\varphi_n(\UU_m)$ for all $m \geq 2n + 1$. 
\end{lemma}
\begin{proof}
    Follows immediately from  Lemma \ref{lem:derived subalgebra Fn O} and Proposition \ref{prop:abelianisation}.
\end{proof}

\begin{defn}
    Let $\varphi_n \colon \F_n\OO \to \kk$ be a Whittaker function and $\kk_{\varphi_n} = \kk1_{\varphi_n}$ be the one-dimensional $\F_n\OO$-module defined by $\varphi_n$, in other words,
    $$x \cdot 1_{\varphi_n} = \varphi_n(x)1_{\varphi_n}$$
    for all $x \in \F_n\bcca$. Define the \emph{universal Whittaker module for $\OO$ of type $\varphi_n$} to be the induced module
    $$M_{\varphi_n} \coloneqq \Ind_{\F_n\OO}^{\mathcal{O}}\kk_{\varphi_n}.$$
\end{defn}

\begin{remark} \label{rem:continuous 1-dim F_n O modules}
    Note that $\kk_{\varphi_n}$ is a continuous module for $\F_n\OO$ in the sense of Rudakov \cite{Rudakov} if and only if $\varphi_n(\UU_{2n + 1}) = \varphi_n(\UU_{2n + 2}) = 0$. In fact, if $\varphi_n(\UU_{2n + 1}) \neq 0$, then  $\varphi(\mathcal{U}_{2n + 2k + 1}) \neq 0$ for all $k \geq 0$, and if $\varphi_n(\UU_{2n + 2}) \neq 0$, then $\varphi_n(\UU_{2m + 2k + 2}) \neq 0$ for all $k \geq 0$, by Lemma \ref{lem:Whittaker generators}.
\end{remark}

\begin{remark}\label{rem:phi0 is not Whittaker}
    The Lie algebra $\F_0 \OO$ is not quasi-nilpotent \cite[Remark 3.5]{Buzaglo}, so $M_{\varphi_0}$ is not a Whittaker module according to Definition \ref{def:Whittaker module}. As we will see, the criteria for irreducibility for $M_{\varphi_0}$ and their proofs are slightly different to those for $M_{\varphi_n}$ with $n \geq 1$, highlighting the importance of the quasi-nilpotence condition. However, to avoid piling onto the nomenclature, we abuse terminology and opt to call $M_{\varphi_0}$ a universal Whittaker module.

    We also highlight that $\F_n \OO$ is not an ideal of $\OO$ if $n \geq 0$. Therefore, the universal Whittaker modules $M_{\varphi_n}$ are never quasi-Whittaker modules according to Definition \ref{def:quasi-Whittaker modules}.
\end{remark}

The following result is the main theorem of this subsection.

\begin{theorem}\label{theo: simplicity}
    Let $n \in \NN$, and let $\varphi_n \colon \F_n \OO \to \kk$ be a Whittaker function. If $n = 0$, the universal Whittaker module $M_{\varphi_0}$ is irreducible if and only if $\varphi_0(\UU_{2}) \neq 4 \varphi_0(\UU_{0})$. If $n \geq 1$, the universal Whittaker module $M_{\varphi_n}$ is irreducible if and only if $\varphi_n(\UU_{2n + 1}) \neq 4 \varphi_n(\UU_{2n - 1})$ or $\varphi_n(\UU_{2n + 2}) \neq 4\varphi_n(\UU_{2n})$.
\end{theorem}

We now proceed to prove that $M_{\varphi_n}$ is reducible if the condition from Theorem \ref{theo: simplicity} is not satisfied, proving one direction of Theorem \ref{theo: simplicity}. As mentioned in Remark \ref{rem:phi0 is not Whittaker}, the proofs are slightly different depending on whether $n = 0$ or $n \geq 1$.

\begin{lemma}\label{lem:M0 converse}
    Let $\varphi_0 \colon \F_0 \OO \to \kk$ be a Whittaker function with $\varphi_0(\UU_{2}) = 4\varphi_0(\UU_{0})$. Then $U(\OO)\UU_{-1} \cdot 1_{\varphi_0}$ is a proper submodule of the universal Whittaker module $M_{\varphi_0}$.
\end{lemma}

\begin{proof}
    Let $\lambda_i \coloneqq \varphi_n(\UU_{i})$ for $i \geq 0$, and $w \coloneqq (\lambda_1 - 4 \, \UU_{-1}) \cdot 1_{\varphi_0}=[\UU_0, \UU_{-1}]\cdot 1_{\varphi_0}$. We claim that $U(\OO) \cdot w$ is a proper $\OO$-submodule of $M_{\varphi_0}$.
    
    Lemma \ref{lem:Whittaker generators} together with the assumption that $\lambda_2 = 4\lambda_0$ imply that $\lambda_{n + 2} = 4\lambda_n$ for all $n \geq 0$. Then
    \begin{align*}
        \UU_0 \cdot w &= \UU_0(\lambda_1 - 4 \, \UU_{-1}) \cdot 1_{\varphi_0} \\
        &= \Big(\lambda_0\lambda_1 - 4[\UU_0,\UU_{-1}] - 4 \, \UU_{-1}\UU_0\Big) \cdot 1_{\varphi_0} \\
        &= \Big(\lambda_0\lambda_1 - 4(\UU_1 - 4 \, \UU_{-1}) - 4 \lambda_0 \, \UU_{-1}\Big) \cdot 1_{\varphi_0} \\
        &= \Big((\lambda_0 - 4)\lambda_1 - 4(\lambda_0 - 4)\UU_{-1}\Big) \cdot 1_{\varphi_0} \\
        &= (\lambda_0 - 4)w.
    \end{align*}
    Now letting $n \geq 1$, we have
    \begin{align*}
        \UU_n \cdot w &= \UU_n(\lambda_1 - 4\,\UU_{-1}) \cdot 1_{\varphi_0} \\
        &= \Big(\lambda_n\lambda_1 - 4[\UU_n,\UU_{-1}] - 4\,\UU_{-1}\UU_n\Big) \cdot 1_{\varphi_0} \\
        &= \Big(\lambda_n\lambda_1 - 4(n + 1)(\UU_{n + 1} - 4\,\UU_{n - 1}) - 4\lambda_n \, \UU_{-1}\Big) \cdot 1_{\varphi_0} \\
        &= \lambda_n (\lambda_1 - 4\,\UU_{-1}) \cdot 1_{\varphi_0} - 4(n + 1)(\lambda_{n + 1} - 4\lambda_{n - 1}) \cdot 1_{\varphi_0} \\
        &= \lambda_n w,
    \end{align*}
    where we used that $\lambda_{n + 1} = 4\lambda_{n - 1}$ in the last equality, since $n \geq 1$. This proves the claim, and the result follows.
\end{proof}

\begin{remark}
    When $\lambda_2 = 4\lambda_0$, the submodule $U(\OO)(\lambda_1 - 4 \, \UU_{-1}) \cdot 1_{\varphi_0}$ in the proof of Lemma \ref{lem:M0 converse} is isomorphic to $M_{\varphi_0'}$, where $\varphi'_0 \colon \F_0\OO \to \kk$ is the Whittaker function with $\varphi'_0(\UU_0) = \lambda_0 - 4$ and $\varphi'_0(\UU_n) = \lambda_n$ for $n \geq 1$.
    
    In this case, the quotient module $M_{\varphi_0}/(U(\OO)(\lambda_1 - 4 \, \UU_{-1}) \cdot 1_{\varphi_0})$ is a $1$-dimensional $\OO$-module determined by $\UU_{-1} \cdot 1_{\varphi_0} = \frac{\lambda_1}4 1_{\varphi_0}, \UU_{0} \cdot 1_{\varphi_0} = \lambda_0 1_{\varphi_0}$.
\end{remark}    

The case $n \geq 1$ follows from the existence of an analogous submodule.

\begin{lemma}\label{lem:Mn converse}
    Let $n \geq 1$ and let $\varphi_n \colon \F_n \OO \to \kk$ be a Whittaker function with $\varphi_n(\UU_{2n + 1}) = 4 \varphi_n(\UU_{2n - 1})$ and $\varphi_n(\UU_{2n + 2}) = 4\varphi_n(\UU_{2n})$. Then $U(\OO) \UU_{n - 1} \cdot 1_{\varphi_n}$ is a submodule of the universal Whittaker module $M_{\varphi_n}$.
\end{lemma}
\begin{proof}
    Let $\lambda_i \coloneqq \varphi_n(\UU_{i})$ for $i \geq n$, and $w \coloneqq \UU_{n - 1} \cdot 1_{\varphi_n}$. We claim that $U(\OO) \cdot w$ is a proper submodule of $M_{\varphi_n}$.
    
    Lemma \ref{lem:Whittaker generators} together with the assumptions that $\lambda_{2n + 1} = 4\lambda_{2n - 1}$ and $\lambda_{2n + 2} = 4\lambda_{2n}$ imply that $\lambda_{k + 2} = 4\lambda_k$ for all $k \geq 2n - 1$. For $m \geq n$, we have
    \begin{align*}
        \UU_m \cdot w &= \UU_m \UU_{n - 1} \cdot 1_{\psi_n} = [\UU_m,\UU_{n - 1}] \cdot 1_{\psi_n} + \UU_{n - 1} \UU_m \cdot 1_{\psi_n} \\
        &= (m - n + 1) (\UU_{m + n + 1} - 4 \, \UU_{m + n - 1}) \cdot 1_{\psi_n} + \lambda_m w \\
        &= \lambda_m w,
    \end{align*}
    since $m + n - 1 \geq 2n - 1$, where we used that $\lambda_{k + 2} = 4\lambda_k$ for all $k \geq 2n - 1$. 

    The above shows that $\kk w$ is a one-dimensional $\F_n \OO$-module isomorphic to $\kk_{\varphi_n}$. It is now easy to see that $U(\OO) \cdot w$ is a proper submodule of $M_{\varphi_n}$ isomorphic to $M_{\varphi_n}$ itself.
\end{proof}

Note that together Lemmas \ref{lem:M0 converse} and \ref{lem:Mn converse} prove one direction of Theorem \ref{theo: simplicity}. To prove the converse, we use the theory of local functions and the orbit method, which we introduce in the next subsection.

\subsection{Local functions and the orbit method for \texorpdfstring{$\OO$}{O}}

In this subsection, we make a key observation that the Whittaker function $\varphi_n$ from Theorem \ref{theo: simplicity} can be viewed as the restriction of a certain element $\chi$ of $\WW_1^*$ whose stabilizer has finite codimension in $\WW_1$. Then Whittaker modules are exactly restrictions of modules from the Kirillov orbit method for $\WW_1$ (see \cite{Ushirobira1998, pham2025orbit}). Kirillov's orbit method \cite{Kirillov2004LecturesOT} is a guiding principle in Lie theory, establishing a correspondence between equivalence classes of unitary irreducible representations and coadjoint orbits. As a result, we will readily obtain the irreducibility criteria in Theorem \ref{theo: simplicity}.

Note that in this subsection, we view $\WW_1 = \Der(\kk[s]) = \kk[s]\del_s$ (in other words, we use the variable $s$ instead of $t$), so that we can view $\OO$ as the subalgebra $(s^2 - 4)\WW_1$ of $\WW_1$. We start by recalling the notion of a local function on $\WW_{1}$, originally defined in \cite{PetukhovSierra}. 

\begin{defn}\label{def: local function}
    Letting $x, \alpha_0, \alpha_1, \dots ,\alpha_n \in \kk$ with $\alpha_n \neq 0$, we define a \emph{one-point local function} $\chi_{x; \alpha_0, \dots,\alpha_n} \in \WW_1^\ast$ by
    \begin{equation}\label{eq:one-pointlf}
        \chi_{x; \alpha_0, \dots,\alpha_n} \colon \WW_{1} \to \kk, \quad f \partial  \mapsto \alpha_0 f(x) + \alpha_1 f'(x) + \dots +\alpha_n f^{(n)}(x),
    \end{equation}
    where $f \in \kk[s]$ and $f^{(i)}$ denotes taking the $i^\text{th}$ derivative of $f$ with respect to $s$. We call $x$ the \emph{base point} of $\chi$ and $n$ the \emph{order} of $\chi$.

    A \emph{local function} $\chi$ on $\WW_1$ is a sum of finitely many functions of the form ~\eqref{eq:one-pointlf} with (possibly) distinct $x$. In other words,
    \begin{align*}
        \chi = \chi_1 + \dots  + \chi_{\ell},
    \end{align*}
    where $\chi_i$ is a one-point local function based at $x_i \in \kk$ of order $n_i$, and the $x_i$ are pairwise distinct.
\end{defn}

Next, we recall the notion of a polarisation for a local function on $\WW_{1}$ in \cite{pham2025orbit}.

\begin{defn}\label{def:polarization}
    For every $\chi \in \WW_{1}^*$, we define a bilinear form 
    \begin{align*}
        B_\chi: \WW_{1} \times \WW_{1} &\to \kk\\  
        (u, v) &\mapsto \chi ([u, v]).
    \end{align*}
    A \emph{polarisation} $\p_\chi$ of $\chi$ is a finite-codimensional subalgebra of $\WW_{1}$ that is maximally totally isotropic with respect to $B_\chi$, in other words, $B_\chi([\p_\chi, \p_\chi]) = 0$.
\end{defn}

\begin{remark}
    In general, an arbitrary element $\chi \in \g^*$ for a Lie algebra $\g$ might not have a polarisation. However, if $\chi$ is a local function on $\WW_1$, then $\chi$ always has a polarisation \cite[Proposition 3.2]{pham2025orbit}. In this case, the polarisation can moreover be chosen as a submodule-subalgebra of $\WW_1$.
\end{remark}

We consider an example of a local function and its polarisation to illustrate the concept.

\begin{example}[{\cite{PetukhovSierra}}]
    Let $\chi \coloneqq \chi_{x; \nu_0, \nu_1}$ be a one-point local function on $\WW_1$ where $x \in \kk$ and $(\nu_0, \nu_1) \in \kk^2 \setminus \{(0, 0)\}$. In other words,
    $$\chi(f\del) = \nu_0 f(x) + \nu_1 f'(x).$$
    Consider the submodule-subalgebra $(s - x)\WW_1$ of $\WW_1$. Certainly, $(s - x)\WW_1$ has codimension 1 in $\WW_1$. Moreover, \eqref{eq:fW bracket} implies that
    $$\chi([(s - x)\WW_1, (s - x)\WW_1]) = \chi((s - x)^2\WW_1) = 0,$$
    and therefore $(s - x)\WW_1$ is a totally isotropic subspace of $\WW_1$ with respect to $B_\chi$. Note that $(s - x)\WW_1$ is maximally totally isotropic by a dimension count \cite[Lemma 3.1]{pham2025orbit}, and therefore is a polarisation of $\chi$.
\end{example}

\begin{remark}
    The polarisation $(s - x)\WW_1$ is in fact the unique polarisation for the local function $\chi_{x; \nu_0, \nu_1}$ (see, for example, \cite[Section 7.1]{PetukhovSierra} for a proof). However, polarisations for local functions of $\WW_{1}$ are in general not unique (for an example, see \cite[Example 3.3]{pham2025orbit}).
\end{remark}

\begin{remark}
    If $\mathfrak{q}$ is a subalgebra of a polarisation of a local function $\chi$, then $\chi([\mathfrak{q},\mathfrak{q}]) = 0$, and thus $\chi$ defines a one-dimensional representation of $\mathfrak{q}$:
    \begin{align*}
        \restr{\chi}{\mathfrak{q}} \colon \mathfrak{q} &\to \kk \\
        f\del &\mapsto \chi(f\del).
    \end{align*}
\end{remark}

We now explain the connection between the Whittaker functions on $\OO$ and the orbit method for the Lie algebra $\WW_{1}$ constructed in \cite{Ushirobira1998, pham2025orbit}, which will allow us to use results in \cite{Ushirobira1998, pham2025orbit} to prove the other direction of Theorem \ref{theo: simplicity}. The next result shows how to associate a local function to each Whittaker function $\varphi_n \colon \F_n \OO \to \kk$.

\begin{proposition}\label{prop:local function existence}
    Let $n \in \NN$ and let $\varphi_n \colon \F_n \OO \to \kk$ be a Whittaker function. Then there exists a local function $\chi = \chi_{0;  \alpha_0, \dots, \alpha_{2n+1}} + \chi_{2; \beta, \beta_1} + \chi_{-2; \gamma, \gamma_1}$ on $\WW_1$ such that $\varphi_n = \restr{\chi}{\F_n \OO}$.
    
    Moreover, we can always choose $\alpha_0, \beta, \gamma \in \kk$ arbitrarily, and
    \begin{align*}
        \varphi_n(\UU_{2n + 2}) \neq 4 \varphi_n(\UU_{2n}) &\iff \alpha_{2n+1}\neq 0, \\
        \varphi_n(\UU_{2n + 1}) \neq 4 \varphi_n(\UU_{2n - 1})  &\iff \alpha_{2n} \neq 0 \quad (\text{provided } n \geq 1).
    \end{align*}
\end{proposition}
\begin{proof}
    Choose arbitrary $\alpha_i, \beta, \beta_1, \gamma, \gamma_1 \in \kk$, and let
    \begin{align}\label{eq:local function form}
        \chi \coloneqq \chi_{0; \alpha_0, \dots, \alpha_{2n+1}} + \chi_{2; \beta, \beta_1} + \chi_{-2; \gamma, \gamma_1}.
    \end{align}
    It is easy to check that $\chi$ vanishes on $s^{2n+2} (s^2-4)^2\WW_1 = [\F_n \OO, \F_n \OO]$, thus $\F_n \OO$ is contained in a polarisation of $\chi$. Hence, $\chi$ restricts to a one-dimensional representation $ \restr{\chi}{\F_n \OO}$ of $\F_n \OO$. Since both $\restr{\chi}{\F_n \OO}$ and $\varphi_n$ vanish on $[\F_n \OO, \F_n \OO]$, we require the following system of equations to have $\restr{\chi}{\F_n \OO} = \varphi_n$:
    \begin{align}
        \varphi_n(\UU_k) &= (k + 3)! \alpha_{k + 3} - 4(k + 1)! \alpha_{k + 1} + 2^{k + 3}(\beta_1 + (-1)^k \gamma_1) \text{ 
        for } n \leq k \leq 2n-2,\label{eq:uk}\\
        \varphi_n(\UU_{2n-1}) &= - 4(2n)! \alpha_{2n} + 2^{2n + 2}(\beta_1 + (-1)^{2n-1} \gamma_1) \label{eq:u2n-1},\\
        \varphi_n(\UU_{2n}) &= - 4(2n+1)! \alpha_{2n+1} + 2^{2n + 3}(\beta_1 + (-1)^{2n} \gamma_1)\label{eq:u2n}, \\
        \varphi_n(\UU_{2n+1}) &= 2^{2n + 4}(\beta_1 + (-1)^{2n+1} \gamma_1)\label{eq:u2n+1}, \\
        \varphi_n(\UU_{2n+2}) &= 2^{2n + 5}(\beta_1 + (-1)^{2n+2} \gamma_1)\label{eq:u2n+2}.
    \end{align}
    Note that $\eqref{eq:uk}$ only exists for $n \geq 2$, and $\eqref{eq:u2n-1}$ only exists for $n \geq 1$. For all $n$, there are $n+3$ equations (from $\varphi_n(\UU_{n})$ to $\varphi_n(\UU_{2n+2})$) with $n+3$ variables (from $\alpha_{n+1}$ to $\alpha_{2n+1}$, as well as $\beta_1$, $\gamma_1$). Furthermore, there are no inconsistencies on the right hand side, thus there always exist $\alpha_{n+1}, \dots, \alpha_{2n+1}, \beta_1, \gamma_1 \in \kk$ satisfying the above system. In other words, there always exists $\chi$ of the form \eqref{eq:local function form} such that $\varphi_n = \restr{\chi}{\F_n\OO}$. Since $\alpha_0, \beta, \gamma$ are not involved in the system of equations, we can choose them arbitrarily.
    
    The last part of the statement follows from the combination of \eqref{eq:u2n-1} and \eqref{eq:u2n+1} (if $n \geq 1$), and the combination of \eqref{eq:u2n} and \eqref{eq:u2n+2}.
\end{proof}

For the rest of this section, we fix $n \in \NN$ and a Whittaker function $\varphi_n \colon \F_n \OO \to \kk$. Given Proposition \ref{prop:local function existence}, we introduce notation for the local function associated to $\varphi_n$.

\begin{notation}\label{not:special local function}
    For $k \geq n$, define $\lambda_k \coloneqq \varphi_n(\UU_k)$. Furthermore, using Proposition \ref{prop:local function existence}, choose $\alpha_i, \beta_1, \gamma_1 \in \kk$ such that $\varphi_n = \restr{\chi}{\F_n \OO}$, where
    $$\chi \coloneqq \chi_{0; 1, \alpha_1, \dots, \alpha_{2n+1}} + \chi_{2; 1, \beta_1} + \chi_{-2; 1, \gamma_1}.$$
    We will denote
    $$\eta \coloneqq \chi_{0; 1, \alpha_1, \dots, \alpha_{2n+1}}, \quad \theta_{a,b} \coloneqq \chi_{2; 1, a} + \chi_{-2; 1, b}, \quad \chi_{a,b} \coloneqq \eta + \theta_{a,b}$$
    for $a,b \in \kk$. For simplicity of notation, write $\theta \coloneqq \theta_{\beta_1,\gamma_1}$, so that $\chi = \eta + \theta$.
\end{notation}

Although it is not obvious at this point why we have introduced the family of local functions $\chi_{a,b}$, this will soon become clear.

\begin{remark}
    The splitting $\chi = \eta + \theta$ was introduced because $\theta$ defines a Lie algebra homomorphism $\theta \colon \OO \to \kk$, since $\theta$ vanishes on $[\OO,\OO] = (s^2 - 4)^2\WW_1$. Therefore, $\kk_\theta$ is a one-dimensional representation of $\OO$. This is not true for $\eta$, and we will soon see that $\eta$ completely controls the reducibility of $M_{\varphi_n}$.
\end{remark}

Having introduced the necessary background and notation, we proceed to  prove the irreducibility of these representations under the conditions of Theorem \ref{theo: simplicity}, using the perspective of local functions. We begin this proof with realising the subalgebra $\F_n \OO$ as a polarisation of the local function $\chi$, using results in \cite{Ushirobira1998} and \cite{pham2025orbit}.

\begin{lemma}[{\cite[Theorem 3.5]{Ushirobira1998}, \cite[Proposition 3.2]{pham2025orbit}}]\label{lem:polarization}
    Let $\chi = \chi_1 + \dots + \chi_\ell$ be a local function on $\WW_{1}$, where $\chi_i$ is a one-point local function at $x_i$ with order $n_i$. Letting $m_i \coloneqq \left\lfloor \frac{n_i}{2} \right\rfloor$, the subalgebra
    $$\p_{\chi_i} \coloneqq (s - x_i)^{m_i + 1}\WW_1$$
    is a polarisation of $\chi_i$. Furthermore,    
    $$\p_\chi \coloneqq \p_{\chi_1} \cap \dots \cap \p_{\chi_\ell} = (s - x_1)^{m_1 + 1} \dots (s - x_\ell)^{m_\ell + 1}\WW_1$$
    is a polarisation of $\chi$.
\end{lemma}

\begin{corollary}\label{cor:polarization}
     If $n = 0$, then $\F_0 \OO$ is a polarisation of $\chi_{a,b}$ for all $a,b \in \kk$. If $n \geq 1$, and $\lambda_{2n + 1} \neq 4 \lambda_{2n - 1}$ or $\lambda_{2n + 2} \neq 4\lambda_{2n}$, then $\F_n \OO$ is a polarisation of $\chi_{a,b}$ for all $a,b \in \kk$.
\end{corollary}
\begin{proof}
    Throughout this proof, we assume that $\lambda_{2n + 1} \neq 4 \lambda_{2n - 1}$ or $\lambda_{2n + 2} \neq 4\lambda_{2n}$ if $n \geq 1$.
    
    We first claim that $s^{n + 1}\WW_1$ is a polarisation of $\eta$. If $n = 0$, then by definition we have $\eta = \chi_{0;1,\alpha_1}$ for some $\alpha_1 \in \kk$. Therefore, $\p_\eta = s\WW_1$ is a polarisation of $\eta$, by Lemma \ref{lem:polarization}. If $n \geq 1$, then Proposition \ref{prop:local function existence} implies that either $\alpha_{2n} \neq 0$ or $\alpha_{2n + 1} \neq 0$. Thus, by Lemma \ref{lem:polarization}, it follows that $\p_{\eta} = s^{n+1} \WW_1$ is a polarisation of $\eta$. This proves the claim.

    Recalling that $\theta_{a,b} = \chi_{2; 1, a} + \chi_{-2; 1, b}$ by definition, we also see that $\p_{\theta_{a,b}} = (s^2 - 4)\WW_1$ is a polarisation of $\theta_{a,b}$, by Lemma \ref{lem:polarization}. We conclude that 
    \begin{equation*}
        \p_{\chi_{a,b}} = \p_{\eta + \theta_{a,b}} = s^{n+1} \WW_1 \cap (s^2 - 4)\WW_1 = s^{n + 1} (s^2 - 4)\WW_1 = \F_n \OO
    \end{equation*}
    is a polarisation of $\chi_{a,b}$. 
\end{proof}

The conditions under which $\Ind^{\WW_{1}}_{\p_\chi} \kk_\chi$ is a simple $\WW_{1}$-module are completely characterised in \cite{Ushirobira1998} and \cite{pham2025orbit}. We summarise the relevant results in the following theorem.

\begin{theorem}[{\cite[Theorem 4.17]{Ushirobira1998}, \cite[Corollary 4.27, Proposition 4.23]{pham2025orbit}}]\label{theo:simplicity local rep}
    Let $\chi = \chi_1 + \dots + \chi_\ell$ be a local function on $\WW_{1}$, where $\chi_i$ is a one-point local function at $x_i$ with order $n_i$. Letting $m_i \coloneqq \left\lfloor \frac{n_i}{2} \right\rfloor$, let
    $$\p_\chi = (s - x_1)^{m_1 + 1} \dots (s - x_\ell)^{m_\ell + 1}\WW_1$$
    be the polarisation of $\chi$ from Lemma \ref{lem:polarization}. Then 
    $\Ind_{\p_\chi}^{\WW_{1}} \kk_\chi$
    is an irreducible $\WW_{1}$-module if and only if $n_i > 0$ for all $i$.
\end{theorem}

Corollary \ref{cor:polarization} and Theorem \ref{theo:simplicity local rep} imply that $\Ind_{\F_n \OO}^{\WW_1} \kk_\chi$ can only be an irreducible $\WW_1$-module if $\beta_1$ and $\gamma_1$ are nonzero. However, this is not guaranteed by Proposition \ref{prop:local function existence}. Thankfully, we can solve this issue: the reducibility of $M_{\varphi_n} = \Ind_{\F_n \OO}^{\OO} \kk_{\chi}$ does not depend on the values of $\beta_1$ and $\gamma_1$, as we aim to show next. In other words, the reducibility of $M_{\varphi_n}$ is equivalent to that of $\Ind_{\F_n \OO}^{\OO} \kk_{\chi_{a,b}}$ for any choice of $a,b \in \kk$.

We now decompose the universal Whittaker module of $\OO$ into the tensor product of an induced module of a local function and a one-dimensional representation. This helps us view the universal Whittaker module of $\OO$ as a twisted induced module of the one-point local function $\eta$.

\begin{lemma}\label{prop:module tensor}
    Let $a,b \in \kk$, and let $\kk_{\eta}$ and $\kk_{\theta_{a,b}}$ be the one-dimensional representations of $\F_n \OO$ and $\OO$ induced by $\eta$ and $\theta_{a,b}$, respectively. Then 
    $$\Ind_{\F_n \OO}^\OO \kk_{\chi_{a,b}} \cong (\Ind_{\F_n \OO}^\OO \kk_\eta) \otimes \kk_{\theta_{a,b}}.$$
    In particular,
    $$M_{\varphi_n} \cong (\Ind^{\OO}_{\F_n \OO} \kk_{\eta}) \otimes \kk_{\theta}.$$
\end{lemma}
\begin{proof}
    Note that $\kk_{\chi_{a,b}} = \kk_{\eta + \theta_{a,b}} \cong \kk_\eta \otimes \kk_{\theta_{a,b}}$ as a representation over $\F_n \OO$. The proof now follows by \cite[Proposition 5.1.15]{dixmier1996enveloping} or \cite[Lemma 8]{LuZhao2020}. The final sentence of the result follows by the observation that $\kk_{\varphi_n} = \kk_{\chi}$ as $\F_n \OO$-modules.
\end{proof}

The reducibility of a module is preserved when taking the tensor product with a one-dimensional representation of a Lie algebra. Thus, by Lemma \ref{prop:module tensor}, the irreducibility of the universal Whittaker module is completely governed by the local function $\eta$.

\begin{corollary}\label{cor:simplicity two induced}
    Let $a,b \in \kk$. Then the $\OO$-module $\Ind^{\OO}_{\F_n \OO} \kk_{\chi_{a,b}}$ is irreducible if and only if $\Ind^{\OO}_{\F_n \OO} \kk_\eta$ is irreducible. Consequently, the universal Whittaker module $M_{\varphi_n} = \Ind_{\F_n \OO}^{\OO} \kk_\chi$ is irreducible if and only if $\Ind^{\OO}_{\F_n \OO} \kk_{\chi_{1,1}}$ is irreducible. \qed
\end{corollary}

\begin{remark}
    Corollary \ref{cor:simplicity two induced} implies that reducibility of the universal Whittaker module of type $\varphi_n$ is equivalent to the reducibility of the universal Whittaker module of type $\varphi'_n \coloneqq \restr{\chi_{1,1}}{\F_n \OO}$.
\end{remark}

We are now ready to complete the proof of Theorem \ref{theo: simplicity}. 
\begin{proof}[Proof of Theorem \ref{theo: simplicity}]
    First, assume that $\lambda_{2} \neq 4 \lambda_0$ (if $n = 0$) or that $\lambda_{2n + 1} \neq 4 \lambda_{2n - 1}$ or $\lambda_{2n + 2} \neq 4\lambda_{2n}$ (if $n \geq 1$). By Corollary \ref{cor:simplicity two induced}, $M_{\varphi_n}$ is irreducible if and only if $\Ind_{\F_n \OO}^{\OO} \kk_{\chi_{1,1}}$ is irreducible. Now, Proposition \ref{prop:local function existence} implies the following:
    \begin{itemize}
        \item If $n = 0$, then $\alpha_1 \neq 0$. 
        \item If $n \geq 1$, then $\alpha_{2n}$ and $\alpha_{2n+1}$ are not simultaneously 0.
    \end{itemize}
    By Corollary \ref{cor:polarization}, we know that $\p_{\chi_{1,1}} = \F_n \OO$ is a polarisation of $\chi_{1,1}$. Furthermore, the order of each one-point local function which makes up $\chi_{1,1}$ is at least 1. Thus, Theorem \ref{theo:simplicity local rep} implies that $\Ind_{\F_n \OO}^{\WW_1} \kk_{\chi_{1,1}}$ is an irreducible $\WW_1$-module. It is easy to see that
    $$\Ind_{\F_n \OO}^{\WW_1} \kk_{\chi_{1,1}} = \Ind_{\OO}^{\WW_1}(\Ind_{\F_n \OO}^{\OO} \kk_{\chi_{1,1}}),$$
    so the irreducibility of $\Ind_{\F_n \OO}^{\WW_1} \kk_{\chi_{1,1}}$ as a $\WW_1$-module implies the irreducibility of $\Ind_{\F_n \OO}^{\OO} \kk_{\chi_{1,1}}$ as an $\OO$-module. We therefore conclude that that $M_{\varphi_n}$ is irreducible.

    For the converse, the case $n = 0$ follows from Lemma \ref{lem:M0 converse} and the case $n \geq 1$ follows from Lemma \ref{lem:Mn converse}.
\end{proof}

\subsection{Whittaker modules for the centreless BCCA}

We move on to the study of Whittaker modules over the full centreless BCCA. Although these may seem more complicated, since the BCCA is a larger Lie algebra than $\OO$, we will see that the conditions governing the reducibility of the universal Whittaker modules over $\bcca$ are much simpler, and our proof is more elementary and direct. This is thanks to the abelian subalgebra $\PP$ of $\bcca$. As we did in Subsection \ref{subsec:Whittaker O}, we start by defining Whittaker functions on $\bcca$.

\begin{defn}
    A \emph{Whittaker function on $\bcca$} is a Lie algebra homomorphism $\psi_n \colon \F_n\bcca \to \kk$ for some $n \in \NN$.
\end{defn}

The following result is an immediate consequence of Lemma \ref{lem:derived subalgebra Fn b} and Proposition \ref{prop:abelianisation}.

\begin{lemma}\label{lem:Whittaker generators BCCA}
    A Whittaker function $\psi_n \colon  \F_n\bcca \to \kk$ is determined by its values on $\mathcal{U}_n, \dots,  \mathcal{U}_{2n+2}$ and $\VV_n, \dots, \VV_{2n}$. Once these values are determined, we have
    $$\psi_n(\UU_{m + 2}) = 4\psi_n(\UU_m), \quad \psi_n(\VV_m) = 0$$
    for all $m \geq 2n + 1$. \qed
\end{lemma}

We now define the universal Whittaker modules for the full centreless BCCA, and then proceed to study their reducibility.

\begin{defn}\label{UnivWhithBCCA}
    Let $\psi_n \colon \F_n\bcca \to \kk$ be a Whittaker function and $\kk_{\psi_n} = \kk 1_{\psi_n}$ be the one-dimensional $\F_n\bcca$-module defined by $\psi_n$, in other words,
    $$x \cdot 1_{\psi_n} = \psi_n(x)1_{\psi_n}$$
    for all $x \in \F_n\bcca$. Define the \emph{universal Whittaker module for $\bcca$ of type $\psi_n$} to be the induced module
    $$M_{\psi_n} \coloneqq \Ind_{\F_n\bcca}^{\bcca} \kk_{\psi_n}.$$
\end{defn}

As in Remark \ref{rem:phi0 is not Whittaker}, the $\bcca$-modules $M_{\psi_n}$ for $n \geq 1$ are Whittaker modules, but $M_{\psi_0}$ are not Whittaker modules.

The following result is the main theorem of this subsection.

\begin{theorem}\label{thm:bcca Whittaker}
    Let $n \in \NN$, and let $\psi_n \colon \F_n\bcca \to \kk$ be a Whittaker function. If $n = 0$, then $M_{\psi_0}$ is irreducible if and only if $\psi_0(\VV_0) \neq 0$. If $n \geq 1$, then $M_{\psi_n}$ is irreducible if and only if $\psi_n(\VV_{2n}) \neq 0$ or $\psi_n(\VV_{2n - 1}) \neq 0$.
\end{theorem}

We first prove the easier direction of Theorem \ref{thm:bcca Whittaker}: we construct a nonzero proper submodule of $M_{\psi_n}$ when the conditions of Theorem \ref{thm:bcca Whittaker} are not satisfied. We first consider the case $n = 0$.

\begin{lemma}\label{lem:proper submodule psi0}
    Suppose $n = 0$ and $\psi_0(\VV_0) = 0$. Then $U(\bcca)\VV_{-1} \cdot 1_{\psi_0}$ is a proper submodule of $M_{\psi_0}$.
\end{lemma}
\begin{proof}
    By Lemma \ref{lem:Whittaker generators BCCA}, we have $\psi_0(\VV_m) = 0$ for all $m \geq 0$. In other words, $\VV_m \cdot 1_{\psi_0} = 0$ for all $m \geq 0$. Let $w \coloneqq \VV_{-1} \cdot 1_{\psi_n}$. It is clear that $\VV_m \cdot w = 0$ for all $m \geq 0$, since the $\VV_k$ commute. Furthermore, for $m \geq 0$, we have
    \begin{align*}
        \UU_m \cdot w &= \UU_m \VV_{-1} \cdot 1_{\psi_0} = [\UU_m,\VV_{-1}] \cdot 1_{\psi_0} + \VV_{-1} \UU_m \cdot 1_{\psi_0} \\
        &= \Big((m + 2)\VV_{m + 1} - 4(m + 1)\VV_{m - 1}\Big) \cdot 1_{\psi_0} + \psi_0(\UU_m) w \\
        &= -4(m + 1)\VV_{m - 1} \cdot 1_{\psi_0} + \psi_0(\UU_m) w,
    \end{align*}
    since $m + 1 \geq 0$, where we used that $\VV_k \cdot w = 0$ for all $k \geq 0$. If $m \geq 1$, then the above implies that $\UU_m \cdot w = \psi_0(\UU_m) w$. On the other hand, if $m = 0$, then
    $$\UU_0 \cdot w = (\psi_0(\UU_0) - 4)w.$$
    It follows that $\kk w$ is a one-dimensional $\F_0 \bcca$-module, and thus $U(\bcca) \cdot w$ is a proper submodule of $M_{\psi_0}$.
\end{proof}

\begin{remark}
    In the notation from the proof of Lemma \ref{lem:proper submodule psi0}, it is easy to see that $U(\bcca) \cdot w$ is isomorphic to $M_{\psi_0'}$, where $\psi_0' \colon \F_0 \bcca \to \kk$ is given by $\psi_0'(\UU_0) = \psi_0(\UU_0) - 4$, $\psi_0'(\UU_m) = \psi_0(\UU_m)$ for $m \geq 1$, and $\psi_0'(\VV_m) = 0$ for $m \geq 0$. Therefore, Lemma \ref{lem:proper submodule psi0} implies that there is chain of submodules
    $$M_{\psi_0} \supseteq U(\bcca)\VV_{-1} \cdot 1_{\psi_0} \supseteq U(\bcca)\VV_{-1}^2 \cdot 1_{\psi_0} \supseteq U(\bcca)\VV_{-1}^3 \cdot 1_{\psi_0} \supseteq \dots.$$
\end{remark}

Next, we construct an analogous submodule to the one in Lemma \ref{lem:proper submodule psi0} in the case $n \geq 1$.

\begin{lemma}\label{lem:reducible Whittaker for full bcca}
    Let $n \geq 1$. If $\psi_n(\VV_{2n}) = \psi_n(\VV_{2n - 1}) = 0$, then $U(\bcca) \VV_{n - 1} \cdot 1_{\varphi_n}$ is a proper submodule of $M_{\psi_n}$ isomorphic to $M_{\psi_n}$ itself.
\end{lemma}
\begin{proof}
    Let $w \coloneqq \VV_{n - 1} \cdot 1_{\psi_n}$. For $m \geq n$, we have
    \begin{align*}
        \UU_m \cdot w &= \UU_m \VV_{n - 1} \cdot 1_{\psi_n} = [\UU_m,\VV_{n - 1}] \cdot 1_{\psi_n} + \VV_{n - 1} \UU_m \cdot 1_{\psi_n} \\
        &= \Big((m - n + 2)\VV_{n + m + 1} - 4(m - n + 1)\VV_{n + m - 1}\Big) \cdot 1_{\psi_n} + \psi_n(\UU_m) w \\
        &=\psi_n(\UU_m) w,
    \end{align*}
    since $n + m - 1 \geq 2n - 1$, where we used that $\mu_k = 0$ for all $k \geq 2n - 1$. Similarly,
    $$\VV_m \cdot w = \VV_m \VV_{n - 1} \cdot 1_{\psi_n} = \VV_{n - 1} \VV_m \cdot 1_{\psi_n} = \psi_n(\VV_m) w,$$
    where we used that $\VV_m$ and $\VV_{n - 1}$ commute.

    The above shows that $\kk w$ is a one-dimensional $\F_n \bcca$-module isomorphic to $\kk_{\psi_n}$. It is now easy to see that $U(\bcca) \cdot w$ is isomorphic to $M_{\psi_n}$.
\end{proof}

For the remainder of this section, we assume that at least one of $\psi_n(\VV_{2n})$ and $\psi_n(\VV_{2n - 1})$ is nonzero (if $n \geq 1)$ or that $\psi_0(\VV_0) \neq 0$ (if $n = 0$). As a result, we introduce the following notation.

\begin{notation}
    Fix $n \in \NN$ and a Whittaker function $\psi_n \colon \F_n \bcca \to \kk$. For $m \geq n$, define
    $$\lambda_m  \coloneqq \psi_n(\UU_m), \quad \mu_m \coloneqq \psi_n(\VV_m).$$
    If $n \geq 1$, let $\kappa \in \{2n - 1, 2n\}$ be maximal such that $\mu_\kappa \neq 0$. If $n = 0$, let $\kappa = 0$.
\end{notation}

By Lemma \ref{lem:Whittaker generators BCCA}, we have $\lambda_{m + 2} = 4\lambda_{m}$ and $\mu_m = 0$ for all $m \geq 2n + 1$.

Having introduced the necessary notation, we now proceed to prove the other direction of Theorem \ref{thm:bcca Whittaker}. To achieve this, we need to show that we can use actions of $\bcca$ to get $1_{\varphi_n}$ from any nonzero element of $M_{\psi_n}$ when $\mu_{2n}\neq 0$ or $\mu_{2n-1}\neq 0$. We do this by successively removing $\UU_i$ and $\VV_j$ from elements of $U(\bcca)$. First, we prove that $\VV_m$ annihilates elements in $M_{\psi_n}$ which do not contain $\UU_{-1}$, provided $m > \kappa$. Recall that $\PP = \spn\{\VV_k| k \geq -1 \}$ is an abelian subalgebra of $\bcca$.

\begin{lemma}\label{lem:large V kills no U-1}
   Let $x = \UU_{r_1} \dots \UU_{r_\ell} v \in M_{\psi_n}$, where $v \in U(\PP) \cdot 1_{\psi_n}$ and $0 \leq r_1 \leq r_2 \leq \dots \leq r_\ell \leq n - 1$. Then $\VV_{\kappa + m} \cdot x = 0$ for all $m \geq 1$.
\end{lemma}
\begin{proof}
    The result is clearly true if $n = 0$, so assume that $n \geq 1$. We proceed by induction on $\ell$, the base case $\ell = 0$ being trivial. For the induction step, assume that the statement holds for some $\ell \in \NN$. Let $x = \UU_{r_0} \UU_{r_1} \dots \UU_{r_\ell} v$, where $v \in U(\PP) \cdot 1_{\psi_n}$ and $0 \leq r_0 \leq r_1 \leq \dots \leq r_\ell \leq n - 1$. For all $m \geq 1$, we have
    \begin{align*}
        \VV_{\kappa + m} \cdot x &= \VV_{\kappa + m} \UU_{r_0} \UU_{r_1} \dots \UU_{r_\ell} v \\
        &= [\VV_{\kappa + m},\UU_{r_0}]\UU_{r_1} \dots \UU_{r_\ell} v + \UU_{r_0} \VV_{\kappa + m} \UU_{r_1} \dots \UU_{r_\ell} v \\
        &= \Big((\kappa + m - r_0 - 1)\VV_{\kappa + m + r_0 + 2} - 4(\kappa + m - r_0)\VV_{\kappa + m + r_0}\Big)\UU_{r_1} \dots \UU_{r_\ell} v \\
        &= 0,
    \end{align*}
    where we used the induction hypothesis in the last two equalities.
\end{proof}

Next, we extend Lemma \ref{lem:large V kills no U-1} to the situation where the element of $M_{\psi_n}$ contains $\UU_{-1}$.

\begin{lemma}\label{lem:large V kills}
    Let $x \in U(\F_0 \OO \ltimes \PP) \cdot 1_{\psi_n}$. Then $\VV_{\kappa + m} \cdot \UU_{-1}^N x = 0$ if $m > N$.
\end{lemma}
\begin{proof}
    Follows via an easy induction on $N$, with Lemma \ref{lem:large V kills no U-1} as the base case.
\end{proof}

In the following result, we show that the action $\VV_{\kappa + N} \cdot \UU_{-1}^N x$ results in the complete removal of $\UU_{-1}$, in other words, that $\VV_{\kappa + N} \cdot \UU_{-1}^N x$ is a nonzero element of $U(\F_0 \OO \ltimes \PP) \cdot 1_{\psi_n}$, where $x \in U(\F_0 \OO \ltimes \PP) \cdot 1_{\psi_n}$.

\begin{proposition}\label{prop:remove U-1}
    Let $x \in U(\F_0 \OO \ltimes \PP) \cdot 1_{\psi_n} \nonzero$. Then, for any $N \geq 1$, there exists  $c \in \kk \nonzero$ such that $$\VV_{\kappa + N} \cdot \UU_{-1}^N x = c x + \text{lower degree terms in } U(\F_0 \OO \ltimes \PP) \neq 0,$$
    where \emph{lower degree terms} means terms with lower degree than $x$ under the PBW filtration on $U(\bcca)$.
\end{proposition}
\begin{proof}
    We proceed by induction on $N$. First, consider the base case $N = 0$:
    $$\VV_\kappa \cdot x = \mu_\kappa x + \text{lower degree terms in } U(\F_0 \OO \ltimes \PP),$$
    which is nonzero.

    For the induction step, assume the statement is true for some $N \in \NN$. Then
    \begin{align*}
        \VV_{\kappa + N + 1} \cdot \UU_{-1}^{N + 1} x &= [\VV_{\kappa + N + 1}, \UU_{-1}] \UU_{-1}^N x + \UU_{-1} \VV_{\kappa + N + 1} \UU_{-1}^N x \\
        &= \Big((\kappa + N + 1)\VV_{\kappa + N + 2} - 4(\kappa + N + 2)\VV_{\kappa + N}\Big) \UU_{-1}^N x \\
        &= -4(\kappa + N + 2)\VV_{\kappa + N} \UU_{-1}^N x,
    \end{align*}
    where we used Lemma \ref{lem:large V kills} in the last two equalities. By the induction hypothesis, it follows that $\VV_{\kappa + N + 1} \cdot \UU_{-1}^{N + 1} x \in U(\F_0 \OO \ltimes \PP) \nonzero$, which concludes the proof.
\end{proof}

Having shown how $\UU_{-1}$ can be removed from any element of $U(\bcca)$, we now proceed to also remove $\UU_0$.

\begin{proposition}\label{prop:remove U0}
    Suppose $n \geq 1$. Then the following hold.
    \begin{enumerate}
        \item Let $x \in U(\F_1 \OO \ltimes \PP) \cdot 1_{\psi_n} \nonzero$. Then $\VV_\kappa \cdot x = \mu_\kappa x$.\label{item:remove U0 part 1}
        \item Let $x \in U(\F_1 \OO \ltimes \PP) \cdot 1_{\psi_n} \nonzero$. Then there exists $c \in \kk \nonzero$ such that $(\VV_{\kappa}-\mu_\kappa)^N \cdot \UU_{0}^N x = cx$.\label{item:remove U0 part 2}
    \end{enumerate}
\end{proposition}
\begin{proof}
    For \eqref{item:remove U0 part 1}, it suffices to consider $x = \UU_{r_1} \dots \UU_{r_\ell}y$, where $y \in U(\PP) \cdot 1_{\psi_n}$ and $1 \leq r_1 \leq r_2 \leq \dots \leq r_\ell \leq n - 1$. The result is clearly true for $n = 1$, so assume that $n \geq 2$. We proceed by induction on $\ell$, the case $\ell = 0$ being trivial. For the induction step, assume that the statement is true for some $\ell \in \NN$. Then
    $$\VV_\kappa \cdot \UU_{r_0} \UU_{r_1} \dots \UU_{r_\ell} y = \UU_{r_0} \VV_\kappa \UU_{r_1} \dots \UU_{r_\ell} y + [\VV_\kappa,\UU_{r_0}]\UU_{r_1} \dots \UU_{r_\ell} y.$$
    By the induction hypothesis, we have
    $$\VV_\kappa \UU_{r_1} \dots \UU_{r_\ell} y = \mu_\kappa \VV_\kappa \UU_{r_1} \dots \UU_{r_\ell} y.$$
    Therefore,
    $$\VV_\kappa \cdot \UU_{r_0} \UU_{r_1} \dots \UU_{r_\ell} y = \mu_\kappa \UU_{r_0} \UU_{r_1} \dots \UU_{r_\ell} y + \Big((\kappa - r_0 - 1)\VV_{\kappa + r_0 + 2} - 4(\kappa - r_0)\VV_{\kappa + r_0}\Big)\UU_{r_1} \dots \UU_{r_\ell} y.$$
    By Lemma \ref{lem:large V kills no U-1}, it follows that $\Big((\kappa - r_0 - 1)\VV_{\kappa + r_0 + 2} - 4(\kappa - r_0)\VV_{\kappa + r_0}\Big)\UU_{r_1} \dots \UU_{r_\ell} y = 0$, since $r_0 \geq 1$. We conclude that $\VV_\kappa \cdot \UU_{r_0} \UU_{r_1} \dots \UU_{r_\ell} y = \mu_\kappa \UU_{r_0} \UU_{r_1} \dots \UU_{r_\ell} y$, which concludes the proof of \eqref{item:remove U0 part 1}.
    
    For \eqref{item:remove U0 part 2}, the proof is similar to that of Proposition \ref{prop:remove U-1}, using \eqref{item:remove U0 part 1}.
\end{proof}

The next result shows how to remove all the other $\UU_i$ from an element of $U(\bcca)$ by successively acting by $(\VV_m - \mu_m)$ for some appropriately chosen $m \in \NN$. Of course, if $n = 0$ or $n = 1$, then we have already removed all $\UU_i$ in the previous results, so we do not need to consider these cases.

\begin{proposition}\label{prop:remove all U}
    Suppose $n \geq 2$, and let $x = \UU_{q_1}^{p_{1}} \dots \UU_{q_k}^{p_{k}} y$, where $y \in U(\PP) \cdot 1_{\psi_n}$ and $1 \leq q_1 < q_2 < \dots < q_k \leq n - 1$. Then
    $$(\VV_m - \mu_m) \cdot x = \begin{cases}
        -4 p_1 \mu_\kappa (\kappa - 2q_1) \UU_{q_1}^{p_{1} - 1} \dots \UU_{q_k}^{p_{k}} y, &\text{if } m = \kappa - q_1, \\
        0, &\text{if } m > \kappa - q_1.
    \end{cases}$$
\end{proposition}
\begin{proof}
    Rewrite $x$ as
    $$x = \UU_{s_1} \dots \UU_{s_\ell} y,$$
    where $1 \leq s_1 \leq s_2 \leq \dots \leq s_\ell \leq n - 1$. Of course, we have $s_1 = q_1$. For $m \geq n$, we have
    \begin{align*}
        \VV_m \cdot x &= \VV_m \UU_{s_1} \dots \UU_{s_\ell} y \\
        &= \UU_{s_1} \dots \UU_{s_\ell} \VV_m y + [\VV_m,\UU_{s_1} \dots \UU_{s_\ell}]y \\
        &= \mu_m x + \sum_{j = 1}^\ell \UU_{s_1} \dots \UU_{s_{j - 1}} [\VV_m,\UU_{s_j}] \UU_{s_{j + 1}} \dots \UU_{s_\ell}y \\
        &= \mu_m x + \sum_{j = 1}^\ell \UU_{s_1} \dots \UU_{s_{j - 1}} \Big((m - s_j - 1)\VV_{m + s_j + 2} - 4(m - s_j)\VV_{m + s_j}\Big) \UU_{s_{j + 1}} \dots \UU_{s_\ell}y.
    \end{align*}
    If $m + s_j > \kappa$, then the $j^\text{th}$ term of the above summation is zero, by Lemma \ref{lem:large V kills no U-1}. In particular, if $m > \kappa - s_1 = \kappa - q_1$, then the entire summation vanishes, and thus $\VV_m \cdot x = \mu_m x$, as required.

    Now suppose $m = \kappa - s_1 = \kappa - q_1$. By the above discussion, we can only get a contribution from the $j^\text{th}$ term of the above summation if $s_j = s_1$, since in this case we have $m + s_j = \kappa$. Therefore, switching back to the original notation for $x$, we get
    $$\VV_{\kappa - q_1} \cdot x = \mu_{\kappa - q_1} x - 4(\kappa - 2q_1) \sum_{i = 0}^{p_1 - 1} \UU_{q_1}^i \VV_\kappa \UU_{q_1}^{p_1 - i - 1} \UU_{q_2}^{p_2} \dots \UU_{q_k}^{p_k} y.$$
    By Proposition \ref{prop:remove U0}\eqref{item:remove U0 part 1}, it follows that
    $$\VV_\kappa \cdot \UU_{q_1}^{p_1 - i - 1} \UU_{q_2}^{p_2} \dots \UU_{q_k}^{p_k} y = \mu_\kappa \UU_{q_1}^{p_1 - i - 1} \UU_{q_2}^{p_2} \dots \UU_{q_k}^{p_k} y,$$
    and thus
    $$\VV_{\kappa - q_1} \cdot x = \mu_{\kappa - q_1} x - 4 p_1 \mu_\kappa (\kappa - 2q_1) \UU_{q_1}^{p_1 - 1} \UU_{q_2}^{p_2} \dots \UU_{q_k}^{p_k} y,$$
    which concludes the proof.
\end{proof}

The next result is a direct consequence of Propositions \ref{prop:remove U-1}, \ref{prop:remove U0}, and \ref{prop:remove all U}.

\begin{corollary}\label{cor:Whittaker contains polynomial in V}
    Let $M$ be a nonzero submodule of $M_{\psi_n}$. Then $M$ contains a nonzero element in $U(\PP) \cdot 1_{\psi_n}$. \qed
\end{corollary}

Thanks to Corollary \ref{cor:Whittaker contains polynomial in V}, we can simply focus on elements of $M_{\psi_n}$ not containing any $\UU_i$. Thankfully, the removal of $\VV_j$ is much easier, since the $\VV_j$ commute with each other.

\begin{lemma}\label{lem:U decreases V-length}
    Let $y = \VV_{q_1}^{p_1} \dots \VV_{q_k}^{p_k} \cdot 1_{\psi_n} \in M_{\psi_n}$ with $-1 \leq q_1 < \dots < q_k \leq n - 1$. Then
    $$(\UU_m - \lambda_m) \cdot y = \begin{cases}
        -4 p_1 \mu_\kappa (\kappa - 2q_1) \VV_{q_1}^{p_1 - 1} \VV_{q_2}^{p_2} \dots \VV_{q_k}^{p_k} \cdot 1_{\psi_n}, &\text{if } m = \kappa - q_1, \\
        0, &\text{if } m > \kappa - q_1.
    \end{cases}.$$
\end{lemma}
\begin{proof}
    Rewrite $y$ as
    $$y = \VV_{s_1} \dots \VV_{s_\ell} \cdot 1_{\psi_n},$$
    where $-1 \leq s_1 \leq s_2 \leq \dots \leq s_\ell \leq n - 1$ (and thus $s_1 = q_1$). We have
    \begin{align*}
        \UU_m \cdot y &= \UU_m \VV_{s_1} \dots \VV_{s_\ell} \cdot 1_{\psi_n} = \VV_{s_1} \dots \VV_{s_\ell} \UU_m \cdot 1_{\psi_n} + \sum_{j = 1}^\ell \VV_{s_1} \dots \VV_{s_{j - 1}} [\UU_m,\VV_{s_j}] \VV_{s_{j + 1}} \dots \VV_{s_\ell} \cdot 1_{\psi_n} \\
        &= \lambda_m y + \sum_{j = 1}^\ell \VV_{s_1} \dots \VV_{s_{j - 1}} \Big((m - s_j + 1)\VV_{m + s_j + 2} - 4(m - s_j)\VV_{m + s_j}\Big) \VV_{s_{j + 1}} \dots \VV_{s_\ell} \cdot 1_{\psi_n}
    \end{align*}
    Suppose $m = \kappa - q_1 = \kappa - s_1$. Then $m + s_j > \kappa$ if $s_j > s_1$. Therefore, we only get a contribution from the $j^\text{th}$ term in the above summation if $s_j = s_1$, since in this case we have $m + s_j = \kappa$. It follows that
    $$\UU_{\kappa - q_1} \cdot y = \lambda_{\kappa - q_1} y - 4 p_1 \mu_\kappa (\kappa - 2q_1) \VV_{q_1}^{p_1 - 1} \VV_{q_2}^{p_2} \dots \VV_{q_k}^{p_k} \cdot 1_{\psi_n},$$
    which concludes the proof. The case where $m > \kappa - q_1$ follows similarly.
\end{proof}

We are now ready to prove Theorem \ref{thm:bcca Whittaker}.

\begin{proof}[Proof of Theorem \ref{thm:bcca Whittaker}]
    One direction of the statement is simply Lemmas \ref{lem:proper submodule psi0} and \ref{lem:reducible Whittaker for full bcca}. For the converse, assume that $\mu_{2n} \neq 0$ or $\mu_{2n - 1} \neq 0$, and let $M$ be a nonzero submodule of $M_{\psi_n}$. Using Corollary \ref{cor:Whittaker contains polynomial in V}, choose an element $w \in M \cap (U(\PP) \cdot 1_{\psi_n}) \nonzero$. Now apply Lemma \ref{lem:U decreases V-length} to $w$ repeatedly to deduce that $M$ contains a nonzero scalar multiple of the generating vector $1_{\psi_n}$ of $M_{\psi_n}$. Therefore, $M = M_{\psi_n}$, which concludes the proof.
\end{proof}

\begin{remark}
    One may wonder about using the methods of \cite{ChengGaoLiuZhaoZhao} to prove Theorems \ref{theo: simplicity} and \ref{thm:bcca Whittaker}. The main issue is that $\F_n \OO$ and $\F_n \bcca$ are not ideals of $\OO$ and $\bcca$ for $n \geq 0$, making it impossible to apply their results directly, since quasi-Whittaker functions must be defined on ideals of the Lie algebra (recall Definition \ref{def:quasi-Whittaker modules}). But $\F_n \OO$ and $\F_n \bcca$ are ideals of $\F_0 \OO$ and $\F_0 \bcca$, so one could get some irreducibility results by inducing from $\F_n \OO$ or $\F_n \bcca$ to $\F_0 \OO$ or $\F_0 \bcca$ and analysing the Whittaker annihilator, as is done in \cite{ChengGaoLiuZhaoZhao}. However, deducing the irreducibility of the universal Whittaker module $M_{\varphi_n}$ or $M_{\psi_n}$ from the irreducibility of $\Ind_{\F_n \OO}^{\F_0 \OO} \kk_{\varphi_n}$ or $\Ind_{\F_n \bcca}^{\F_0 \bcca} \kk_{\psi_n}$ is a non-trivial task which would require similar analysis to our proof of Theorem \ref{thm:bcca Whittaker}. For this reason, we opted for a more direct proof in the case of $\bcca$, and a proof using the orbit method for $\OO$.
\end{remark}

\section{Future work} \label{sec:conclusions}

In summary, this paper provides new insights into the algebraic structure of the BCCA and studies some of its representations, namely restrictions of representations of larger Lie algebras and Whittaker modules. 
We hope our work serves as a useful starting point for further exploration of the BCCA and the related mathematical and physical concepts that one would encounter. We conclude this paper by presenting some of these future research directions:
\begin{enumerate}
    \item The BCCA can be obtained as the Lie algebra contraction of $\Vir$ (see \cite[Equation (8)]{BCCFGP2024}). Hence, one could study $\bcca$-modules obtained as contractions of $\Vir$-modules. Contractions at the level of representations are essential in Carrollian physics \cite{CampoleoniGonzalezOblakRiegler}, but there is still much that is unclear in this regard. For instance, the contractions of $\Vir \times \Vir$-modules built from ``mode algebras'' or free fields (i.e., $BC$-systems, Heisenberg algebra, etc.) have been considered in the physics literature\footnote{However, the procedure by which these representations are contracted seems different to the one introduced in \cite{CampoleoniGonzalezOblakRiegler}.} 
    (see for example \cite[Sections 5.1.2 and 5.2]{Hao:2021urq}, \cite[Section 4]{BagchiBanerjeeChakraborttyDuttaParekh} and \cite[Section 3.2.1]{Bagchi:2024unl}). However, there are still technical subtleties (e.g., issues with ``normal-ordering'') when it comes to contracting these infinite-dimensional representations that need to be addressed. Attempting to construct sensible representations of the BCCA as contractions of representations of the Virasoro algebra may help address some of these subtleties (see for instance \cite{RagoucyRasmussenRaymond} for some work in this direction on VOA modules). They should also help guide us toward the quantisation of the tensionless open string since we expect such representations to describe the possible spectra of this string. It would be particularly interesting if representations constructed in this manner belong to one of the classes of $\widehat{\bcca}$-modules discussed in this paper.
    
    \item The above line of enquiry leads naturally to the consideration of restrictions of $\bms$-modules that are vertex operator algebra (VOA) modules \cite{Borcherds, FLM, FrenkelBen-Zvi} (see \cite{ZhangDong} for the first appearance of $\bms$-modules\footnote{The BMS$_3$ algebra was called $W(2,2)$ algebra in this work, with both central elements being equal.} in mathematics), such as the $\bms$-module formed from two fermionic $BC$ systems \cite{Feigin, Friedan:1985ey, Friedan:1985ge}, each of conformal weight $(2,-1)$. This module is also known as the semi-infinite wedge or fermionic Fock representation \cite{Feigin, FGZ1986}, denoted $\semiinfforms(\bms)$. For one of the three vacua of closed tensionless string theory, $\semiinfforms(\bms)$ is the so-called ``ghost sector'' of the theory that appears during  Becchi--Rouet--Stora--Tyutin (BRST) quantisation \cite{Becchi:1974md, Becchi:1974xu, Becchi:1975nq, Tyutin:1975qk}. Thus, understanding the restriction of $\semiinfforms(\bms)$ to $\bcca$ (note that it was shown in \cite[Theorem 3.19]{Figueroa-OFarrillVishwa} that $c_M = 0$ for such a $\bms$-module) could be the first step towards formulating a BRST quantisation  procedure for the tensionless open string, or at least one of its vacua (if there exist many). It may also help formulate a generalised notion of semi-infinite cohomology for Lie algebras without integer-grading.

    \item In general, studying the restrictions and, if possible, contractions of VOA modules over $\bms$, $\Vir \times \Vir$ and $\Vir$ would provide field-theoretic insights for the tensionless open string, possibly informing us about its spectrum, quantisation etc. Some of these modules are well-known examples of highest weight modules, so one should expect that after carefully dealing with normal-ordering ambiguities, their restriction to $\widehat{\bcca}$-modules does indeed give rise to the ``almost free'' modules discussed in Section \ref{sec:restriction of modules}. Overall, together with the work in Section \ref{sec:restriction of modules}, we expect these two strands of exploration to aid in understanding the commutative diagram given in Figure \ref{fig:restrictions and contractions} at the level of Lie algebra representations while providing new field-theoretic realisations of the BCCA, which would be very informative in deducing the possible spectra of quantised tensionless open strings.
    
    \item Proving Conjecture \ref{conj:Ind} could not only aid the search for field-theoretic realisations of the BCCA, but also result in the development of novel methods to check the (in)decomposability of Lie algebra modules and/or study representations of infinite-dimensional non-$\ZZ$-graded Lie algebras.
    
    \item A similar programme to the one in this paper can be undertaken for higher spin BCCAs \cite[Equation (26)]{BCCFGP2024}. Following our notation in Section \ref{sec:preliminaries}, the higher spin BCCAs form the one-parameter family of Lie algebras $\OO \ltimes \PP_b$ (and their central extensions), where $b$ is usually taken to be an integer in physics, though the programme undertaken in this paper could be done for arbitrary $b \in \kk$.
    
    \item One could compute the low-dimensional cohomology of the two-parameter family of higher spin BCCAs, analogous to what was done for the Lie algebras $\W(a,b) \coloneqq \W \ltimes I (a,b)$ (see Definition \ref{def:tensor density}) in \cite{GaoJiangPei}. This would include a computation of the central extensions of $\bcca$.

    \item One can go beyond the methods in this paper to construct modules over $\widehat{\bcca}$ or the centrally closed version of $\bcca$, if it is proven not to be $\widehat{\bcca}$. Such representations would inevitably be vital in tensionless open string theory and other fields of theoretical physics.
    
    \item Finally, our work can be extended to the supersymmetric versions of the BCCA. Two such algebras, dubbed the \emph{homogenous and inhomogeneous boundary superconformal Carrollian algebras (BSCCAs)}, were obtained in \cite{Bagchi:2025jgu} via similar constructions to the ones in \cite{BCCFGP2024}. The study of modules over the homogenous BSCCA should serve as an invaluable stepping stone towards understanding the spectrum of the tensionless open superstring.
\end{enumerate}

\begin{figure}
    \centering
    \begin{tikzcd}[row sep=huge, column sep = huge]
        \Vir \times \Vir \arrow[r] \arrow[d, dashed] & \bms \arrow[d, dashed] \\
        \Vir \arrow[r] & \widehat{\bcca}
    \end{tikzcd}
    \caption{Diagram summarising the relations between $\Vir$, $\bms$ and $\widehat{\bcca}$. The vertical dashed arrows denote restrictions while the horizontal solid ones denote Lie algebra contractions.}
    \label{fig:restrictions and contractions}
\end{figure}

\appendix

\section{More subalgebras of the Witt algebra}\label{appendix}

Given the results of Subsection \ref{subsec:new basis for O}, it is natural to ask if we get similar subalgebras of the Witt algebra by choosing generators
$$\mathcal{X}_1 \coloneqq L_1 - \lambda L_{-1}, \quad \mathcal{X}_2 \coloneqq L_2 - \mu L_{-2}$$
for some $\lambda, \mu \in \kk^*$. As the following lemma shows, these elements generate the entirety of $\W$ unless $\mu = \lambda^2$.

\begin{lemma}
    If $\mu \neq \lambda^2$, then $\XX_1$ and $\XX_2$ generate $\W$.
\end{lemma}
\begin{proof}
    This proof is most easily done with the help of a computer. As such, we omit the tedious computations and instead focus on the main arguments of the proof.
    
    Let $\g$ be the subalgebra of $\W$ generated by $\XX_1$ and $\XX_2$. For $n \geq 3$, inductively define $\XX_n \coloneqq \frac{1}{n - 2}[\XX_{n - 1},\XX_1] \in \g$. In this way, we have $\XX_n = L_n + \text{lower degree terms}$ for all $n \geq 1$. There is another way to generate an element of degree 5: we have $\YY_5 \coloneqq [\XX_3,\XX_2] \in \g$. It can be checked that
    $$\XX_5 - \YY_5 = -\frac{16}{3}\lambda L_3 + \text{lower degree terms},$$
    so we define $\YY_3 \coloneqq -\frac{3}{16\lambda}(\XX_5 - \YY_5) \in \g$. Similarly, we can check that
    $$\XX_3 - \YY_3 = -\frac{45}{16\lambda}(\lambda^2 - \mu)L_1 + \text{lower degree terms}.$$
    Since we assumed that $\mu \neq \lambda^2$, we know that $\XX_3 - \YY_3 \neq 0$ (in fact, $\XX_3 = \YY_3$ if and only if $\mu = \lambda^2$). Thus, define $\YY_1 \coloneqq \frac{16\lambda}{45(\lambda^2 - \mu)}(\XX_3 - \YY_3) \in \g$. Continuing this process, we have
    $$\XX_1 - \YY_1 = -\frac{16}{15}\lambda L_{-1} + \text{lower degree terms},$$
    so we define $\YY_{-1} \coloneqq -\frac{15}{16\lambda}(\XX_1 - \YY_1) \in \g$. Specifically, we have
    $$\YY_{-1} = L_{-1} + \frac{3\mu}{16\lambda} L_{-3} + \frac{\mu}{16} L_{-5}.$$
    We also have the following element of $\g$:
    $$\mathcal{Z} \coloneqq [\XX_1,\YY_{-1}] = 2\lambda L_0 + \frac{3}{4} \mu L_{-2} - \frac{\lambda^2 \mu}{4} L_{-6}.$$
    Notice that $\W_{\leq 1} \coloneqq \spn\{L_n \mid n \leq 1\}$ is a subalgebra of $\W$ isomorphic to $\WW_1$. By \cite[Lemma 4.7]{Buzaglo2}, it follows that $\YY_{-1}$ and $\mathcal{Z}$ generate a subalgebra of $\W_{\leq 1}$ finite codimension. Therefore, $\g \cap \W_{\leq 1}$ has finite codimension in $\W_{\leq 1}$.

    By a completely symmetric argument, we also get that $\g \cap \WW_1$ has finite codimension in $\WW_1$. Therefore, $\g$ has finite codimension in $\W$. But now Proposition \ref{prop:finite codimension subalgebras} implies that
    \begin{equation}\label{eq:finite codimension containment of X1 X2 subalgebra}
        g^n\W \subseteq \g \subseteq g\W
    \end{equation}
    for some $g \in \kk[t,t^{-1}]$, and therefore $g$ must divide both $\XX_1$ and $\XX_2$. Without loss of generality, we may assume that $g \in \kk[t]$, $g$ is monic, and $g(0) \neq 0$. Note that $\XX_1 = -(t^2 - \lambda)\del$ and $\XX_2 = -t^{-1}(t^4 - \mu)\del$. But $\gcd(t^2 - \lambda, t^4 - \mu) = 1$, since $\mu \neq \lambda^2$, and therefore $g = 1$, since $g$ is a common divisor of $\XX_1$ and $\XX_2$. Now \eqref{eq:finite codimension containment of X1 X2 subalgebra} gives that $\g = \W$.
\end{proof}


\begin{thebibliography}{BFOGV24}

\bibitem[ABD{\etalchar{+}}25]{Aggarwal:2025hji}
Ankit Aggarwal, Arjun Bagchi, Stephane Detournay, Daniel Grumiller, Max Riegler, and Joan Sim\'on, \emph{Universal sectors of two-dimensional {C}arrollian {CFT}s}, J. High Energy Phys. (2025), no.~11, Paper No. 39, 57 pp.

\bibitem[ABS97]{AshtekarBicakSchmidt}
Abhay Ashtekar, Ji{\v r}{\'i} Bi{\v c}{\'a}k, and Bernd~G. Schmidt, \emph{Asymptotic structure of symmetry-reduced general relativity}, Phys. Rev. D (3) \textbf{55} (1997), no.~2, 669--686.

\bibitem[AFVH25]{Albrychiewicz:2024tqe}
Emil Albrychiewicz, Andr{\'e}s Franco~Valiente, and Vi~Hong, \emph{{Tropical branes}}, J. High Energy Phys. \textbf{06} (2025), 012.

\bibitem[ALZ16]{AdamovicLuZhao}
Dra{\v z}en Adamovi\'c, Rencai L\"u, and Kaiming Zhao, \emph{Whittaker modules for the affine {L}ie algebra {$A_1^{(1)}$}}, Adv. Math. \textbf{289} (2016), 438--479.

\bibitem[AP74]{ArnalPinczon}
D.~Arnal and G.~Pinczon, \emph{On algebraically irreducible representations of the {L}ie algebra {${\rm sl}(2)$}}, J. Mathematical Phys. \textbf{15} (1974), 350--359.

\bibitem[AR16]{AdamovicRadobolja}
Dra{\v{z}}en Adamovi\'{c} and Gordan Radobolja, \emph{On free field realizations of {$W(2,2)$}-modules}, SIGMA Symmetry Integrability Geom. Methods Appl. \textbf{12} (2016), Paper No. 113, 13 pp.

\bibitem[BB25]{BellBuzaglo}
Jason Bell and Lucas Buzaglo, \emph{Maximal dimensional subalgebras of general {C}artan-type {L}ie algebras}, Bull. Lond. Math. Soc. \textbf{57} (2025), no.~2, 605--624.

\bibitem[BBC{\etalchar{+}}20]{BagchiBanerjeeChakraborttyDuttaParekh}
Arjun Bagchi, Aritra Banerjee, Shankhadeep Chakrabortty, Sudipta Dutta, and Pulastya Parekh, \emph{{A tale of three \textemdash{} tensionless strings and vacuum structure}}, J. High Energy Phys. \textbf{04} (2020), 061.

\bibitem[BBM{\etalchar{+}}24]{Bagchi:2024unl}
Arjun Bagchi, Aritra Banerjee, Saikat Mondal, Debangshu Mukherjee, and Hisayoshi Muraki, \emph{{Beyond Wilson? Carroll from current deformations}}, J. High Energy Phys. \textbf{06} (2024), 215.

\bibitem[BC07]{BarnichCompere}
Glenn Barnich and Geoffrey Comp\`ere, \emph{Classical central extension for asymptotic symmetries at null infinity in three spacetime dimensions}, Classical Quantum Gravity \textbf{24} (2007), no.~5, F15--F23.

\bibitem[BCC{\etalchar{+}}25a]{Bagchi:2025jgu}
Arjun Bagchi, Shankhadeep Chakrabortty, Pronoy Chakraborty, Ritankar Chatterjee, and Priyadarshini Pandit, \emph{{Boundary Carroll CFTs: SUSY and Superstrings}}, 8 2025, arXiv:\href{https://arxiv.org/abs/2508.20165}{\texttt{2508.20165}}.

\bibitem[BCC{\etalchar{+}}25b]{BCCFGP2024}
Arjun Bagchi, Pronoy Chakraborty, Shankhadeep Chakrabortty, Stefan Fredenhagen, Daniel Grumiller, and Priyadarshini Pandit, \emph{{Boundary Carrollian Conformal Field Theories and Open Null Strings}}, Phys. Rev. Lett. \textbf{134} (2025), no.~7, 071604.

\bibitem[BCG17]{BatlleCampelloGomis}
Carles Batlle, Victor Campello, and Joaquim Gomis, \emph{{Canonical realization of ( 2+1 )-dimensional Bondi-Metzner-Sachs symmetry}}, Phys. Rev. D \textbf{96} (2017), no.~2, 025004.

\bibitem[BCG20]{BatlleCampelloGomis2}
Carles Batlle, V{\'\i}ctor Campello, and Joaquim Gomis, \emph{{A canonical realization of the Weyl BMS symmetry}}, Phys. Lett. B \textbf{811} (2020), 135920.

\bibitem[BFOGV24]{BatlleFigueroaGomisVishwa}
Carles Batlle, Jos{\'e}~M. Figueroa-O'Farrill, Joaquim Gomis, and Girish~S. Vishwa, \emph{{BMS-like algebras: canonical realisations and BRST quantisation}}, 11 2024, arXiv:\href{https://arxiv.org/abs/2411.14866}{\texttt{2411.14866}}.

\bibitem[BGMM10]{BagchiGopakumarMandalMiwa}
Arjun Bagchi, Rajesh Gopakumar, Ipsita Mandal, and Akitsugu Miwa, \emph{{GCA in 2d}}, J. High Energy Phys. \textbf{08} (2010), 004.

\bibitem[BI26]{BuzagloIngalls}
Lucas Buzaglo and Colin Ingalls, \emph{Lie subalgebras of vector fields on curves}, 2026, arXiv:\href{https://arxiv.org/abs/2606.21748}{\texttt{2606.21748}}.

\bibitem[BJL{\etalchar{+}}16]{BanerjeeJatkarDileepLodatoMukhiNeogi}
Nabamita Banerjee, Dileep~P. Jatkar, Ivano Lodato, Sunil Mukhi, and Turmoli Neogi, \emph{{Extended Supersymmetric BMS$_3$ algebras and Their Free Field Realisations}}, J. High Energy Phys. \textbf{11} (2016), 059.

\bibitem[BJMN16]{BanerjeeJatkarMukhiNeogi}
Nabamita Banerjee, Dileep~P. Jatkar, Sunil Mukhi, and Turmoli Neogi, \emph{{Free-field realisations of the BMS$_{3}$ algebra and its extensions}}, J. High Energy Phys. \textbf{06} (2016), 024.

\bibitem[BJP11]{BaoJiangPei}
Yixin Bao, Cuipo Jiang, and Yufeng Pei, \emph{Representations of affine {N}appi-{W}itten algebras}, J. Algebra \textbf{342} (2011), 111--133.

\bibitem[BKRS21]{BabichenkoKawasetsuRidoutStewart}
Andrei Babichenko, Kazuya Kawasetsu, David Ridout, and William Stewart, \emph{Representations of the {N}appi-{W}itten vertex operator algebra}, Lett. Math. Phys. \textbf{111} (2021), no.~5, Paper No. 131, 30 pp.

\bibitem[BM11]{BatraMazorchuk}
Punita Batra and Volodymyr Mazorchuk, \emph{Blocks and modules for {W}hittaker pairs}, J. Pure Appl. Algebra \textbf{215} (2011), no.~7, 1552--1568.

\bibitem[BO14]{BarnichOblak}
Glenn Barnich and Blagoje Oblak, \emph{{Notes on the BMS group in three dimensions: I. Induced representations}}, J. High Energy Phys. \textbf{06} (2014), 129.

\bibitem[BO15]{BarnichOblak2}
\bysame, \emph{{Notes on the BMS group in three dimensions: II. Coadjoint representation}}, J. High Energy Phys. \textbf{03} (2015), 033.

\bibitem[Bor86]{Borcherds}
Richard~E. Borcherds, \emph{Vertex algebras, {K}ac-{M}oody algebras, and the {M}onster}, Proc. Nat. Acad. Sci. U.S.A. \textbf{83} (1986), no.~10, 3068--3071.

\bibitem[BPZ84]{BPZ1984}
A.~A. Belavin, Alexander~M. Polyakov, and A.~B. Zamolodchikov, \emph{{Infinite Conformal Symmetry in Two-Dimensional Quantum Field Theory}}, Nucl. Phys. B \textbf{241} (1984), 333--380.

\bibitem[BRS74]{Becchi:1974xu}
C.~Becchi, A.~Rouet, and R.~Stora, \emph{{The Abelian Higgs-Kibble Model. Unitarity of the S Operator}}, Phys. Lett. B \textbf{52} (1974), 344--346.

\bibitem[BRS75]{Becchi:1974md}
\bysame, \emph{{Renormalization of the Abelian Higgs-Kibble Model}}, Commun. Math. Phys. \textbf{42} (1975), 127--162.

\bibitem[BRS76]{Becchi:1975nq}
\bysame, \emph{{Renormalization of Gauge Theories}}, Annals Phys. \textbf{98} (1976), 287--321.

\bibitem[Buz23]{Buzaglo2}
Lucas Buzaglo, \emph{Enveloping algebras of {K}richever-{N}ovikov algebras are not {N}oetherian}, Algebr. Represent. Theory \textbf{26} (2023), no.~5, 2085--2111.

\bibitem[Buz24]{Buzaglo}
\bysame, \emph{Derivations, extensions, and rigidity of subalgebras of the {W}itt algebra}, J. Algebra \textbf{647} (2024), 230--276.

\bibitem[BV25]{BuzagloVishwa}
Lucas Buzaglo and Girish~S. Vishwa, \emph{Central extensions, derivations, and automorphisms of semi-direct sums of the {W}itt algebra with its intermediate series modules}, J. Lie Theory \textbf{35} (2025), no.~3, 455--506.

\bibitem[CGL{\etalchar{+}}26]{ChengGaoLiuZhaoZhao}
Cunguang Cheng, Wenting Gao, Shiyuan Liu, Kaiming Zhao, and Yueqiang Zhao, \emph{Quasi-{W}hittaker modules}, J. Algebra \textbf{698} (2026), 288--315.

\bibitem[CGLW24]{ChenGeLiWang}
Hongjia Chen, Lin Ge, Zheng Li, and Longhui Wang, \emph{Classical {W}hittaker modules for the affine {K}ac-{M}oody algebras {$A_N^{( 1 )}$}}, Adv. Math. \textbf{454} (2024), Paper No. 109874, 60.

\bibitem[CGOR16]{CampoleoniGonzalezOblakRiegler}
Andrea Campoleoni, Hernan~A. Gonzalez, Blagoje Oblak, and Max Riegler, \emph{{BMS Modules in Three Dimensions}}, Int. J. Mod. Phys. A \textbf{31} (2016), no.~12, 1650068.

\bibitem[Che21]{Chen}
Chiwei Chen, \emph{Whittaker modules for classical {L}ie superalgebras}, Commun. Math. Phys \textbf{388} (2021), 351--388.

\bibitem[Chr08]{Konstantina}
Konstantina Christodoulopoulou, \emph{Whittaker modules for {H}eisenberg algebras and imaginary {W}hittaker modules for affine {L}ie algebras}, J. Algebra \textbf{320} (2008), no.~7, 2871--2890.

\bibitem[CJ20]{ChenJiang}
Xue Chen and Cuipo Jiang, \emph{Whittaker modules for the twisted affine {N}appi-{W}itten {L}ie algebra {$\widehat H_4[\tau]$}}, J. Algebra \textbf{546} (2020), 37--61.

\bibitem[CSZ24]{Chen:2024voz}
Bin Chen, Haowei Sun, and Yu-fan Zheng, \emph{{Quantization of Carrollian conformal scalar theories}}, Phys. Rev. D \textbf{110} (2024), no.~12, 125010.

\bibitem[DGL24]{DilxatGaoLiu}
Munayim Dilxat, Shoulan Gao, and Dong Liu, \emph{Whittaker modules for the {$N=1$} super-{$\rm BMS_3$} algebra}, J. Algebra Appl. \textbf{23} (2024), no.~5, Paper No. 2450088, 16 pp.

\bibitem[Dix96]{dixmier1996enveloping}
Jacques Dixmier, \emph{Enveloping algebras}, Graduate Studies in Mathematics, vol.~11, American Mathematical Society, Providence, RI, 1996, Revised reprint of the 1977 translation.

\bibitem[FBZ01]{FrenkelBen-Zvi}
Edward Frenkel and David Ben-Zvi, \emph{Vertex algebras and algebraic curves}, Mathematical Surveys and Monographs, vol.~88, American Mathematical Society, Providence, RI, 2001. \MR{1849359}

\bibitem[Fe{\u \i}84]{Feigin}
B.~L. Fe{\u \i}gin, \emph{Semi-infinite homology of {L}ie, {K}ac-{M}oody and {V}irasoro algebras}, Uspekhi Mat. Nauk \textbf{39} (1984), no.~2(236), 195--196.

\bibitem[FGZ86]{FGZ1986}
I.~B. Frenkel, H.~Garland, and G.~J. Zuckerman, \emph{Semi-infinite cohomology and string theory}, Proc. Nat. Acad. Sci. U.S.A. \textbf{83} (1986), no.~22, 8442--8446.

\bibitem[FLM88]{FLM}
Igor Frenkel, James Lepowsky, and Arne Meurman, \emph{Vertex operator algebras and the {M}onster}, Pure and Applied Mathematics, vol. 134, Academic Press, Inc., Boston, MA, 1988.

\bibitem[FMS86]{Friedan:1985ge}
Daniel Friedan, Emil~J. Martinec, and Stephen~H. Shenker, \emph{{Conformal Invariance, Supersymmetry and String Theory}}, Nucl. Phys. B \textbf{271} (1986), 93--165.

\bibitem[FOV25]{Figueroa-OFarrillVishwa}
Jos\'e~M. Figueroa-O'Farrill and Girish~S. Vishwa, \emph{{The BRST quantisation of chiral BMS-like field theories}}, J. Math. Phys. \textbf{66} (2025), no.~4, 042303.

\bibitem[FPSSJ19]{FarahmandParsaSheikh-Jabbari}
A.~Farahmand~Parsa, H.~R. Safari, and M.~M. Sheikh-Jabbari, \emph{{On Rigidity of 3d Asymptotic Symmetry Algebras}}, J. High Energy Phys. \textbf{03} (2019), 143.

\bibitem[FSM85]{Friedan:1985ey}
Daniel Friedan, Stephen~H. Shenker, and Emil~J. Martinec, \emph{{Covariant Quantization of Superstrings}}, Phys. Lett. B \textbf{160} (1985), 55--61.

\bibitem[Fuc86]{Fuchs}
D.~B. Fuchs, \emph{Cohomology of infinite-dimensional {L}ie algebras}, Contemporary Soviet Mathematics, Consultants Bureau, New York, 1986, Translated from the Russian by A. B. Sosinski\u{\i}.

\bibitem[GJP11]{GaoJiangPei}
Shoulan Gao, Cuipo Jiang, and Yufeng Pei, \emph{Low-dimensional cohomology groups of the {L}ie algebras {$W(a,b)$}}, Comm. Algebra \textbf{39} (2011), no.~2, 397--423.

\bibitem[GL11]{GuoLiu}
Xiangqian Guo and Xuewen Liu, \emph{Whittaker modules over {V}irasoro-like algebra}, J. Math. Phys. \textbf{52} (2011), no.~9, Paper No. 093504, 9 pp.

\bibitem[GLP16]{GaoLiuPei}
Shoulan Gao, Dong Liu, and Yufeng Pei, \emph{Structure of the planar {G}alilean conformal algebra}, Rep. Math. Phys. \textbf{78} (2016), no.~1, 107--122.

\bibitem[GPSJ{\etalchar{+}}20]{GrumillerPerezSheikh-JabbariTroncosoZwikel}
Daniel Grumiller, Alfredo P\'erez, M.~M. Sheikh-Jabbari, Ricardo Troncoso, and C\'eline Zwikel, \emph{Spacetime structure near generic horizons and soft hair}, Phys. Rev. Lett. \textbf{124} (2020), no.~4, Paper No. 041601, 7 pp.

\bibitem[HSXZ22]{Hao:2021urq}
Peng-xiang Hao, Wei Song, Xianjin Xie, and Yuan Zhong, \emph{{BMS-invariant free scalar model}}, Phys. Rev. D \textbf{105} (2022), no.~12, 125005.

\bibitem[JLPZ24]{JiangLiuPeiZhao}
Wei Jiang, Dong Liu, Yufeng Pei, and Kaiming Zhao, \emph{Singular vectors, characters, and composition series for the {N}=1 {BMS} superalgebra}, 2024, arXiv:\href{https://arxiv.org/abs/2412.17000}{\texttt{2412.17000}}.

\bibitem[JPZ18]{JiangPeiZhang}
Wei Jiang, Yufeng Pei, and Wei Zhang, \emph{Determinant formula and a realization for the {L}ie algebra {$W(2, 2)$}}, Sci. China Math. \textbf{61} (2018), no.~4, 685--694.

\bibitem[JZ15]{JiangZhang}
Wei Jiang and Wei Zhang, \emph{Verma modules over the {$W(2,2)$} algebras}, J. Geom. Phys. \textbf{98} (2015), 118--127.

\bibitem[Kir04]{Kirillov2004LecturesOT}
A.~A. Kirillov, \emph{Lectures on the orbit method}, Graduate Studies in Mathematics, vol.~64, American Mathematical Society, Providence, RI, 2004.

\bibitem[Kos78]{Kostant}
Bertram Kostant, \emph{On {W}hittaker vectors and representation theory}, Invent. Math. \textbf{48} (1978), no.~2, 101--184.

\bibitem[LPX19]{LiuPeiXia}
Dong Liu, Yufeng Pei, and Limeng Xia, \emph{Whittaker modules for the super-{V}irasoro algebras}, J. Algebra Appl. \textbf{18} (2019), no.~11, 1950211, 13.

\bibitem[LPX22]{DongPeiXia}
\bysame, \emph{A category of restricted modules for the {O}visenko-{R}oger algebra}, Algebr. Represent. Theory \textbf{25} (2022), no.~3, 777--791.

\bibitem[LPXZ24]{LiuPeiXiaZhao}
Dong Liu, Yufeng Pei, Limeng Xia, and Kaiming Zhao, \emph{Smooth modules over the $n = 1$ {B}ondi–{M}etzner–{S}achs superalgebra}, Commun. Contemp. Math. \textbf{electronic} (2024), Paper No. 2450021.

\bibitem[LWZ10]{LiuWuZhu}
Dong Liu, Yuezhu Wu, and Linsheng Zhu, \emph{Whittaker modules for the twisted {H}eisenberg-{V}irasoro algebra}, J. Math. Phys. \textbf{51} (2010), no.~2, Paper No. 023524, 12 pp.

\bibitem[LZ93]{LianZuckerman}
Bong~H. Lian and Gregg~J. Zuckerman, \emph{{New perspectives on the BRST algebraic structure of string theory}}, Commun. Math. Phys. \textbf{154} (1993), 613--646.

\bibitem[LZ10]{LuZhao10}
Rencai L\"u and Kaiming Zhao, \emph{Classification of irreducible weight modules over the twisted {H}eisenberg-{V}irasoro algebra}, Commun. Contemp. Math. \textbf{12} (2010), no.~2, 183--205.

\bibitem[LZ14]{LuZhao14}
\bysame, \emph{Irreducible {V}irasoro modules from irreducible {W}eyl modules}, J. Algebra \textbf{414} (2014), no.~2, 271--287.

\bibitem[LZ16]{LiuZhao}
Genqiang Liu and Yueqiang Zhao, \emph{Irreducible {$A_1^{(1)}$}-modules from modules over two-dimensional non-abelian {L}ie algebra}, Front. Math. China \textbf{11} (2016), no.~2, 353--363.

\bibitem[LZ20]{LuZhao2020}
Rencai L\"u and Kaiming Zhao, \emph{Generalized oscillator representations of the twisted {H}eisenberg-{V}irasoro algebra}, Algebr. Represent. Theory \textbf{23} (2020), no.~4, 1417--1442.

\bibitem[MZ14]{MazorchukZhao}
Volodymyr Mazorchuk and Kaiming Zhao, \emph{Simple {V}irasoro modules which are locally finite over a positive part}, Selecta Math. (N.S.) \textbf{20} (2014), no.~3, 839--854.

\bibitem[OW09]{OndrusWiesner}
Matthew Ondrus and Emilie Wiesner, \emph{Whittaker modules for the {V}irasoro algebra}, J. Algebra Appl. \textbf{8} (2009), no.~3, 363--377.

\bibitem[Pha25]{pham2025orbit}
Tuan~Anh Pham, \emph{The orbit method for the {V}irasoro algebra}, 2025, arXiv:\href{https://arxiv.org/abs/2504.14670}{\texttt{2504.14670}}.

\bibitem[PS23]{PetukhovSierra}
Alexey Petukhov and Susan Sierra, \emph{The {P}oisson spectrum of the symmetric algebra of the {V}irasoro algebra}, Compos. Math. \textbf{159} (2023), no.~5, 933--984.

\bibitem[Rad13]{Radobolja}
Gordan Radobolja, \emph{Subsingular vectors in {V}erma modules, and tensor product of weight modules over the twisted {H}eisenberg-{V}irasoro algebra and {$W(2,2)$} algebra}, J. Math. Phys. \textbf{54} (2013), no.~7, 071701, 24.

\bibitem[RRR22]{RagoucyRasmussenRaymond}
Eric Ragoucy, Jorgen Rasmussen, and Christopher Raymond, \emph{{Asymmetric Galilean conformal algebras}}, Nucl. Phys. B \textbf{981} (2022), 115857.

\bibitem[Rud74]{Rudakov}
A.~N. Rudakov, \emph{Irreducible representations of infinite-dimensional {L}ie algebras of {C}artan type}, Izv. Akad. Nauk SSSR Ser. Mat. \textbf{38} (1974), 835--866.

\bibitem[Rud86]{Rudakov2}
\bysame, \emph{Subalgebras and automorphisms of {L}ie algebras of {C}artan type}, Funktsional. Anal. i Prilozhen. \textbf{20} (1986), no.~1, 83--84.

\bibitem[Sch03]{SchlichenmaierCrelle}
Martin Schlichenmaier, \emph{Local cocycles and central extensions for multipoint algebras of {K}richever-{N}ovikov type}, J. Reine Angew. Math. \textbf{559} (2003), 53--94.

\bibitem[Sch17]{Schlichenmaier}
\bysame, \emph{{$N$}-point {V}irasoro algebras are multipoint {K}richever-{N}ovikov-type algebras}, Comm. Algebra \textbf{45} (2017), no.~2, 776--821.

\bibitem[Tyu75]{Tyutin:1975qk}
I.~V. Tyutin, \emph{{Gauge Invariance in Field Theory and Statistical Physics in Operator Formalism}}, Preprint of P.N. Lebedev Physical Institute (1975), no.~39, 22 pp.

\bibitem[TZ13]{TanZhao}
Haijun Tan and Kaiming Zhao, \emph{Irreducible {V}irasoro modules from tensor products ({II})}, J. Algebra \textbf{394} (2013), 357--373.

\bibitem[Ush98]{Ushirobira1998}
Rosane Ushirobira, \emph{On the orbit method for the {L}ie algebra of vector fields on a curve}, J. Algebra \textbf{203} (1998), no.~2, 596--620.

\bibitem[Wan11]{Bin}
Bin Wang, \emph{Whittaker modules for graded {L}ie algebras}, Algebr. Represent. Theory \textbf{14} (2011), no.~4, 691--702.

\bibitem[Wei94]{Weibel}
Charles~A. Weibel, \emph{An introduction to homological algebra}, Cambridge Studies in Advanced Mathematics, vol.~38, Cambridge University Press, Cambridge, 1994.

\bibitem[ZD09]{ZhangDong}
Wei Zhang and Chongying Dong, \emph{{$W$}-algebra {$W(2,2)$} and the vertex operator algebra {$L(\frac 12,0)\otimes L(\frac 12,0)$}}, Comm. Math. Phys. \textbf{285} (2009), no.~3, 991--1004.

\bibitem[ZTL10]{ZhangTanLian}
Xiufu Zhang, Shaobin Tan, and Haifeng Lian, \emph{Whittaker modules for the {S}chr\"odinger-{W}itt algebra}, J. Math. Phys. \textbf{51} (2010), no.~8, 083524, 17.

\end{thebibliography}
\end{document}